\documentclass{article}
\usepackage{amssymb}
\usepackage{amsfonts}
\usepackage{amsmath}

\newtheorem{theorem}{Theorem} [section]

\newtheorem{corollary}[theorem]{Corollary}

\newtheorem{example}[theorem]{Example}

\newtheorem{lemma}[theorem]{Lemma}

\newtheorem{remark}[theorem]{Remark}

\numberwithin{equation}{section}
\newenvironment{proof}[1][Proof]{\noindent\textbf{#1.} }{\ \rule{0.5em}{0.5em}}
\input{tcilatex}
\begin{document}

\title{On elliptic factors in real endoscopic transfer I}
\author{D. Shelstad}
\maketitle

\begin{abstract}
This paper is concerned with the structure of packets of representations and
some refinements that are helpful in endoscopic transfer for real groups. It
includes results on the structure and transfer of packets of limits of
discrete series representations. It also reinterprets the Adams-Johnson
transfer of certain nontempered representations via spectral analogues of
the Langlands-Shelstad factors, thereby providing structure and transfer
compatible with the associated transfer of orbital integrals. The results
come from two simple tools introduced here. The first concerns a family of
splittings of the algebraic group $G$ under consideration; such a splitting
is based on a fundamental maximal torus of $G$ rather than a maximally split
maximal torus. The second concerns a family of Levi groups attached to the
dual data of a Langlands or an Arthur parameter for the group $G$. The
introduced splittings provide explicit realizations of these Levi groups.
The tools also apply to maps on stable conjugacy classes associated with the
transfer of orbital integrals. In particular, they allow for a simpler
version of the definitions of Kottwitz-Shelstad for twisted endoscopic
transfer in certain critical cases. The paper prepares for spectral factors
in twisted endoscopic transfer that are compatible in a certain sense with
the standard factors discussed here. This compatibility is needed for
Arthur's global theory. The twisted factors themselves will be defined in a
separate paper.
\end{abstract}

\tableofcontents

\section{\textbf{Introduction}}

Our main purpose is to continue a study of the coefficients appearing in the
spectral identities of endoscopic transfer for real groups. The coefficients
carry information about the structure of packets of irreducible
representations, and in the global theory of endoscopy this structure plays
a central role in determining if certain irreducible representations are
automorphic or not; see \cite{Ar13}.

Here we will consider both the standard and the more general twisted
versions of endoscopic transfer. We focus on the \textit{fundamental case}
where the endoscopic group and the ambient group share, in a certain precise
sense,\textit{\ }fundamental maximal tori; see Section 3.3. It includes the
case where the ambient group $G$ is cuspidal and the endoscopic group $H_{1}$
is elliptic. We call this the cuspidal-elliptic setting; see Section 3.4.
Then $G(\mathbb{R})$ and $H_{1}(\mathbb{R})$ share fundamental Cartan
subgroups that are elliptic, \textit{i.e.}, compact modulo the centers of $G(%
\mathbb{R})$ and $H_{1}(\mathbb{R})$ respectively. Thus there is a discrete
series of representations for each of $G(\mathbb{R})$ and $H_{1}(\mathbb{R})$
\cite{HC75}, along with limits of discrete series representations (see \cite%
{KZ82}).

Endoscopic transfer begins with the matching of orbital integrals, the
so-called geometric side. In the standard version we use the transfer
factors of Langlands-Shelstad (\cite{LS87}, see also \cite{Sh14}) for the
geometric side. Factors with a parallel definition appear in the tempered
dual spectral transfer, \textit{i.e.}, as coefficients in the dual spectral
identities for tempered irreducible representations \cite{Sh10, Sh08b}.
Properties of these spectral factors simplify the related harmonic analysis;
for example, inversion of the identities becomes a short exercise (see \cite%
{Sh08b}).

In preparation for generalizing (in \cite{ShII}) the definition of spectral
factors to the twisted setting of Kottwitz-Shelstad \cite{KS99} we will
establish three refinements. First, we make use of an alternative simpler
description of limits of discrete series packets in terms of elliptic data, 
\textit{i.e.}, data attached to an elliptic Cartan subgroup (see Remark
5.6), to simplify transfer and structure in that setting.

Second, introducing the nontempered spectrum to our picture, we reinterpret
the transfer of Adams-Johnson in terms of data attached directly to the
associated Arthur parameters. Here we will consider only parameters that are
elliptic in the sense of Arthur. Our new factors are related very simply to
the tempered factors already defined, and we check that they do provide the
transfer that is precisely dual (\textit{i.e.}, there are no extraneous
constants) to that of orbital integrals with the Langlands-Shelstad factors.
The inversion properties of our spectral transfer are more delicate than for
the tempered case (\cite{AJ87}, \cite{Ar89} explain why this must be the
case) and will be described in \cite{ShII}.

For the third refinement we turn to twisted transfer and the underlying
definitions of \cite{KS99}. The transfer of orbital integrals is based on an
abstract norm correspondence $(\gamma _{1},\delta )$ for suitably regular
points $\gamma _{1}$ in endoscopic $H_{1}(\mathbb{R})$ and $\delta $ in $G(%
\mathbb{R}).$ For the fundamental part of the correspondence that concerns
us here we will see that we may limit the twisting automorphism to a family
for which the norm correspondence is well-behaved. The standard spectral
factors generalize readily for this family \cite{ShII} and we have the
standard-twisted compatibility needed in Arthur's global theory \cite{Ar13}.

To obtain these refinements we introduce two simple tools. The first
involves fundamental splittings. These are particular splittings based on
fundamental maximal tori and exist for any $G$. They work well with both
elliptic data and Whittaker data; here Vogan's characterization of generic
representations plays a critical role. See Sections 2.3, 6.1. The second
tool involves a family of Levi groups in $G$. First we attach to a Langlands
or Arthur parameter a family of $L$-groups and then we use fundamental
splittings to identify their real duals as a family of (nonstandard) Levi
groups in $G$ and its inner forms; see Sections 5.2, 6.1.

We begin the paper with fundamental splittings and their properties. The
main result is Lemma 2.5. In Part 3 we review the norm correspondence and
see, in particular, that the fundamental part of geometric transfer is
nonempty if and only if the twisting automorphism preserves a fundamental
splitting up to a further twist by an element of $G(\mathbb{R})$. This may
be expressed precisely in terms of the norm correspondence itself or in
terms of nonvanishing of the geometric transfer factors of \cite{KS99}; see
Theorem 3.12, Corollary 3.13. We prove a spectral analogue, but not until
Part 9 where we also finish the discussion of Part 4 on certain properties
of endoscopic data that will be used in the definition of twisted spectral
factors.

The rest of the paper concerns standard transfer and the first two
refinements. In Part 5 we turn to the Langlands and Arthur parameters
attached to the representations of interest to us here. In the
cuspidal-elliptic setting these are the $s$-elliptic Langlands parameters
and the elliptic $u$-regular Arthur parameters of Sections 5.5 - 5.7. Given
a parameter, we generate data for the various attached packets of
representations by means of pairs $(G,\eta ),$ where $\eta $ is an inner
twist of $G$ to a given quasi-split form $G^{\ast }$.

For the limits of discrete series representations attached to an $s$%
-elliptic Langlands parameter we reformulate some well-known properties in
terms of our attached Levi groups. For example, the critical Lemma 6.1
characterizes the pairs $(G,\eta )$ for which we obtain a well-defined (%
\textit{i.e.,} nonzero) representation. Then Lemma 6.2 gives a description
of the packet that allows us to attach an elliptic invariant to each member;
see Section 6.4. Lemmas 6.4 and 6.5\ describe the application to endoscopic
transfer.

Part 6 has further results on limits of discrete series representations that
we will apply in various places. For example, as in Section 6.7, every $s$%
-elliptic parameter factors through a totally degenerate parameter for an
attached Levi group. We will check in \cite{ShII} that this gives a simple
characterization of those pairs $(G,\eta )$ for which the distribution
character of the attached representation is elliptic.

The representations attached to elliptic $u$-regular Arthur parameters are
the derived functor modules of Vogan and Zuckerman \cite{Vo84} from the main
setting in \cite{AJ87}; we allow without harm a nontrivial split component
in the center of $G$. In the case of regular infinitesimal character they
are discussed in \cite{Ar89, Ko90}; see the last paragraph of Section 7.1.
Generalizing some familiar $L$-group constructions we attach directly to the
Arthur parameter the following: elliptic data, a family of Levi groups, and
an $s$-elliptic Langlands parameter with same infinitesimal character and
central behavior. We will use these again in \cite{ShII} in the twisted
setting. In the present paper we pursue only the case of regular
infinitesimal character so that the attached Langlands parameter is
elliptic. Lemma 7.5 describes the Arthur packet by means of pairs $(G,\eta ).
$ In Section 8 we introduce spectral transfer factors for the Arthur packet
by tethering the packet to the elliptic Langlands packet via \textit{relative%
} factors with good transitivity properties \cite{Sh10}. We then verify in
Section 8.3 that the corresponding absolute factors are correct, in the
sense already mentioned, for endoscopic transfer with Langlands-Shelstad
factors on the geometric side.

To finish this brief sketch we refer to Sections 3.5, 4.1, 4.2, 6.4 and 6.5
where there are further remarks on the properties of transfer factors that
are crucial for our approach to work.

\textbf{Note:} This paper is an expanded version of part of the preprint "On
spectral transfer factors in real twisted endoscopy" posted on the author's
website, May 2011.

\section{\textbf{Automorphisms and inner forms}}

This section introduces notation we will use throughout the paper, along
with definitions and properties related to fundamental splittings. We finish
with an application to the \textit{inner forms of a quasi-split pair}.

\subsection{Quasi-split pairs and inner forms}

By a quasi-split pair we mean a pair $(G^{\ast },\theta ^{\ast }),$ where $%
G^{\ast }$ is a connected, reductive algebraic group defined and quasi-split
over $\mathbb{R}$, and $\theta ^{\ast }$ is an $\mathbb{R}$-automorphism of $%
G^{\ast }$ that preserves an $\mathbb{R}$-splitting $spl^{\ast }=(B^{\ast
},T^{\ast },\{X_{\alpha }\})$ of $G^{\ast }$. We assume that the restriction
of $\theta ^{\ast }$ to the identity component of the center of $G^{\ast }$
is semisimple or, equivalently, that $\theta ^{\ast }$ has finite order.

Recall from \cite[Appendix B]{KS99}\ that $(G,\theta ,\eta )$ is defined to
be an inner form of $(G^{\ast },\theta ^{\ast })$ if $G$ is connected,
reductive and defined over $\mathbb{R}$, $\eta $ is an isomorphism from $G$
to $G^{\ast }$ that is an inner twist, $\theta $ is an $\mathbb{R}$%
-automorphism of $G$, and $\theta $ coincides with the transport of $\theta
^{\ast }$ to $G$ via $\eta $ up to an inner automorphism. Notice that if $%
\theta ^{\ast }$ is the identity then $\theta $ must be an inner
automorphism of $G$ defined over $\mathbb{R}$, \textit{i.e.,} $\theta $ must
act on $G$ as an element of $G_{ad}(\mathbb{R}).$

Let $(G,\theta ,\eta )$ be an inner form of $(G^{\ast },\theta ^{\ast })$.
By the \textit{inner class} of $(\theta ,\eta )$ we will mean the set of all
pairs $(\theta ^{\prime },\eta ^{\prime })$ where $(G,\theta ^{\prime },\eta
^{\prime })$ is an inner form of $(G^{\ast },\theta ^{\ast })$ such that (i)%
\textit{\ }$\eta ^{\prime }\circ \eta ^{-1}$ is inner and (ii) the
automorphism $\theta ^{\prime }\circ \theta ^{-1}$ of $G,$ which is inner
and acts as an element of $G_{ad}(\mathbb{R})$ by (i), is induced by an
element of $G(\mathbb{R})$, \textit{i.e.,} acts as an element of the image
of $G(\mathbb{R})$ in $G_{ad}(\mathbb{R})$ under the natural projection. We
will see that replacing $(\theta ,\eta )$ by a member of its inner class has
no effect on our final results.

Let $(G,\theta ,\eta )$ be an inner form of $(G^{\ast },\theta ^{\ast })$.
Then we choose $u(\sigma )\in G_{sc}^{\ast }$ such that%
\begin{equation}
\eta \circ \sigma (\eta )^{-1}=Int(u(\sigma )).
\end{equation}%
Here and throughout the paper we use $\sigma $ to denote the nontrivial
element of $\Gamma =Gal(\mathbb{C}/\mathbb{R}).$ The action of $\sigma $ on
a $\Gamma $-set $X$ will be denoted by $\sigma _{X}$ or by $\sigma $ itself
when $X$ is evident.

\subsection{Fundamental splittings}

While $\mathbb{R}$-splittings exist only for quasi-split groups, \textit{%
fundamental splittings} may be constructed for any connected, reductive $G$
defined over $\mathbb{R}$. We recall the definition (see \cite{Sh12}).

Consider a pair $(B,T)$, where $T$ is a maximal torus in $G$ defined over $%
\mathbb{R}$ and $B$ is a Borel subgroup of $G$ containing $T$. We call $%
(B,T) $ a \textit{fundamental pair} if (i) $T$ is fundamental, \textit{i.e.,}
$T$ is minimally $\mathbb{R}$-split or, equivalently, $T$ has no roots fixed
by $\sigma _{T}$ and (ii) the set of (simple) roots of $T$ in $B$ is
preserved by $-\sigma _{T}.$ The existence of fundamental pairs is noted in 
\cite[Section 10.4]{Ko86}.

\begin{lemma}
The set of all fundamental pairs for $G$ forms a single stable conjugacy
class in the sense that another fundamental pair $(B^{\prime },T^{\prime })$
is conjugate to $(B,T)$ by an element $g$ of $G$ for which $%
Int(g):T\rightarrow T^{\prime }$ is defined over $\mathbb{R}.$
\end{lemma}

\begin{proof}
Observe that $(B^{\prime },T^{\prime })$ is conjugate to $(B,T)$ under some
element $g$ of $G,$ and then $g^{-1}\sigma (g)$ acts as an element of the
Weyl group of $T$ preserving the roots of $B,$ \textit{i.e.,} as the
identity element.
\end{proof}

To prescribe a fundamental splitting we start with a fundamental pair $(B,T)$
and pick an $\mathfrak{sl}_{2}$-triple $\{X_{\alpha },H_{\alpha },X_{-\alpha
}\}$ for each simple root $\alpha $ of $T$ in $B.$ Here we identify the Lie
algebra of $T$ with $X_{\ast }(T)\otimes \mathbb{C}$ and require $H_{\alpha
} $ be the element identified with the coroot $\alpha ^{\vee }$ of $\alpha ;$
$X_{\alpha },X_{-\alpha }$ are to be root vectors for $\alpha ,-\alpha $
respectively. There is an attached splitting $spl=(B,T,\{X_{\alpha }\})$ for 
$G.$ Conversely, each splitting for $G$ determines uniquely a collection of $%
\mathfrak{sl}_{2}$-triples of the above form. We call $spl$ \textit{%
fundamental} if the Galois action satisfies: $\sigma X_{\alpha }=X_{\sigma
_{T}\alpha }$ in the case $\sigma _{T}\alpha \neq -\alpha ,$ and $\sigma
X_{\alpha }=\varepsilon _{\alpha }X_{-\alpha }\ $in the case $\sigma
_{T}\alpha =-\alpha $, where $\varepsilon _{\alpha }=\pm 1$.

If $\sigma _{T}\alpha =-\alpha $ then such a triple $\{H_{\alpha },X_{\alpha
},X_{-\alpha }\}$ determines an $\mathbb{R}$-homomor-phism from a real form
of $SL(2)$ into $G;$ examples are written in \cite{Sh79a}. The isomorphism
class of that real form, split or anisotropic, is uniquely determined by $%
\alpha .$ If the real form is split then $\varepsilon _{\alpha }=1$ and $%
\alpha $ is called \textit{noncompact}. If the real form is anisotropic then 
$\varepsilon _{\alpha }=-1$ and $\alpha $ is \textit{compact}.

\begin{lemma}
Each fundamental pair $(B,T)$ extends to a fundamental splitting $%
spl=(B,T,\{X_{\alpha }\})$. Moreover, two fundamental splittings extending $%
(B,T)$ are conjugate under $T_{sc}(\mathbb{R})$.
\end{lemma}

\begin{proof}
For each simple root $\alpha $ of $T$ in $B$, pick an $\mathfrak{sl}_{2}$%
-triple $\{X_{\alpha },H_{\alpha },X_{-\alpha }\}$. If $\sigma _{T}\alpha
\neq -\alpha $ then we may arrange that $\sigma X_{\alpha }=X_{\sigma
_{T}\alpha }$ and $\sigma X_{-\alpha }=X_{-\sigma _{T}\alpha }$ since $%
\alpha ,\sigma _{T}\alpha $ are distinct. If $\sigma _{T}\alpha =-\alpha $
then calculation shows that $\sigma X_{\alpha }=\lambda X_{-\alpha }\ $and $%
\sigma X_{-\alpha }=\lambda ^{-1}X_{\alpha }$, where $\lambda $ is real.
Then we can adjust the choice of $X_{\alpha }$ and $X_{-\alpha }$ to arrange
that $\sigma X_{\alpha }=\varepsilon _{\alpha }X_{-\alpha }\ $and $\sigma
X_{-\alpha }=\varepsilon _{\alpha }X_{\alpha }$, where $\varepsilon _{\alpha
}=\pm 1.$

Suppose we have two such splittings. Then they are conjugate under $T_{ad}(%
\mathbb{R})$ \ since if $X_{\alpha }$ is replaced by $Int(t)X_{\alpha }$ for
each $B$-simple root $\alpha ,$ where $t\in T_{sc}$, then our requirements
on the action of $\sigma $ imply that $\alpha (\sigma (t)t^{-1})=1.$ Because 
$T_{sc},T_{ad}$ are fundamental, $T_{sc}(\mathbb{R})$ and $T_{ad}(\mathbb{R}%
) $ are connected; see \cite[Section 10]{Ko86} and \cite[Section 6]{Sh12}.
Then the projection $T_{sc}(\mathbb{R})\rightarrow T_{ad}(\mathbb{R})$ is
surjective, and the desired conjugation exists.
\end{proof}

\begin{corollary}
Each fundamental splitting $spl^{\prime }=(B^{\prime },T^{\prime
},\{X_{\alpha }^{\prime }\})$ is conjugate to $spl$ by an element $g$ of $G$
for which $Int(g):T\rightarrow T^{\prime }$ is defined over $\mathbb{R}$.
\end{corollary}

\subsection{Fundamental splittings of Whittaker type}

We return to a quasi-split pair $(G^{\ast },\theta ^{\ast }).$ Recall that $%
\theta ^{\ast }$ preserves the $\mathbb{R}$-splitting $spl^{\ast }=(B^{\ast
},T^{\ast },\{X_{\alpha }\})$ of $G^{\ast }$. \textit{From now on we will
typically use the same notation }$\{X_{\alpha }\}$\textit{\ for the root
vectors in any splitting.} We will say a fundamental pair\textit{\ }$(B,T),$%
\textit{\ }or\textit{\ }a fundamental splitting\textit{\ }$%
spl_{f}=(B,T,\{X_{\alpha }\})$\textit{\ of }$G^{\ast },$\textit{\ }is\textit{%
\ of Whittaker type} if all imaginary simple roots of $(B,T)$ are
noncompact. We use this terminology because of Vogan's classification
theorem \cite[Corollary 5.8, Theorem 6.2]{Vo78} for representations with
Whittaker model, \textit{i.e.}, for \textit{generic} representations. It is
not difficult to check directly that a group $G$ has a fundamental pair of
Whittaker type if and only if $G$ is quasi-split over $\mathbb{R}$, although
this characterization is naturally part of Vogan's classification.

\begin{lemma}
(i) There exists a fundamental pair of Whittaker type preserved by $\theta
^{\ast }$ and (ii) each fundamental pair of Whittaker type preserved by $%
\theta ^{\ast }$ has an extension to a fundamental splitting $spl_{Wh}$ of $%
G^{\ast }$ preserved by $\theta ^{\ast }$.
\end{lemma}

\begin{proof}
(i) We use Steinberg's structure theorems as described in \cite[Section 3]%
{KS12} and \cite[Section 1.3]{KS99}. First attach to $spl^{\ast }$ an $%
\mathbb{R}$-splitting for $(G_{sc}^{\ast })^{\theta _{sc}^{\ast }}.$ We may
then find $h$ in $(G_{sc}^{\ast })^{\theta _{sc}^{\ast }}$ conjugating the
pair determined by this $\mathbb{R}$-splitting to a fundamental pair in $%
(G_{sc}^{\ast })^{\theta _{sc}^{\ast }}$ of Whittaker type; such a pair
exists since $(G_{sc}^{\ast })^{\theta _{sc}^{\ast }}$ is quasi-split. This
pair determines uniquely a pair $(B,T)$ for $G^{\ast \text{ }}$preserved by $%
\theta ^{\ast }.$ Then $(B,T)$ is fundamental because $T$ can have no real
roots; see the proof of Lemma 3.4.1 below. An examination of root vectors
shows further that $(B,T)$ is of Whittaker type.

(ii) Now attach to any fundamental $(B,T)$ of Whittaker type a fundamental
pair in $(G_{sc}^{\ast })^{\theta _{sc}^{\ast }}$ also of Whittaker type,
and define $h$ in $(G_{sc}^{\ast })^{\theta _{sc}^{\ast }}$ as in (i).
Extend $(B,T)\ $to a fundamental splitting $spl_{f}=(B,T,\{X_{\alpha }\})$
for $G^{\ast }$. There is $t\in T_{sc}^{\ast }$ such that $th$ transports $%
spl^{\ast }$ to $spl_{f}.$ Then 
\begin{equation}
\theta _{f}=Int(th)\circ \theta ^{\ast }\circ Int(th)^{-1}=Int(t\theta
_{sc}^{\ast }(t)^{-1})\circ \theta ^{\ast }
\end{equation}%
preserves $spl_{f}$ and coincides with $\theta ^{\ast }$ on $T.$ A
calculation on root vectors shows that $\sigma (\theta _{f})=\theta _{f}$.
For this, note that the Whittaker property of $(B,T)$ implies that $\sigma
X_{\alpha }=X_{-\alpha },$ for each imaginary root vector $X_{\alpha }$ in $%
spl_{f}$. Thus $\theta _{f}$ is defined over $\mathbb{R}$. Then $Int(t\theta
_{sc}^{\ast }(t)^{-1})$ lies in $T_{ad}(\mathbb{R}).$ Since $T_{sc}(\mathbb{R%
})\rightarrow T_{ad}(\mathbb{R})$ is surjective, we may take $t\theta
_{sc}^{\ast }(t)^{-1}$ in $V(\mathbb{R})=T_{sc}(\mathbb{R})\cap V$, where $%
V=[1-\theta _{sc}^{\ast }](T_{sc})$. Now we claim that for fundamental $T,$
the kernel of $H^{1}(\Gamma ,(T_{sc})^{\theta _{sc}^{\ast }})\rightarrow
H^{1}(\Gamma ,T_{sc})$ is trivial. From the Tate-Nakayama isomorphisms it is
enough to show the kernel of $H^{-1}(\Gamma ,[X_{\ast }(T_{sc})]^{\theta
_{sc}^{\ast }})\rightarrow H^{-1}(\Gamma ,X_{\ast }(T_{sc}))$ is trivial.
This is immediate since both $-\sigma _{T}$ and $\theta _{sc}^{\ast }$
preserve a base for the coroot lattice $X_{\ast }(T_{sc})$. Triviality of
the kernel implies that $V(\mathbb{R})$ is connected. Thus we may assume $%
t\in T_{sc}(\mathbb{R}).$ Then $\theta ^{\ast }=Int(t^{-1})\circ \theta
_{f}\circ Int(t)$ preserves the splitting $Int(t^{-1})(spl_{f})$ which is
fundamental and of Whittaker type.
\end{proof}

\subsection{An application}

We continue with an inner form $(G,\theta ,\eta )$ of the quasi-split pair $%
(G^{\ast },\theta ^{\ast }).$ Following \cite[Chapter 3]{KS99} we say an
element $\delta $ of $G(\mathbb{R})$ is $\theta $-semisimple\textit{\ }if $%
Int(\delta )\circ \theta $ preserves a pair $(B,T).$ We will say that the $%
\theta $-semisimple element $\delta $ of $G(\mathbb{R})$ is $\theta $\textit{%
-fundamental }if $Int(\delta )\circ \theta $ preserves a fundamental pair $%
(B,T).$

Recall that $G$ is \textit{cuspidal} if and only if a fundamental maximal
torus $T$ is \textit{elliptic}, \textit{i.e.,} $T$ is anisotropic modulo the
center $Z_{G}$ of $G.$ In a setting where $G$ is assumed cuspidal we will
use the term $\theta $\textit{-elliptic} interchangeably with $\theta $%
\textit{-fundamental}. For strongly $\theta $\textit{-}regular $\theta $%
\textit{-}semisimple elements there is another definition of $\theta $%
\textit{-}ellipticity (which does not require $G$ to be cuspidal) in \cite[%
Introduction]{KS99}. We observe that a strongly $\theta $\textit{-}regular $%
\theta $\textit{-}semisimple element $\delta $ of cuspidal $G(\mathbb{R})$
is $\theta $\textit{-}elliptic in our present sense if and only if it is $%
\theta $-elliptic in the sense of \cite{KS99}; see Lemma 3.8(i). \textit{In
the general setting we will use exclusively the term }$\theta $\textit{%
-fundamental}. The strongly $\theta $\textit{-}regular $\theta $\textit{-}%
semisimple elements of $G(\mathbb{R})$ that are $\theta $\textit{-}elliptic
in the sense of \cite{KS99} are $\theta $-fundamental; this is another
consequence of the observation about real roots in the proof of Lemma 3.8(i).

Following Lemma 2.4 we choose a fundamental splitting $spl_{Wh}$ of $G^{\ast
}$ of Whittaker type preserved by $\theta ^{\ast }$.

\begin{lemma}
(i) There exists a $\theta $-fundamental element in $G(\mathbb{R})\ $if and
only if there is $(\theta _{f},\eta _{f})$ in the inner class of $(\theta
,\eta )$ such that $\theta _{f}$ preserves a fundamental splitting for $G$.

(ii) If such $(\theta _{f},\eta _{f})$ exists and $\theta _{f}$ preserves
the fundamental splitting $spl_{G}$ then we may further assume $\eta _{f}$
transports $spl_{G}$ to $spl_{Wh}$ and $\theta _{f}$ to $\theta ^{\ast }$.
\end{lemma}

\begin{proof}
Assume that there exists a $\theta $-fundamental element in $G(\mathbb{R}).$
Then we may multiply $\theta $ by an element of $Int(G(\mathbb{R}))$ to
obtain an $\mathbb{R}$-automorphism $\theta ^{\prime }$ preserving a
fundamental pair. Now apply Lemma 2.2.2 to extend this pair to a fundamental
splitting $spl_{G}.$ Since $\theta ^{\prime }$ carries $spl_{G}$ to another
fundamental splitting, the lemma also shows that a further multiplication by
an element of $Int(G(\mathbb{R}))$ provides an $\mathbb{R}$-automorphism $%
\theta _{f}$ which preserves $spl_{G}.$ We choose $\eta _{f}:G\rightarrow
G^{\ast }$ in the inner class of $\eta $ carrying $spl_{G}$ to $spl_{Wh}.$
Then $\eta _{f}\circ \theta _{f}\circ \eta _{f}^{-1}$ and $\theta ^{\ast }$
are automorphisms of $G^{\ast }$ which preserve $spl_{Wh}$ and differ by an
inner automorphism. Hence they coincide. The converse assertion in (i) is
immediate, and so the lemma is proved.
\end{proof}

For $\theta _{f}$ as in (ii) of the lemma, write $\eta _{f}\circ \sigma
(\eta _{f})^{-1}$ as $Int(u_{f}(\sigma )).$ Then, applying $\sigma $ to the
equation%
\begin{equation}
\eta _{f}\circ \theta _{f}\circ \eta _{f}^{-1}=\theta ^{\ast },
\end{equation}%
we see that $Int(u_{f}(\sigma ))$ lies in the torus $(T_{ad})^{\theta
_{ad}^{\ast }}.$ Since $(T_{sc})^{\theta _{sc}^{\ast }}\rightarrow
(T_{ad})^{\theta _{ad}^{\ast }}$ is surjective (both are connected; see \cite%
[Section 1.1]{KS99}) we may now assume 
\begin{equation}
u_{f}(\sigma )\in (T_{sc})^{\theta _{sc}^{\ast }}.
\end{equation}

\section{\textbf{Norms and the fundamental case}}

Here we include notation and review, and show that the norm correspondence
is well-behaved in the fundamental case.

\subsection{Endoscopic data}

We now consider as quasi-split data, a triple $(G^{\ast },\theta ^{\ast
},a), $ where $(G^{\ast },\theta ^{\ast })$ is a quasi-split pair as above,
and $a$ is a $1$-cocycle of the Weil group $W_{\mathbb{R}}$ of $\mathbb{C}/%
\mathbb{R} $ in the center of the connected Langlands dual group $G^{\vee }.$
Then $\varpi $ will denote the character on $G^{\ast }(\mathbb{R}),$ or on
the real points of an inner form of $G^{\ast }$, attached to $a.$ As always,
and without harm, we provide an explicit transition of data between $G^{\ast
}$ and its Langlands dual $^{L}G=G^{\vee }\rtimes W_{\mathbb{R}}$ by the
choice of $\mathbb{R}$-splitting $spl^{\ast }=(B^{\ast },T^{\ast
},\{X_{\alpha }\})$ of $G^{\ast }$ preserved by $\theta ^{\ast }$ and dual $%
\Gamma $-splitting $spl^{\vee }=(\mathcal{B},\mathcal{T},\{X_{\alpha ^{\vee
}}\})$ for $G^{\vee }.$ The action of $W_{\mathbb{R}}$ on $G^{\vee }$
factors through $W_{\mathbb{R}}\rightarrow \Gamma .$ Then $\theta ^{\vee }$
is the $\Gamma $-automorphism of $G^{\vee }$ that preserves $spl^{\vee }$
and is dual to $\theta ^{\ast }$ as automorphism of the dual based root
data. We write $^{L}\theta _{a}$ for the extension%
\begin{equation*}
g\times w\rightarrow a(w).\theta ^{\vee }(g)\times w
\end{equation*}%
of $\theta ^{\vee }$ to an automorphism of $^{L}G.$

We assume $\mathfrak{e}_{z}$ is a supplemented set of endoscopic data (SED)
for $(G^{\ast },\theta ^{\ast },a)$ and its inner forms. The SED consists of
a set $\mathfrak{e}=(H,\mathcal{H},s)$ of endoscopic data for $(G^{\ast
},\theta ^{\ast },a)$ and a $z$-pair $(H_{1},\xi _{1})$ for $\mathfrak{e}$
in the sense of \cite{KS99}, although we avoid the additional choice $%
a^{\prime }$ from Section 2.1 of \cite{KS99} by adjusting $s.$ There is no
harm in assuming $\mathfrak{e}_{z}$ is bounded in the sense of \cite[Section
2]{Sh14}. Recall that $H_{1}$ is what we call the endoscopic group defined
by the SED; $Z_{1}$ will denote the kernel of the $z$-extension $%
H_{1}\rightarrow H.$ We remark that SEDs exist for $(G^{\ast },\theta ^{\ast
},a)$ precisely when there are Langlands parameters preserved by $^{L}\theta
_{a};$ see Section 9.1.

As noted in the introduction, we will be concerned mainly with the \textit{%
fundamental case} for which we make an ad hoc definition in Section 3.3, and
more particularly with the \textit{cuspidal-elliptic setting} of Section 3.4.

\subsection{Norm correspondence}

A norm correspondence for $G(\mathbb{R})$ and an endoscopic group $H_{1}(%
\mathbb{R})$ is defined via maps on (twisted) conjugacy classes \cite[%
Chapter 3, Section 5.4]{KS99}. In general, the correspondence is not
uniquely determined by $(\theta ,\eta )$ and there are examples where it is
empty on all or much of the \textit{very regular set} defined in the
paragraph following (3.1) below. In preparation for the fundamental case to
be introduced in Section 3.3, we review two simpler settings indicated as I,
II.

(I) Assume that $\theta $ preserves a fundamental splitting or, more
precisely, that $(\theta ,\eta )$ is of the form $(\theta _{f},\eta _{f})$
from (ii) in Lemma 2.5. The equation (2.3) allows us to attach a unique norm
correspondence to $(\theta ,\eta )$. To begin, there is no need to choose
the datum $g_{\theta }$ of \cite[Chapter 3]{KS99}; in the formulas there,
set $g_{\theta }=1.$ To compute the cochain $z_{\sigma }$ of Lemma 3.1.A of 
\cite{KS99}, write $u(\sigma )$ from (2.1) above as $u_{1}(\sigma ).z(\sigma
)$, where $u_{1}(\sigma )\in (T_{sc})^{\theta _{sc}^{\ast }}$ as in (2.4)
and $z(\sigma )$ is central in $G_{sc}^{\ast }$. Thus $z_{\sigma }=(1-\theta
_{sc}^{\ast })$ $z(\sigma ).$ Then, by (2) of Lemma 3.1.A in \cite{KS99}, $%
\eta $ determines uniquely a $\Gamma $-equivariant bijective map from the
set $Cl_{\theta \text{-}ss}(G,\theta )$ of $\theta $-twisted conjugacy
classes of $\theta $-semisimple elements in $G(\mathbb{C})$ to the
corresponding set $Cl_{\theta ^{\ast }\text{-}ss}(G^{\ast },\theta ^{\ast })$
for $(G^{\ast },\theta ^{\ast }).$ This map provides the first step in
defining the norm correspondence. By restriction, we obtain a $\Gamma $%
-equivariant bijective map from the set $Cl_{str\text{ }\theta \text{-}%
reg}(G,\theta )$ of $\theta $-twisted conjugacy classes of strongly $\theta $%
-regular elements in $G(\mathbb{C})$ to the corresponding set $Cl_{str\text{ 
}\theta ^{\ast }\text{-}reg}(G^{\ast },\theta ^{\ast })$ for $(G^{\ast
},\theta ^{\ast }).$

For the second step, the endoscopic datum $\mathfrak{e}$ provides a unique $%
\Gamma $-equivariant surjective map from the set $Cl_{ss}(H)$ of semisimple
conjugacy classes in $H(\mathbb{C})$ to $Cl_{\theta ^{\ast }\text{-}%
ss}(G^{\ast },\theta ^{\ast }).\ $The inverse image of $Cl_{str\text{ }%
\theta ^{\ast }\text{-}reg}(G^{\ast },\theta ^{\ast })$ is, by definition,
the set $Cl_{str\text{ }G\text{-}reg}(H)$ of strongly $G$-regular conjugacy
classes in $H(\mathbb{C});$ see \cite[Lemma 3.3.C]{KS99}.

Third, the $z$-extension $H_{1}\rightarrow H$ provides a $\Gamma $%
-equivariant surjective map from $Cl_{ss}(H_{1})$ to $Cl_{ss}(H),$ and then
by restriction, a $\Gamma $-equivariant surjective map from $Cl_{str\text{ }G%
\text{-}reg}(H_{1})$ to $Cl_{str\text{ }G\text{-}reg}(H).$

In summary, we have established the following diagram with all arrows $%
\Gamma $-equivariant.

\begin{equation}
\begin{array}{ccccc}
Cl_{str\text{-}\theta \text{-}reg}(G,\theta ) &  &  &  &  \\ 
& \searrow &  &  &  \\ 
&  & Cl_{str\text{-}\theta ^{\ast }\text{-}reg}(G^{\ast },\theta ^{\ast }) & 
&  \\ 
&  &  & \nwarrow &  \\ 
&  & \uparrow &  & Cl_{str\text{ }G\text{-}reg}(H_{1}) \\ 
&  &  & \swarrow &  \\ 
&  & Cl_{str\text{ }G\text{-}reg}(H) &  & 
\end{array}%
\end{equation}

Turning now to real points, by the \textit{very regular set} in $H_{1}(%
\mathbb{R})\times G(\mathbb{R})$ we mean the set of all pairs $(\gamma
_{1},\delta )$ where $\gamma _{1}\in H_{1}(\mathbb{R})$ is strongly $G$%
-regular and $\delta \in G(\mathbb{R})$ is strongly $\theta $-regular.
Restricting to the real points of the classes in (3.2.1) we obtain maps from
stable $\theta $-twisted conjugacy classes of strongly $\theta $-regular
elements in $G(\mathbb{R})$ to stable $\theta ^{\ast }$-twisted conjugacy
classes of strongly $\theta ^{\ast }$-regular elements in $G^{\ast }(\mathbb{%
R})$, from stable conjugacy classes of strongly $G$-regular elements in $H(%
\mathbb{R})$ to stable $\theta ^{\ast }$-twisted conjugacy classes of
strongly $\theta ^{\ast }$-regular elements in $G^{\ast }(\mathbb{R})$, and
from the set of stable classes of strongly regular elements in $H_{1}(%
\mathbb{R})$ to the stable conjugacy classes of strongly regular elements in 
$H(\mathbb{R}).$ Because $H_{1}\rightarrow H$ is a $z$-extension, the last
map is surjective and remains surjective when we replace "strongly regular"
by "strongly $G$-regular". As in \cite[Section 3.3]{KS99}, we now define a 
\textit{norm correspondence on the very regular set}: $\delta \in G(\mathbb{R%
})$ has norm $\gamma _{1}$ in $H_{1}(\mathbb{R}),$ \textit{i.e.,} $(\gamma
_{1},\delta )$ lies in the norm correspondence, if and only if the images of
the respective stable classes of $\gamma _{1},\delta $ have the same image
among the stable $\theta ^{\ast }$-twisted conjugacy classes of strongly $%
\theta ^{\ast }$-regular elements in $G^{\ast }(\mathbb{R}).$

To attach data to the norm correspondence as in \cite[Section 4.4]{KS99},
consider strongly $\theta $-regular $\delta \in G(\mathbb{R})$. Then
unraveling the definition of the last paragraph shows that $\delta $ has
norm $\gamma _{1}$ in $H_{1}(\mathbb{R})$ if and only if there exist a $%
\theta ^{\ast }$-stable pair $(B,T)$ in $G^{\ast }$ with $T$ defined over $%
\mathbb{R}$ and elements $g$ in $G_{sc}^{\ast },$ $\delta ^{\ast }$ in $T$
such that%
\begin{equation}
\delta ^{\ast }=g.\eta (\delta ).\theta ^{\ast }(g)^{-1}
\end{equation}%
and the image $\gamma $ of $\delta ^{\ast }$ under some admissible $%
T\rightarrow T_{\theta ^{\ast }}\rightarrow T_{H}$ coincides with the image
of $\gamma _{1}$ under $H_{1}\rightarrow H.$ See \cite[Section 3.3]{KS99}.
This is summarized in the following diagram, where $N$ denotes the
projection $T\rightarrow T_{\theta ^{\ast }}$ to coinvariants.%
\begin{equation*}
\begin{array}{ccccccc}
G(\mathbb{R})\ni \delta & \longrightarrow & \delta ^{\ast }\in T &  &  &  & 
\\ 
&  &  &  &  &  &  \\ 
&  & \downarrow &  &  &  & \gamma _{1}\in H_{1}(\mathbb{R}) \\ 
&  &  &  &  & \swarrow &  \\ 
&  & N\delta ^{\ast }\in T_{\theta ^{\ast }}(\mathbb{R}) & \longrightarrow & 
\gamma \in T_{H}(\mathbb{R}) &  & 
\end{array}%
\end{equation*}

As in \cite[Section 3.3]{KS99}, we say a maximal torus $T$ in $G^{\ast }$ is 
$\theta ^{\ast }$\textit{-admissible} if there exists a $\theta ^{\ast }$%
-stable pair $(B,T)$ in $G^{\ast }.$ Also, $T_{H}$ and its inverse image $%
T_{1}$ in $H_{1}$ are $\theta ^{\ast }$-\textit{norm groups for\ }$T$ if
there exists an admissible $T\rightarrow T_{\theta ^{\ast }}\rightarrow
T_{H} $. Then write $T_{1}$ for the inverse image of $T_{H}$ in $H_{1}.$
Every maximal torus over $\mathbb{R}$ in $H_{1}$ is a $\theta ^{\ast }$-norm
group for some $\theta ^{\ast }$-admissible maximal torus $T$ in $G^{\ast }$ 
\cite[Lemma 3.3.B]{KS99}.

In regard to (3.2) we note the following for use in calculations.

\begin{remark}
Suppose $(B_{\delta },T_{\delta })$ is preserved by $Int(\delta )\circ
\theta $, where $\delta $ is a strongly $\theta $-regular element of $G(%
\mathbb{R})$ as above. Then $T_{\delta }$ is the centralizer in $G$ of the
abelian reductive subgroup $Cent_{\theta }(\delta ,G)$ of $G.$ We may
arrange that $Int(g)\circ \eta $ carries $(B_{\delta },T_{\delta })$ to $%
(B,T)\ $and $Cent_{\theta }(\delta ,G)$ to $T^{\theta ^{\ast }}$ with the
restriction of $Int(g)\circ \eta $ to $T_{\delta }$ defined over $\mathbb{R}$%
.
\end{remark}

(II) Here we consider the effect of replacing $\theta $ in (I) by $\theta
^{\prime }$ of the form $Int(g_{\mathbb{R}})\circ \theta ,$ where $g_{%
\mathbb{R}}\in G(\mathbb{R}).$ The norm correspondence is no longer
canonical but there is a quick and transparent definition in this case.
Namely, the map $\delta \rightarrow \delta .g_{\mathbb{R}}$ carries strongly 
$\theta ^{\prime }$-regular elements in $G(\mathbb{R})$ to strongly $\theta $%
-regular elements in $G(\mathbb{R}),$ providing a bijection between the
stable classes of strongly $\theta ^{\prime }$-regular elements and the
stable classes of strongly $\theta $-regular elements. We then extend the
definition of the norm correspondence on stable classes to this case in the
obvious way. This norm for $\theta ^{\prime }$ depends on our choice of $g_{%
\mathbb{R}}$ and so we use the terminology $g_{\mathbb{R}}$-norm. The
dependence is that $g_{\mathbb{R}}$ may be replaced by $z_{\mathbb{R}}g_{%
\mathbb{R}},$ with $z_{\mathbb{R}}\in Z_{G}(\mathbb{R}).$ Then strongly $G$%
-regular $\gamma _{1}$ in endoscopic group $H_{1}(\mathbb{R})$ is a $g_{%
\mathbb{R}}$-norm of $\delta $ if and only if $\gamma _{1}$ is a $z_{\mathbb{%
R}}g_{\mathbb{R}}$-norm of $\delta z_{\mathbb{R}}.$ See Section 4.2 for the
role of $z_{\mathbb{R}}$ in transfer statements.

\subsection{Fundamental case}

Let $(G^{\ast },\theta ^{\ast })$ be a quasi-split pair. A fundamental
maximal torus $T_{1}$ in $H_{1}$ is a $\theta ^{\ast }$-norm group for some $%
\theta ^{\ast }$-admissible maximal torus $T$ in $G^{\ast }$; see Section
3.2. It will be convenient to call $(G^{\ast },\theta ^{\ast },\mathfrak{e}%
_{z}\mathfrak{)}$ \textit{fundamental} if we may choose $T$ to be
fundamental; see Remark 3.4 below.

In general, write $StrReg(G^{\ast },\theta ^{\ast })$ for the set of all
strongly $\theta ^{\ast }$-regular elements in $G^{\ast }(\mathbb{R})$ and $%
StrReg(G^{\ast },\theta ^{\ast })_{f}$ for the subset of $\theta ^{\ast }$%
-fundamental elements as defined in Section 2.4.

\begin{lemma}
$StrReg(G^{\ast },\theta ^{\ast })_{f}$ is nonempty and a union of stable $%
\theta ^{\ast }$-twisted conjugacy classes.
\end{lemma}

\begin{proof}
There is a fundamental pair $(B,T)$ in $G^{\ast }$ preserved by $\theta
^{\ast }$; see the proof of Lemma 3.8 below. Then $T(\mathbb{R})$ contains
(many) elements in $StrReg(G^{\ast },\theta ^{\ast })_{f}$. The rest is
immediate from definitions.
\end{proof}

Write $StrReg_{G^{\ast }}(H_{1})$ for the set of all strongly $G^{\ast }$%
-regular elements in $H_{1}(\mathbb{R})$ and $StrReg_{G^{\ast }}(H_{1})_{f}$
for the subset of elements $\gamma _{1}$ such that the maximal torus $%
Cent(\gamma _{1},H_{1})$ is fundamental in $H_{1}.$ We call $(\gamma
_{1},\delta )$ in the very regular set, \textit{i.e.}, in $StrReg_{G^{\ast
}}(H_{1})\times StrReg(G^{\ast },\theta ^{\ast }),$ a \textit{related pair}
if it lies in the uniquely defined norm correspondence for $(G^{\ast
},\theta ^{\ast }),$ \textit{i.e.,} if $\gamma _{1}$ is a norm of $\delta .$

\begin{lemma}
(i) $(G^{\ast },\theta ^{\ast },\mathfrak{e}_{z}\mathfrak{)}$ is \textit{%
fundamental if and only if}%
\begin{equation*}
StrReg_{G^{\ast }}(H_{1})_{f}\text{ }\times \text{ }StrReg(G^{\ast },\theta
^{\ast })_{f}
\end{equation*}%
contains a related pair.

Now assume $(G^{\ast },\theta ^{\ast },\mathfrak{e}_{z}\mathfrak{)}$ is 
\textit{fundamental}. Then (ii) each $\delta $ in $StrReg(G^{\ast },\theta
^{\ast })_{f}$ has a norm $\gamma _{1}$ in $H_{1}(\mathbb{R})$ and $\gamma
_{1}\in $ $StrReg_{G^{\ast }}(H_{1})_{f},$ (iii) if $\gamma _{1}\in $ $%
StrReg_{G^{\ast }}(H_{1})_{f}$ is a norm of strongly $\theta $-regular $%
\delta $ in $G^{\ast }(\mathbb{R})$ then $\delta \in StrReg(G^{\ast },\theta
^{\ast })_{f}.$
\end{lemma}

\begin{proof}
(i) Assume that $(G^{\ast },\theta ^{\ast },\mathfrak{e}_{z}\mathfrak{)}$ is
fundamental and choose an admissible $T\rightarrow T_{\theta ^{\ast
}}\rightarrow T_{H}$ with both $T,T_{H}$ fundamental. This provides related
pairs in $StrReg_{G^{\ast }}(H_{1})_{f}$ $\times $ $StrReg(G^{\ast },\theta
^{\ast })_{f}.$ Conversely, a related pair in $StrReg_{G^{\ast }}(H_{1})_{f}$
$\times $ $StrReg(G^{\ast },\theta ^{\ast })_{f}$ provides an admissible $%
T\rightarrow T_{\theta ^{\ast }}\rightarrow T_{H}$ with both $T,T_{H}$
fundamental, and so $(G^{\ast },\theta ^{\ast },\mathfrak{e}_{z})$ is
fundamental. To check (ii), we may replace $\delta $ with a twisted
conjugate by an element of $G^{\ast }(\mathbb{R})$ and assume that $%
T_{\delta }=T$. The result then follows easily; see \cite[Lemma 4.4.A]{KS99}%
. For (iii), suppose $(\gamma _{1},\delta )$ is a related pair with attached 
$T_{H}$ fundamental. Then Remark 3.1 implies that the stable class of
attached $\theta ^{\ast }$-admissible $T$ is uniquely determined by $\gamma
_{1}$, and (iii) follows.
\end{proof}

\begin{remark}
The argument for (iii) shows that $(G^{\ast },\theta ^{\ast },\mathfrak{e}%
_{z}\mathfrak{)}$ is \textit{fundamental if and only if }every $\theta
^{\ast }$-admissible maximal torus $T$ in $G^{\ast }$ with a fundamental
maximal torus in $H_{1}$ as $\theta ^{\ast }$-norm group is fundamental.
\end{remark}

Now consider an inner form $(G,\theta ,\eta ),$ and define $StrReg(G,\theta
)_{f}$ in the same way as $StrReg(G^{\ast },\theta ^{\ast })_{f}$. In
general,\textit{\ }we modify $StrReg_{G^{\ast }}(H_{1})_{f}$ slightly as in
Section 5.4 of \cite{KS99}. Namely, we replace $H_{1}(\mathbb{R})$ by a
suitable coset $H_{1}(\mathbb{R})^{\dag }$ of $H_{1}(\mathbb{R})$ in $H_{1}(%
\mathbb{C}).$ Then we define a subset $StrReg_{G^{\ast }}(H_{1})_{f}^{\dag }$
of this coset $H_{1}(\mathbb{R})^{\dag }\ $which may be empty. For $(\theta
,\eta )$ as in (ii) of the next lemma we take, as we may, $H_{1}(\mathbb{R}%
)^{\dag }=H_{1}(\mathbb{R}).$

\begin{lemma}
Assume that $(G^{\ast },\theta ^{\ast },\mathfrak{e}_{z}\mathfrak{)}$ is
fundamental. Then the following are equivalent for an inner form $(G,\theta
,\eta )$ of $(G^{\ast },\theta ^{\ast })$:

(i) there exists a $\theta $-fundamental element in $G(\mathbb{R}),$

(ii) there is $(\theta _{f},\eta _{f})$ in the inner class of $(\theta ,\eta
)$ such that $\theta _{f}$ preserves a fundamental splitting for $G,$

(iii) there exists a related pair in $StrReg_{G^{\ast }}(H_{1})_{f}^{\dag
}\times StrReg(G,\theta )_{f}.$
\end{lemma}

\begin{proof}
We have proved (i) $\Rightarrow $ (ii) in Lemma 2.5. For (ii) $\Rightarrow $
(iii) we may further assume that $\theta =\theta _{f}$ and $\eta $
transports $\theta $ to $\theta ^{\ast }.$ Then the assertion follows
easily. (iii) $\Rightarrow $ (i) is immediate.
\end{proof}

\begin{lemma}
Assume any one of the equivalent conditions from Lemma 3.5 is satisfied.
Then:

(i) each $\delta $ in $StrReg(G,\theta )_{f}$ has a norm $\gamma _{1}$ in $%
H_{1}(\mathbb{R})$ and $\gamma _{1}$ lies in\newline
$StrReg_{G^{\ast }}(H_{1})_{f},$

(ii) if $\gamma _{1}\in $ $StrReg_{G^{\ast }}(H_{1})_{f}$ is a norm of
strongly $\theta $-regular $\delta $ in $G(\mathbb{R})$\newline
then $\delta $ lies in $StrReg(G,\theta )_{f}.$
\end{lemma}

\begin{proof}
We may assume that $\theta $ preserves fundamental splitting $spl_{G},$ that 
$\theta ^{\ast }$ preserves fundamental splitting $spl_{Wh}$ of Whittaker
type, and that $\eta $ transports $\theta $ to $\theta ^{\ast }$. Recall
that $Int(\delta )\circ \theta $ preserves the fundamental pair $(B_{\delta
},T_{\delta }).$ Extend the pair to a fundamental splitting $spl_{\delta }.$
Then there is $t_{\delta }$ in $(T_{\delta })_{sc}(\mathbb{R})$ such that $%
Int(t_{\delta }\delta )\circ \theta $ preserves $spl_{\delta }.$ Here, as
usual, we have used the same notation $t_{\delta }$ for the image of $%
t_{\delta }$ in $(T_{\delta })(\mathbb{R})$ under $G_{sc}\rightarrow G$. We
now choose $g$ in $G_{sc}$ such that $Int(g)$ carries $spl_{\delta }$ to $%
spl_{G}.$ Let $T_{G}$ be the elliptic maximal torus specified by $spl_{G}.$
Then $g_{\sigma }=g\sigma (g)^{-1}$ lies in $(T_{G})_{sc},$ $%
t_{G}=gt_{\delta }^{-1}g^{-1}$ lies in $(T_{G})_{sc}(\mathbb{R})$ and $%
\delta _{G}=g\delta \theta (g)^{-1}$ is of the form $zt_{G},$ where $z$ is
central. Also 
\begin{equation}
\sigma (z)^{-1}z=\sigma (\delta _{G})^{-1}\delta _{G}=(1-\theta )g_{\sigma },
\end{equation}%
so that $N_{\theta }(z)$ lies in $(T_{G})_{\theta }(\mathbb{R}).$ Now apply
the twist $\eta $ which carries $spl_{G}$ to $spl_{Wh}.$ Then (i), (ii)
follow; see Lemma 4.4.A of \cite{KS99}.
\end{proof}

\begin{example}
For general $(G^{\ast },\theta ^{\ast })$, consider a basic SED $\mathfrak{e}%
_{z},$ i.e., assume that $s=1.$ Then an argument along the same lines as
that for Lemma 3.3 shows that $(G^{\ast },\theta ^{\ast },\mathfrak{e}_{z}%
\mathfrak{)}$ is fundamental.
\end{example}

\subsection{Cuspidal-elliptic setting}

By the\textit{\ }cuspidal-elliptic\textit{\ }setting we mean that $G^{\ast }$%
, or equivalently an inner form of $G^{\ast }$, is cuspidal and that the
endoscopic datum $\mathfrak{e}$ is \textit{elliptic} in the sense that the
identity component of the $\Gamma $-invariants in the center of $H^{\vee }$
lies in the center of $G^{\vee }$ \cite{KS99}. We then call $H_{1}$ an 
\textit{elliptic endoscopic group}.

\begin{lemma}
(i) \textit{Assume }$G^{\ast }$\textit{\ is cuspidal. Then} $(G^{\ast
})^{\theta ^{\ast }}$\textit{\ is cuspidal and there exists an elliptic }$%
\theta ^{\ast }$-admissible\textit{\ maximal torus }$T$\textit{\ in }$%
G^{\ast }$.

(ii) Assume also \textit{that }$\mathfrak{e}$\textit{\ is elliptic. }Then $%
H_{1}$ is cuspidal and \textit{each elliptic }$T_{1}$\textit{\ in }$H_{1}$%
\textit{\ is a }$\theta ^{\ast }$-\textit{norm group for each elliptic }$%
\theta ^{\ast }$-admissible\textit{\ }$T$\textit{\ in }$G^{\ast }$.
\end{lemma}

\begin{proof}
There is no harm, for both (i) and (ii), in assuming that $G^{\ast }$ is
semisimple and simply-connected, so that $I=(G^{\ast })^{\theta ^{\ast }}$
is connected (as well as reductive) as algebraic group. Consider a pair $%
(B^{1},T^{1}),$ where $T^{1}$ is a fundamental maximal torus defined over $%
\mathbb{R}$ in $I$ and $B^{1}$ is any Borel subgroup of $I$ containing $%
T^{1}.$ Set $T=Cent(T^{1},G^{\ast })$ and $B=Norm(B^{1},G^{\ast }),$ so that 
$(B,T)$ is a $\theta ^{\ast }$-stable pair for $G^{\ast }$. Then $T$ must be
fundamental, for otherwise $T$ would have a real root and then a multiple of
the restriction of this root to $T^{1}=T^{\theta ^{\ast }}$ would provide us
with a real root for $T^{1}$ in $I$; no such root exists since $T^{1}$ is
fundamental. (i) then follows. For (ii), let $T_{H}$ be a fundamental
maximal torus in $H.$ Then there is some admissible isomorphism $%
T_{H}\rightarrow T_{\theta ^{\ast }}$ associated to a $\theta ^{\ast }$%
-admissible $T$.\textit{\ }Attach to $H$ the standard endoscopic group $J$
for $I$ as in Section 4.2 of \cite{KS99}. Then $T^{1}$ is (isomorphic to) a
fundamental maximal torus in $J,$ and moreover $J$ is elliptic because $H$
is. Thus, by (ii) in the case of standard endoscopy, $T^{1}$ is anisotropic
modulo $Z_{I}.$ Since $T$ is then anisotropic modulo $Z_{G^{\ast }}$ as in
(i), $T_{H}$ is anisotropic modulo $Z_{H},$ and (ii) follows.
\end{proof}

\begin{corollary}
$(G^{\ast },\theta ^{\ast },\mathfrak{e}_{z}\mathfrak{)}$ is fundamental in
the sense of Section 3.3.
\end{corollary}

Consider an inner form $(G,\theta ,\eta )$. We write $sr$-$ell(G,\theta )$
for the set of all $\theta $\textit{-}elliptic strongly $\theta $-regular
elements in $G(\mathbb{R})$ and $sGr$-$ell(H_{1})^{\dag }$ for the set of
all strongly $G$-regular elliptic elements in $H_{1}(\mathbb{R})^{\dag }.$

\begin{corollary}
The following are equivalent:

(i) there exists a $\theta $-elliptic element in $G(\mathbb{R}),$

(ii) there is $(\theta _{f},\eta _{f})$ in the inner class of $(\theta ,\eta
)$ such that $\theta _{f}$ preserves a fundamental splitting for $G,$

(iii) there exists a related pair in $sGr$-$ell(H_{1})^{\dag }\times sr$-$%
ell(G,\theta ).$
\end{corollary}

\begin{proof}
By Lemma 3.8, this is a special case of Lemma 3.5.
\end{proof}

\begin{corollary}
Assume any one of the conditions of Corollary 3.10 is satisfied. Then:

(i) each $\delta $ in $sr$-$ell(G,\theta )$ has a norm $\gamma _{1}$ in $%
H_{1}(\mathbb{R})$ and $\gamma _{1}\in $ $sGr$-$ell(H_{1}),$

(ii) if $\gamma _{1}\in sGr$-$ell(H_{1})$ is a norm of strongly $\theta $%
-regular $\delta $ in $G(\mathbb{R})$ then $\delta \in sr$-$ell(G,\theta ).$
\end{corollary}

\begin{proof}
By Lemma 3.8, this is a special case of Lemma 3.3.
\end{proof}

\subsection{Consequences for geometric transfer factors}

We conclude by summarizing some of the results of Sections 3.3, 3.4 in terms
of the transfer factor $\Delta $ of \cite{KS99} (see also \cite{KS12, Sh14})
for the matching of orbital integrals, \textit{i.e.,} for geometric twisted
transfer \cite{Sh12}. The factor $\Delta $ is defined on the very regular
set of Section 3.2. By construction, $\Delta (\gamma _{1},\delta )\neq 0$ if
and only if $(\gamma _{1},\delta )$ is a related pair, \textit{i.e.}, $%
\gamma _{1}$ is a norm of $\delta .$ We consider (i) transfer for
quasi-split data $(G^{\ast },\theta ^{\ast })$ with SED $\mathfrak{e}_{z}$
and (ii) transfer for an inner form of the quasi-split data in (i) when $%
\mathfrak{e}_{z}$ is fundamental. Then Lemmas 3.3, 3.6 imply that:

\begin{theorem}
(i) There exists fundamental $\gamma _{1}$ and $\theta ^{\ast }$-fundamental 
$\delta $ such that $\Delta (\gamma _{1},\delta )\neq 0$ if and only if $%
(G^{\ast },\theta ^{\ast },\mathfrak{e}_{z}\mathfrak{)}$ is fundamental.

(ii) Assume $(G^{\ast },\theta ^{\ast },\mathfrak{e}_{z}\mathfrak{)}$ is
fundamental and that $(G,\theta ,\eta )$ is an inner form. Then there exist
fundamental $\gamma _{1}$ and $\theta $-fundamental $\delta $ such that $%
\Delta (\gamma _{1},\delta )\neq 0$ if and only if there exists a $\theta $%
-fundamental element in $G(\mathbb{R}).$
\end{theorem}

From this and Lemma 3.4.1 we conclude:

\begin{corollary}
In the cuspidal-elliptic setting:

(i) there exist elliptic $\gamma _{1}$ and $\theta ^{\ast }$-elliptic $%
\delta $ such that $\Delta (\gamma _{1},\delta )\neq 0$ and

(ii) for an inner form $(G,\theta ,\eta ),$ there exist elliptic $\gamma
_{1} $ and $\theta $-elliptic $\delta $ such that $\Delta (\gamma
_{1},\delta )\neq 0$ if and only if there exists a $\theta $-elliptic
element in $G(\mathbb{R}).$
\end{corollary}

We will return to the results of Sections 3.3 and 3.4 in \cite{ShII}.

\section{\textbf{Formulating spectral factors}}

We turn now to some remarks on transfer statements in the setting from Lemma
2.5. We have checked that this setting captures all nontrivial geometric
transfer on the fundamental very regular set. There is an analogous
statement for the spectral side which we will introduce now but make precise
and verify later; see Part 9. We will limit our discussion in the present
section to the cuspidal-elliptic setting, as the general fundamental case
follows quickly.

\subsection{Transfer statements}

For the main case I, we consider an inner form $(G,\theta ,\eta )$ of $%
(G^{\ast },\theta ^{\ast })$ for which (i) the transport of $spl_{Wh}$ to $G$
by $\eta $ is fundamental and (ii) $\theta $ is the transport of $\theta
^{\ast }$ to $G$ by $\eta $. We have assumed for convenience that $G$ is
cuspidal and the endoscopic datum $\mathfrak{e}$ is elliptic. Also for
convenience, we will discuss transfer for the tempered rather than the
essentially tempered spectrum.

First recall geometric transfer. Test functions are Harish-Chandra Schwartz
functions; we consider functions $f\in \mathcal{C}(G(\mathbb{R}),\theta )$
and $f_{1}\in \mathcal{C}(H_{1}(\mathbb{R}),\varpi _{1})$ \cite[Section 1]%
{Sh12}. We may also use $C_{c}^{\infty }(G(\mathbb{R}),\theta )$ and $%
C_{c}^{\infty }(H_{1}(\mathbb{R}),\varpi _{1})$ by Bouaziz's Theorem (see 
\cite[Section 2]{Sh12}), as we will need in (8.2) for the generally
nontempered transfer of Adams-Johnson. Measures and integrals will be
defined and normalized as in \cite{Sh12}. To be more careful, we should use
test measures in place of test functions throughout, in order to have the
transfer depend only on the normalization of transfer factors. However, this
will be ignored here; see instead the note \cite{Sh}.

Theorem 2.1 of \cite{Sh12} shows that for all $f\in \mathcal{C}(G(\mathbb{R}%
),\theta )$ there exists $f_{1}\in \mathcal{C}(H_{1}(\mathbb{R}),\varpi
_{1}) $ such that%
\begin{equation}
SO(\gamma _{1},f_{1})=\sum_{\delta ,\text{ }\theta \text{-}conj}\Delta
(\gamma _{1},\delta )\text{ }O^{\theta ,\varpi }(\delta ,f)
\end{equation}%
for all strongly $G$-regular $\gamma _{1}$ in $H_{1}(\mathbb{R}).$ Here $%
O^{\theta ,\varpi }$ denotes a $(\theta ,\varpi )$-twisted orbital integral
and $SO$ denotes a standard (untwisted) stable orbital integral. We write $%
f_{1}\in Trans_{\theta ,\varpi }(f).$

Suppose $\pi _{1}$ is a tempered irreducible admissible representation of $%
H_{1}(\mathbb{R})$ and $\Pi _{1}$ is its packet. We will assume, usually
without further mention, that $\pi _{1}(Z_{1}(\mathbb{R}))$ acts by the
character $\varpi _{1};$ recall $Z_{1}$ is the central torus $%
Ker(H_{1}\rightarrow H).$ Let $St$-$Tr$ $\pi _{1}$ be the stable tempered
distribution%
\begin{equation*}
f_{1}\rightarrow \sum_{\pi _{1}^{\prime }\in \Pi _{1}}Trace\pi _{1}^{\prime
}(f_{1}).
\end{equation*}%
Because $f_{1}\in \mathcal{C}(H_{1}(\mathbb{R}),\varpi _{1})$ we have taken $%
\pi _{1}(f_{1})$ as the operator%
\begin{equation*}
\int_{H_{1}(\mathbb{R)}/Z_{1}(\mathbb{R)}}f_{1}(h_{1})\pi _{1}(h_{1})\frac{%
dh_{1}}{dz_{1}}.
\end{equation*}%
Following the case of standard endoscopic transfer we may consider the
linear form $f\rightarrow St$-$Tr$ $\pi _{1}(f_{1})$ on $\mathcal{C}(G(%
\mathbb{R}),\theta )$, where $f_{1}$ is attached to $f$ as in (4.1). For the
present discussion we restrict the form to $C_{c}^{\infty }(G(\mathbb{R}%
),\theta ).$ It is well-defined by Lemma 5.3 of \cite{Sh79a}, and results of
Waldspurger \cite{Wa14} (see also \cite{Me13}) show that it is a linear
combination of twisted traces of representations of $G(\mathbb{R}).$ Our
purpose is different. We want to describe certain coefficients closely
related to the geometric factors and then later establish that they are
correct for such a spectral transfer. Our interest in the spectral transfer
statement (4.3)\ below is in certain constraints it places on our factors.
With these constraints in mind we will verify various lemmas before making
our definitions. For example, Lemma 9.2 will be the spectral analogue of
Corollary 3.10, namely that our present assumption on $(G,\theta ,\eta )$
captures all nonempty twistpackets of discrete series representations for
inner forms of $(G^{\ast },\theta ^{\ast })$. Other results require more
effort and for these we will introduce further tools.

Let $\pi $ be a tempered irreducible admissible representation of $G(\mathbb{%
R})$ and $\Pi $ denote its packet. We use the same notation for a
representation and its isomorphism class; we may also work with unitary
representations and unitary isomorphisms. For a related pair of Langlands
parameters (see Part 9) we consider the corresponding packets $\Pi _{1}$ for 
$H_{1}(\mathbb{R})$ and $\Pi $ for $G(\mathbb{R})$. The construction of
endoscopic data ensures that the packet $\Pi $ is preserved under the map%
\begin{equation*}
\pi \rightarrow \varpi ^{-1}\otimes (\pi \circ \theta ).
\end{equation*}%
This last property is a simple condition on the Langlands parameter of $\Pi $%
; whenever it is satisfied we call the attached packet $(\theta ,\varpi )$%
\textit{-stable}. Thus we may define a twisted trace on $\oplus _{\pi
^{\prime }\in \Pi }$ $\pi ^{\prime }.$ Only those $\pi ^{\prime }$ fixed by
the map will contribute nontrivially. We then define $\Pi ^{\theta ,\varpi }$
to be the subset of $\Pi $ consisting of such $\pi ^{\prime }$ and call $\Pi
^{\theta ,\varpi }$ a \textit{twistpacket} for $(\theta ,\varpi ).$

Suppose $\pi $ belongs to the twistpacket $\Pi ^{\theta ,\varpi }$ and that
the unitary operator $\pi (\theta ,\varpi )$ on the space of $\pi $
interwines $\pi \circ \theta $ and $\varpi \otimes \pi $ or, more precisely,
that%
\begin{equation}
\pi (\theta (g))\circ \pi (\theta ,\varpi )=\varpi (g).(\pi (\theta ,\varpi
)\circ \pi (g)),
\end{equation}%
for $g\in G(\mathbb{R})$. Then by the twisted trace of $\pi $ we mean the
linear form%
\begin{equation*}
f\rightarrow Trace\text{ }\pi (f)\circ \pi (\theta ,\varpi ).
\end{equation*}%
Note that we have not fixed a normalization of the operator $\pi (\theta
,\varpi ).$ Also, if $f$ is replaced by $g\rightarrow f(xg\theta (x)^{-1})$
then $Trace$ $\pi (f)$ $\pi (\theta ,\varpi )$ is multiplied by $\varpi (x),$
for $x\in G(\mathbb{R})$.

Spectral transfer factors will be nonzero complex coefficients $\Delta (\pi
_{1},\pi )$ such that%
\begin{equation}
St\text{-}Trace\text{ }\pi _{1}(f_{1})=\tsum\nolimits_{\pi \in \Pi ^{\theta
,\varpi }}\Delta (\pi _{1},\pi )\text{ }Trace\text{ }\pi (f)\text{ }\pi
(\theta ,\varpi ).
\end{equation}

The factors $\Delta (\pi _{1},\pi )$ depend on how we normalize the
geometric factors $\Delta (\gamma _{1},\delta )$ that prescribe the
correspondence $(f,f_{1}).$ Following the method for standard transfer we
will introduce a geometric-spectral compatibility factor. For standard
transfer this factor was canonical. In the twisted case there is a new
dependence: the choice of normalization for the operators $\pi (\theta
,\varpi )$, $\pi \in \Pi ^{\theta ,\varpi }.$ We may multiply $\pi (\theta
,\varpi )$ by a nonzero complex number $\lambda $ (of absolute value one
since we have required unitarity). In standard endoscopy the term $\Delta
_{II}$ in $\Delta (\pi _{1},\pi )$ comes from the explicit local
representation of $f\rightarrow Trace$ $\pi (f)$ around the identity. In the
twisted case, we consider a similar twisted term for $f\rightarrow Trace$ $%
\pi (f)\pi (\theta ,\varpi )$ around a certain point, in general not the
identity element. We will then see that multiplying $\pi (\theta ,\varpi )$
by $\lambda $ has the effect of dividing $\Delta _{II}$ by $\lambda .$ No
other term in $\Delta (\pi _{1},\pi )$ will depend on $\pi (\theta ,\varpi )$
and so%
\begin{equation*}
\Delta (\pi _{1},\pi )\text{ }Trace\text{ }\pi (f)\text{ }\pi (\theta
,\varpi )
\end{equation*}%
will be independent of the choice for $\pi (\theta ,\varpi ).$ Then the
(geometric-spectral) compatibility factor $\Delta (\pi _{1},\pi ;\gamma
_{1},\delta )$ will depend on $\pi (\theta ,\varpi )$ but the quotient%
\begin{equation*}
\Delta (\pi _{1},\pi )\diagup \Delta (\pi _{1},\pi ;\gamma _{1},\delta )
\end{equation*}%
will not. We conclude that we may define geometric-spectral compatibility as
in the standard case \cite[Section 12]{Sh10}.

\subsection{Additional twist by an element of $G(\mathbb{R})$}

We now consider the setting II where we twist an automorphism $\theta $ as
in I by an element $g_{\mathbb{R}}$ of $G(\mathbb{R})$. This yields no new
twistpackets but it will be useful to have a precise formulation for
transfer with the twisted automorphism.

Denote by $\Delta _{g_{\mathbb{R}}}$ the geometric transfer factors defined
using $g_{\mathbb{R}}$-norms. Suppose we replace $g_{\mathbb{R}}$ by $z_{%
\mathbb{R}}g_{\mathbb{R}},$ where $z_{\mathbb{R}}$ lies in the center of $G(%
\mathbb{R}).$ Then the relative factors%
\begin{equation*}
\Delta _{g_{\mathbb{R}}}(\gamma _{1},\delta ;\gamma _{1}^{\prime },\delta
^{\prime })
\end{equation*}%
and%
\begin{equation*}
\Delta _{z_{\mathbb{R}}g_{\mathbb{R}}}(\gamma _{1},\delta z_{\mathbb{R}%
};\gamma _{1}^{\prime },\delta ^{\prime }z_{\mathbb{R}})
\end{equation*}%
coincide. Indeed we see quickly from the definitions that the only
difference between the two is that the element $z_{\mathbb{R}}$ is inserted
in the element $D$ constructed for $\Delta _{III}$ (see p. 33 of \cite{KS99}%
) where it clearly has no effect. This property of the relative factors
allows us to normalize absolute factors so that%
\begin{equation*}
\Delta _{z_{\mathbb{R}}g_{\mathbb{R}}}(\gamma _{1},\delta z_{\mathbb{R}%
})=\Delta _{g_{\mathbb{R}}}(\gamma _{1},\delta )
\end{equation*}%
for all very regular related pairs $(\gamma _{1},\delta )$ for the $g_{%
\mathbb{R}}$-norm.

The choice of $z_{\mathbb{R}}$ affects the correspondence on test functions.
If $f_{1}\in Trans(f)$ for $g_{\mathbb{R}}$-norms then clearly $f_{1}\in
Trans(f_{z_{\mathbb{R}}})$ for $z_{\mathbb{R}}g_{\mathbb{R}}$-norms, where $%
f_{z_{\mathbb{R}}}$ denotes the translate of $f$ by $(z_{\mathbb{R}})^{-1}.$
The extended version of Lemma 5.1.C at the bottom of p.53 of \cite{KS99}
applies also to $g_{\mathbb{R}}$-norms since it is easily rewritten as a
statement about relative factors. Thus if $z_{1}\in Z_{H_{1}}(\mathbb{R})$
has image in $Z_{H}(\mathbb{R})$ equal to the image of $z_{\mathbb{R}}$
under $N$, \textit{i.e., }if $(z_{1},z_{\mathbb{R}})$ belongs to the group $%
C(\mathbb{R})$ from (5.1) of \cite{KS99}, then there is quasicharacter $%
\varpi _{C}$ on $C(\mathbb{R})$ such that 
\begin{equation*}
\Delta _{g_{\mathbb{R}}}(z_{1}\gamma _{1},\delta z_{\mathbb{R}})=\varpi
_{C}(z_{1},z_{\mathbb{R}})^{-1}\Delta _{g_{\mathbb{R}}}(\gamma _{1},\delta ).
\end{equation*}%
A calculation with (4.1) now shows that%
\begin{equation*}
\varpi _{C}(z_{1},z_{\mathbb{R}}).(f_{1})_{z_{1}}\in Trans(f)
\end{equation*}%
for $z_{\mathbb{R}}g_{\mathbb{R}}$-norms.

In Lemma 9.5 we will prove that the central characters $\varpi _{\pi
_{1}},\varpi _{\pi }$ for a related pair $(\pi _{1},\pi )$ have the property
that%
\begin{equation}
\varpi _{\pi _{1}}(z_{1}).\varpi _{\pi }(z)^{-1}=\varpi _{C}(z_{1},z_{%
\mathbb{R}})
\end{equation}%
for all $(z_{1},z_{\mathbb{R}})$ in $C(\mathbb{R}).$ This and (4.2) imply
that if the spectral factors $\Delta _{g_{\mathbb{R}}}(\pi _{1},\pi )$ and $%
\Delta _{z_{\mathbb{R}}g_{\mathbb{R}}}(\pi _{1},\pi )$ are compatible with
geometric $\Delta _{g_{\mathbb{R}}}\ $and $\Delta _{z_{\mathbb{R}}g_{\mathbb{%
R}}}$ respectively, then%
\begin{equation*}
\Delta _{g_{\mathbb{R}}}(\pi _{1},\pi )=\Delta _{z_{\mathbb{R}}g_{\mathbb{R}%
}}(\pi _{1},\pi )
\end{equation*}%
for all pairs $(\pi _{1},\pi )$ as in Section 4.1. Here if $\pi (\theta
,\varpi )$ is used in the definition on the left then $\varpi _{\pi }(z_{%
\mathbb{R}}).\pi (\theta ,\varpi )$ is to be used on the right. Our
conclusion is then that the spectral factors will be independent of the
choice for $g_{\mathbb{R}}.$

\section{\textbf{Packets and parameters I}}

Next we review briefly Langlands parameters and Arthur parameters for real
groups \cite{La89, Ar89}. We make a construction in Section 5.2 that
attaches a $c$\textit{-Levi group} to a parameter. We will show in
subsequent sections how this group provides useful additional information
about the parameters we are concerned with. Twisting will be ignored until
Part 9.

\subsection{Langlands parameters, Arthur parameters}

Consider a homomorphism of the form%
\begin{equation*}
\psi =(\varphi ,\rho ):W_{\mathbb{R}}\times SL(2,\mathbb{C})\rightarrow
^{L}G,
\end{equation*}%
where $\varphi :W_{\mathbb{R}}\rightarrow $ $^{L}G$ is an essentially
bounded admissible homomorphism and $\rho $ is a continuous homomorphism of $%
SL(2,\mathbb{C})$ into $G^{\vee }$. The conditions on $\varphi $ mean that $%
\varphi (w)=\varphi _{0}(w)\times w,$ $w\in W_{\mathbb{R}},$ where $\varphi
_{0}$ is a continuous $1$-cocycle of $W_{\mathbb{R}}$ in $G^{\vee }$ and $%
\varphi _{0}(W_{\mathbb{R}})$ is a group of semisimple elements in $G^{\vee
} $ that is bounded mod center in the sense that the image of $\varphi
_{0}(W_{\mathbb{R}})$ in the adjoint form $G_{ad}^{\vee }$ under the natural
projection $G^{\vee }\rightarrow G_{ad}^{\vee }\ $is bounded.

An element $g$ of $G^{\vee }$ acts on the set of such $\varphi $ by
conjugation: $\varphi \rightarrow Int(g)\circ \varphi $. The $G^{\vee }$%
-orbits are the essentially bounded Langlands parameters for $G^{\ast };$
see \cite{La89}. Similarly, $G^{\vee }$ acts on the set of such $\psi $ and
the orbits are the Arthur parameters for $G^{\ast };$ see \cite{Ar89}.

When we replace $G^{\ast }$ by an inner twist $(G,\eta )$ in our
considerations, we will often limit our attention to Langlands parameters
which are \textit{relevant to }$(G,\eta )$ in the usual sense that the image
of a representative is contained only in parabolic subgroups of $^{L}G$
relevant to $(G,\eta )$ \cite{La89}. The essentially bounded Langlands
parameters relevant to $(G,\eta )$ parametrize the essentially tempered
packets of irreducible admissible representations of $\ G(\mathbb{R})\ $\cite%
{La89}.

Notation: occasionally we distinguish between a homomorphism $\varphi ,\psi $
and its $G^{\vee }$-orbit $\boldsymbol{\varphi }\mathbf{,\psi }$
respectively, but much of the time we use the symbols $\varphi $ or $\psi $
for both.

Let $\psi =(\varphi ,\rho )$ be a Arthur parameter and let $S=S_{\psi }$
denote the centralizer in $G^{\vee }$ of the image of $\psi .$ Recall that
Arthur calls $\psi $ \textit{elliptic} if the identity component of $S$ is
central in $G^{\vee }$ and that this is equivalent to requiring that the
image of $\psi $ be contained in no proper parabolic subgroup of $^{L}G$ 
\cite{Ar89}.

For calculations with the Weil group $W_{\mathbb{R}}$ we fix an element $%
w_{\sigma }$ of $W_{\mathbb{R}}$ such that $w_{\sigma }$ maps to $\sigma $
under $W_{\mathbb{R}}\rightarrow \Gamma $ and $(w_{\sigma })^{2}=-1.$

\subsection{$c$-Levi group attached to a parameter}

Let $\psi =(\varphi ,\rho )$ be an Arthur parameter. Set%
\begin{equation*}
M_{\psi }^{\vee }=Cent(\varphi (\mathbb{C}^{\times }),G^{\vee }).
\end{equation*}%
Then $M_{\psi }^{\vee }=M^{\vee }$ is a connected reductive subgroup of $%
G^{\vee }$. Because $\varphi (\mathbb{C}^{\times })$ is a torus, $M^{\vee }$
is \textit{Levi} in the sense that there is a parabolic subgroup of $G^{\vee
}$ with $M^{\vee }$ as Levi subgroup. Notice that $\varphi (W_{\mathbb{R}})$
normalizes $M^{\vee }.$ We define $\mathcal{M}$ to be the subgroup of $^{L}G$
generated by $M^{\vee }$ and $\varphi (W_{\mathbb{R}}).$ Then $\mathcal{M}$
is a split extension of $W_{\mathbb{R}}$ by $M^{\vee }.$ Notice that $%
\mathcal{M}$ contains $S_{\psi }.$

While $\mathcal{M}$ is typically not endoscopic, \textit{i.e.,} it is not
the group $\mathcal{H}$ in some set $(H,\mathcal{H},s)$ of standard
endoscopic data for\ $G$, we may extract an $L$-action on $M^{\vee }$ in the
same way as for the endoscopic case. For this, recall the fixed splitting $%
spl^{\vee }=(\mathcal{B},\mathcal{T},\{X_{\alpha ^{\vee }}\})$ for $G^{\vee
} $. There is no harm in assuming that $\varphi _{0}(\mathbb{C}^{\times })$
lies in $\mathcal{T}$ and that $\varphi _{0}(w_{\sigma })$ normalizes $%
\mathcal{T}$ \cite{La89}. Then $\mathcal{T}\subset M^{\vee }$ and a simple
root $\alpha ^{\vee }$ for $M^{\vee }\cap \mathcal{B}$ is also simple for $%
\mathcal{B}$. We then use the same root vector $X_{\alpha ^{\vee }}$. Write $%
spl_{M}^{\vee }$ for this splitting $(M^{\vee }\cap \mathcal{B},\mathcal{T}%
,\{X_{\alpha ^{\vee }}\})$ for $M^{\vee }$. To define an $L$-action on $%
M^{\vee }$ we need only to specify an automorphism $\sigma _{M}$ of $M^{\vee
}$ that preserves $spl_{M}^{\vee }$ and has order at most two. Since $Int$ $%
\varphi (w_{\sigma })$ preserves $M^{\vee }$ and has order at most two as
automorphism of $\mathcal{T},$ it is clear that there is a unique such $%
\sigma _{M}$ of the form $Int[m_{\sigma }.\varphi (w_{\sigma })],$ with $%
m_{\sigma }\in M^{\vee }.$ For the $L$-action itself, $W_{\mathbb{R}}$ acts
through $W_{\mathbb{R}}\rightarrow \Gamma ;$ in particular, $w_{\sigma }$
acts as $\sigma _{M}$ and $\mathbb{C}^{\times }$ acts trivially.

Write $^{L}M_{\psi }=$ $^{L}M$ for the corresponding $L$-group $M_{\psi
}^{\vee }\rtimes W_{\mathbb{R}}$ and $M_{\psi }$ for a group defined and
quasi-split over $\mathbb{R}$ that is dual to $^{L}M_{\psi }$. In Section
5.4 we will describe explicitly the $L$-isomophisms $\xi _{M}:$ $^{L}M_{\psi
}\rightarrow \mathcal{M}$ in the critical case, that where we have the
property $\bullet $ of the next section. In Section 6.2 we will define an
embedding over $\mathbb{R}$ of the quasi-split group $M_{\psi }$ in the
quasi-split form $G^{\ast }$ in that case (and the general case follows
quickly). For now, however, the following observation will be sufficient: in
the same sense and by the same arguments as for an endoscopic group (see 
\cite[Lemma 3.3.B]{KS99}), the group $M_{\psi }$ shares various maximal tori
over $\mathbb{R}$ with an inner form $G$ of $G^{\ast }$ and all maximal tori
over $\mathbb{R}$ in $M_{\psi }$ are shared with $G^{\ast }$.

The groups $\mathcal{M}$ and $M_{\psi }$ attached to Arthur parameter $\psi
=(\varphi ,\rho )$ depend only on $\varphi $. We may make the same
definitions for any Langlands parameter $\varphi $ and then we use the
notation $M_{\varphi }$. We call the group $M_{\varphi }$ a $c$\textit{-Levi
group }of $G^{\ast }$. We will define $c$-Levi groups in an inner form via
inner twists; see Section 6.2. In Section 7.4 we will see $M_{\psi }$ in a
more familiar setting, namely as a Levi subgroup defined over $\mathbb{R}$
of a parabolic subgroup preserved by a Cartan involution.

The group $M_{\varphi }$ also appears indirectly in the dual version of the
Knapp-Zuckerman decomposition of unitary principal series (see \cite[%
Sections 4, 5]{Sh82}), as we will recall briefly in Part 6. For certain
Arthur parameters $\psi ,$ a family of inner forms of $M_{\psi }$ is
introduced in \cite[Section 9]{Ko90}.

\subsection{A property for Arthur parameters}

Let $\psi =(\varphi ,\rho )$ be an Arthur parameter. As above, we choose a
representative $\psi $ such that $\varphi _{0}(\mathbb{C}^{\times })$ lies
in $\mathcal{T}$ and $\varphi _{0}(w_{\sigma })$ normalizes $\mathcal{T}$.
Consider the property:

\begin{itemize}
\item there is an element of $\mathcal{M}$ $\mathcal{\cap }$ $(G^{\vee
}\times w_{\sigma })$ that normalizes $\mathcal{T}$ and acts as $-1$ on all
roots of $G^{\vee }$.
\end{itemize}

Notice that $\bullet $ is true if and only if both $G^{\ast },$ $M_{\psi }$
are cuspidal and share an elliptic maximal torus $T,$ \textit{i.e.}, a
maximal torus that is anisotropic modulo the center of $G^{\ast }.$ Here we
may replace $G^{\ast }$ by an inner form $(G,\eta )$ if we wish; we then
write $T_{G}$ in place of $T\ $and assume harmlessly that $\eta $ maps $%
T_{G} $ to $T$ over $\mathbb{R}$.

\subsection{$L$-isomorphisms for attached $c$-Levi group}

We describe next the $L$-isomorphisms $\xi _{M}:$ $^{L}M\rightarrow \mathcal{%
M}$ for the case that $\bullet $ is true. There will be no harm in working
with standard $\chi $-data and we do so as it returns us to a familiar
setting. In particular, Lemma 5.4.1 below is well known; much of it is
stated in \cite{AJ87}, \cite{Ar89} without details of proof. We give a quick
proof based on some explicit calculations we will need. These calculations
also pinpoint dependence on the critical Lemma 3.2 in \cite{La89}.

The element described in $\bullet $ may be written as $n\times w_{\sigma },$
where $n\in G^{\vee }$ normalizes $\mathcal{T}$ and represents the longest
element of the Weyl group of $\mathcal{T}$ in $G^{\vee }.$ Since $n\times
w_{\sigma }\in \mathcal{M}$, our construction of $^{L}M=M^{\vee }\rtimes W_{%
\mathbb{R}}$ yields $n_{M}\times w_{\sigma }$ in the group $^{L}M$
normalizing $\mathcal{T}$ and acting as $-1$ on the roots of $M^{\vee }.$
Then $n_{M}\in M^{\vee }$ normalizes $\mathcal{T}$ and represents the
longest element of the Weyl group of $\mathcal{T}$ in $M^{\vee }.$ Because $%
M^{\vee }$ is Levi, we may multiply $n$ by an element of $\mathcal{T}$ $\cap 
$ $(M^{\vee })_{der}\subseteq \mathcal{T}$ $\cap $ $(G^{\vee })_{der}$ to
obtain $n^{\prime }$ such that the action of $n^{\prime }\times w_{\sigma
}\in $ $^{L}G$ on the entire group $M^{\vee }$ coincides with that of $%
n_{M}\times w_{\sigma }\in $ $^{L}M.$ Notice that 
\begin{equation}
n^{\prime }\sigma (n^{\prime })=n\sigma (n),
\end{equation}%
where $\sigma $ denotes the action of $1\times w_{\sigma }\in $ $^{L}G$ on $%
G^{\vee },$ and that the action of $1\times w_{\sigma }\in $ $^{L}M$ on $%
M^{\vee }$ is given by conjugation by%
\begin{equation}
(n_{M})^{-1}n^{\prime }\times w_{\sigma }\in \mathcal{M}.
\end{equation}

Returning to the construction of an $L$-isomorphism $\xi _{M}:$ $%
^{L}M\rightarrow \mathcal{M}$, we require $\xi _{M}$ to act as the identity
on $M^{\vee }$. It remains to define $\xi _{M}$ on $W_{\mathbb{R}}.$ There
is no harm in assuming that the element $n$ above, and thus also $n^{\prime
} $, belongs to $G_{der}^{\vee },$ and that $n_{M}$ belongs to $%
M_{der}^{\vee }.$ Then set%
\begin{equation*}
\xi _{M}(w_{\sigma })=(n_{M})^{-1}n^{\prime }\times w_{\sigma },
\end{equation*}%
Let $\iota $ denote one-half the sum of the coroots for $\mathcal{T}$ in $%
\mathcal{B}$, and let $\iota _{M}$ be the corresponding term for the coroots
in $M^{\vee }\cap \mathcal{B}$. Notice that because $M^{\vee }$ is Levi, we
have%
\begin{equation}
\left\langle \iota -\iota _{M},\alpha ^{\vee }\right\rangle =0
\end{equation}%
for all roots $\alpha ^{\vee }$ of $\mathcal{T}$ in $M^{\vee }.$ This,
together with $\bullet $, implies that $\sigma _{M}$ acts on $\iota -\iota
_{M}$ as $-1.$ For $z\in \mathbb{C}^{\times },$ define the element $\xi
_{M}(z)$ of $\mathcal{T}\times z$ by%
\begin{equation*}
\xi _{M}(z)=(z/\overline{z})^{\iota -\iota _{M}}\times z.
\end{equation*}

\begin{lemma}
(i) The map $\xi _{M}$ extends to a well-defined homomorphism $\xi _{M}:W_{%
\mathbb{R}}\rightarrow \mathcal{M}$ and thence to an $L$-isomophism $\xi
_{M}:$ $^{L}M\rightarrow \mathcal{M}.$

(ii) An $L$-isomorphism $\xi _{M}^{\prime }:$ $^{L}M\rightarrow \mathcal{M}$
extending the identity on $M^{\vee }$ is of the form $\xi _{M}^{\prime
}=a\otimes \xi _{M},$ where $a$ is a $1$-cocycle of $W_{\mathbb{R}}$ in the
center $Z_{M^{\vee }}$ of $M^{\vee }.$
\end{lemma}

\begin{proof}
Following (5.1) and (5.2) in our construction above, we see that%
\begin{equation*}
\lbrack (n_{M})^{-1}n^{\prime }\times w_{\sigma }]^{2}
\end{equation*}%
may be rewritten as%
\begin{equation*}
(n_{M}\sigma _{M}(n_{M}))^{-1}.n\sigma (n)\times (-1).
\end{equation*}%
By \cite[Lemma 3.2]{La89}, this is%
\begin{equation*}
(-1)^{-2\iota _{M}}(-1)^{2\iota }\times (-1),
\end{equation*}%
and (i) then follows. Also, (ii) is immediate.
\end{proof}

\subsection{$u$-regular Arthur parameters}

We continue with an Arthur parameter $\psi =(\varphi ,\rho )$. Notice that
the image of $\rho $ lies in $M^{\vee }.$ From now on we will limit our
attention to $u$-\textit{regular }Arthur parameters. By this we mean a
parameter $\psi $ for which the image of $\rho $ contains a regular
unipotent element of $M^{\vee }.$ Then\textit{\ }$\rho $ maps regular
unipotent elements of $SL(2,\mathbb{C})$ to regular unipotent elements of $%
M^{\vee }.$ We include the case that $M^{\vee }$ is abelian. Then $\rho $ is
trivial so that $\psi =(\varphi ,triv),$ where $\varphi $ is a Langlands
parameter that is regular in the sense of \cite[Section 2]{Sh10}.\textit{\ }%
Representations in the attached essentially tempered packet have regular
infinitesimal character. On the other hand, if $M^{\vee }$ is nonabelian
then there are $u$-\textit{regular }Arthur parameters where representations
in the attached Arthur packet have singular infinitesimal character; see
Lemma 7.4.

Observe that for a $u$-regular Arthur parameter $\psi ,$ the centralizer $%
S_{\psi }$ of the image of $\psi $ in $G^{\vee }$ consists exactly of the $%
\sigma _{M}$-invariants in $Z_{M^{\vee }}$.

\begin{lemma}
A $u$-regular Arthur parameter $\psi $ is elliptic if and only if $\bullet $
is true.
\end{lemma}

\begin{proof}
There is no harm in assuming $G$ is simply-connected and semisimple, so that 
$Z_{G^{\vee }}$ is trivial. Then a nontrivial torus in the $\sigma _{M}$%
-invariants of $Z_{M^{\vee }}$ determines a nontrivial $\mathbb{R}$-split
torus in a fundamental maximal torus of $M$, and conversely.
\end{proof}

In the next lemma we assume $\psi $ is elliptic since we have yet to
describe $\xi _{M}$ in general (see \cite{ShII}). By construction, we have $%
\varphi (W_{\mathbb{R}})$ contained in $\mathcal{M}$, and so we may factor $%
\varphi $ through $\xi _{M}:$ $^{L}M\rightarrow \mathcal{M}.$ Define the
Langlands parameter $\varphi _{M}$ by $\varphi =\xi _{M}\circ \varphi _{M}.$
Set $^{L}Z_{M}=Z_{M^{\vee }}\rtimes W_{\mathbb{R}}$ $\subseteq $\ $^{L}M.$

\begin{lemma}
The Langlands parameter $\varphi _{M}$ factors through $^{L}Z_{M}.$
\end{lemma}

\begin{proof}
By (5.3), $\varphi _{M}(\mathbb{C}^{\times })$ lies in $^{L}Z_{M}.$ Because $%
1\times w_{\sigma }\in $ $^{L}M$ preserves the splitting $spl_{M}^{\vee }$
of $M^{\vee }$ we may adjust $\psi $ by an element of $M^{\vee }$ to arrange
also that $\rho (SL(2,\mathbb{C}))$ contains a regular unipotent element of $%
M^{\vee }$ that is fixed by $1\times w_{\sigma }.$ Writing $\varphi
_{M}(w_{\sigma })$ as $m(w_{\sigma })\times w_{\sigma },$ we have then that
the semisimple element $m(w_{\sigma })$ commutes with a regular unipotent
element and hence is central in $M^{\vee }.$
\end{proof}

\begin{remark}
By (ii) of Lemma 5.1 we may replace $\xi _{M}$ by $\varphi $ itself. Then $%
\varphi $ factors through the trivial parameter $w\rightarrow 1\times w$;
this factoring is used for the parameters in \cite[Section 9]{Ko90}.
\end{remark}

\subsection{Langlands parameters for discrete series}

Discrete series parameters are defined in \cite[Section 3]{La89}. They are
precisely the Langlands parameters $\boldsymbol{\varphi }$ that are elliptic
as Arthur parameters, \textit{i.e.,} such that $\psi =(\varphi ,triv)$ is
elliptic, and then $M_{\varphi }$ is just the elliptic torus $T$ of Section
5.3.

We recall that there is a representative $\varphi $ for $\boldsymbol{\varphi 
}$ such that%
\begin{equation}
\varphi _{0}(z)=z^{\mu }\overline{z}^{\sigma _{T}\mu },\text{ }z\in \mathbb{C%
}^{\times },
\end{equation}%
and%
\begin{equation}
\varphi _{0}(w_{\sigma })=e^{2\pi i\lambda }n.
\end{equation}%
Here $n$ is the element of the derived group of $G^{\vee }$ constructed from 
$spl^{\vee }$ to represent the longest element of the Weyl group of $%
\mathcal{T}$ (see \cite[ Section 2.6]{LS87}). Then $\varphi (w_{\sigma })$
acts on $\mathcal{T}$ as $\sigma _{T}.$ Also $\mu ,\lambda \in X_{\ast }(%
\mathcal{T})\otimes \mathbb{C}$ and 
\begin{equation}
\frac{1}{2}(\mu -\sigma _{T}\mu )-\iota \equiv \lambda +\sigma _{T}\lambda 
\text{ }\func{mod}X_{\ast }(\mathcal{T}),
\end{equation}%
where $\iota $ is one-half the sum of the coroots for $\mathcal{T}$ in the
Borel subgroup $\mathcal{B}$ provided by $spl^{\vee }$. Notice that the
congruence implies that%
\begin{equation*}
\left\langle \mu ,\alpha ^{\vee }\right\rangle \in \mathbb{Z}
\end{equation*}%
for all roots $\alpha ^{\vee }$ of $\mathcal{T}$ in $G^{\vee }.$ We require $%
\mu $ to be strictly dominant for $spl^{\vee }$. Thus $\mu $ is determined
uniquely by $\boldsymbol{\varphi }$, while $\lambda $ is determined uniquely
modulo $\mathcal{K}$, where 
\begin{equation*}
\mathcal{K}=X_{\ast }(\mathcal{T})+\{\nu \in X_{\ast }(\mathcal{T})\otimes 
\mathbb{C}:\sigma _{T}\nu =-\nu \}.
\end{equation*}%
Finally, the representative $\varphi (\mu ,\lambda )$ for $\boldsymbol{%
\varphi }$ is determined uniquely up to the action of $\mathcal{T}$.

\subsection{Langlands parameters for limits of discrete series}

As in \cite{Sh82}, we may use (5.4) and (5.5) above to construct a Langlands
parameter $\boldsymbol{\varphi }$ with representative $\varphi =\varphi (\mu
,\lambda ),$ where $(\mu ,\lambda )$ satisfy (5.6) but the strict dominance
condition on $\mu $ is relaxed to dominance.

Notice that if $\mu _{0}\in X_{\ast }(\mathcal{T})$ is dominant then we
obtain another such parameter $\varphi (\mu +\mu _{0},\lambda +\frac{1}{2}%
\mu _{0})$. This will provide a translation by the rational character $\mu
_{0}$ in the character data of Parts 6 and 7.

Also, because $\left\langle \mu +\sigma _{T}\mu ,\alpha ^{\vee
}\right\rangle =0$ for all roots $\alpha ^{\vee }$ of $\mathcal{T}$ in $%
G^{\vee },$ the image of $\varphi _{0}=\varphi _{0}(\mu ,\lambda )$ is
bounded mod center. Thus $\boldsymbol{\varphi }$ is essentially bounded.

\begin{remark}
We will now use the term \textbf{s-elliptic} for any Langlands parameter $%
\boldsymbol{\varphi }$ with representative of the form $\varphi =\varphi
(\mu ,\lambda )$\textit{.} By construction, $\bullet $ is true, so that the
attached $c$-Levi group $M_{\varphi }$ is cuspidal and shares an elliptic
maximal torus $T$ with $G$.
\end{remark}

\begin{remark}
Langlands' definition of the packet attached to relevant $\boldsymbol{%
\varphi }$ involves components of principal series representations \cite%
{La89}. Theorem 4.3.2 of \cite{Sh82} shows that $\varphi (\mu ,\lambda )$
may be used directly to identify these components as limits of discrete
series characters, nondegenerate or not.
\end{remark}

\section{\textbf{Packets of limits of discrete series}}

We pause for a more detailed analysis of the packet of tempered irreducible
representations attached to an $s$-elliptic parameter $\boldsymbol{\varphi }$
with representative $\varphi =\varphi (\mu ,\lambda )$. Namely, we expand on
Remark 5.6 using the attached $c$-Levi group $M_{\varphi }$ and the strong
generic base-point property from \cite[Section 11]{Sh08b} which is based on
Vogan's classification of generic representations.

\subsection{Character data, Whittaker data}

Let $\mathcal{C}$ denote the closed Weyl chamber in $X_{\ast }(\mathcal{T}%
)\otimes \mathbb{C}$ dominant for $spl^{\vee }$. Recall that $\mu \in 
\mathcal{C}.$ We choose an inner automorphism of $G^{\ast }$ carrying $%
spl^{\ast }=(spl^{\vee })^{\vee }$ to a fundamental splitting $%
spl_{Wh}=(B,T,\{X_{\alpha }\})$ for $G^{\ast }$ of Whittaker type. We thus
have a transport to $T$ of $(\mu ,\lambda ,\mathcal{C}).$ Then $(\mu
,\lambda ,\mathcal{C})$ becomes character data for a generic discrete series
or limit of discrete series representation $\pi ^{\ast }$ of $G^{\ast }(%
\mathbb{R})$.

To fix Whittaker data for $G^{\ast }$, by which we mean a $G^{\ast }(\mathbb{%
R})$-conjugacy class of the pairs $(B,\lambda )$ of \cite[Section 5.3]{KS99}%
, we choose an additive character $\psi _{\mathbb{R}}$ for $\mathbb{R}$ and
use the conjugacy class of the pair determined by $\psi _{\mathbb{R}}$ and $%
spl^{\ast }$. We may adjust the fundamental splitting $spl_{Wh}$ to arrange
that $\pi ^{\ast }$ is generic for the chosen Whittaker data. We then say
that the splitting is \textit{aligned} with the data. This determines $%
spl_{Wh}$ uniquely up to $G^{\ast }(\mathbb{R})$-conjugacy.

Let $(G,\eta )$ be an inner twist and let $spl_{f}$ be a fundamental
splitting for $G.$ We use a twist $\eta ^{\prime }$ in the inner class of $%
\eta $ to transport $spl_{f}$ to $spl_{Wh}.$ This provides us with a further
transport of $(\mu ,\lambda ,\mathcal{C})$ to the maximal torus specified by 
$spl_{f}.$ The transported triple serves as character data $(\mu _{\pi
},\lambda _{\pi },\mathcal{C}_{\pi })$ for a discrete series or limit of
discrete series representation $\pi $ of $G(\mathbb{R})$ which is determined
uniquely by the $G(\mathbb{R})$-conjugacy class of $spl_{f}.$ In the case of
limits of discrete series \textit{we must now allow} $\pi =0,$ \textit{i.e.,}
that the distribution attached to the character data is zero. We write $%
spl_{f}=spl_{\pi }=(B_{\pi },T_{\pi },\{X_{\alpha }\})$, $\eta ^{\prime
}=\eta _{\pi }$ and the character data for $\pi $ as $(\mu _{\pi },\lambda
_{\pi },\mathcal{C}_{\pi }).$

By Lemma 2.1 and the theorem cited in Remark 5.6, as $spl_{f}$ varies we
generate the packet of essentially tempered representations attached to $%
\varphi =\varphi (\mu ,\lambda )$ and possibly some zeros. By the same
theorem we obtain all zeros if and only if $\varphi $ is irrelevant to $G$.

\subsection{Characterizing nonzero limits}

Our first concern will be to detect when $\pi =0$. There is a well-known
characterization, in terms of roots, for a limit of discrete series to be
nonzero; see \cite[Theorem 1.1 (b)]{KZ82}. The precise statement is noted in
the next proof. We want a characterization in terms of the $c$-Levi group $%
M_{\varphi }.$

Consider the subgroup $M^{\ast }$ generated by $T$ and the root vectors $%
X_{\alpha }$ from $spl_{Wh}$ for which $\alpha $ is the transport to $T$ of
the coroot of a simple root of $\mathcal{T}$ in $M^{\vee }\cap \mathcal{B}.$
Because $spl_{Wh}$ is of Whittaker type, \textit{i.e.,} the simple roots are
all noncompact, the group $M^{\ast }$ is quasi-split over $\mathbb{R}$.
Moreover, we may identify the $L$-group of $M^{\ast }$ with the $L$-group $%
^{L}M_{\varphi }$ constructed in Section 5.2. For this, we reverse the
construction of Section 6.1 to determine an $\mathbb{R}$-splitting $%
spl_{M}^{\ast }$ for $M^{\ast }\ $from the fundamental splitting $spl_{Wh,M}$
attached to $spl_{Wh}$ and the additive character $\psi _{\mathbb{R}}.$ Then 
$spl_{M}^{\ast }$ is unique up to $M^{\ast }(\mathbb{R})$-conjugacy. Each
such $spl_{M}^{\ast }$ determines a unique isomorphism from $^{L}(M^{\ast })$
to $^{L}M_{\varphi }$. In summary, $M^{\ast }$ provides a concrete
realization of the quasi-split group $M_{\varphi }$.

By definition, $\eta _{\pi }$ carries $spl_{\pi }$\ to $spl_{Wh}$. Let $%
M_{\pi }=\eta _{\pi }^{-1}(M^{\ast }).$ Then $\eta _{\pi }:M_{\pi
}\rightarrow M$ is an inner twist in the inner class of $\eta $ that carries 
$T_{\pi }$ to $T^{\ast }$ over $\mathbb{R}$. We say that $M_{\pi }$ is an 
\textit{elliptic }$c$\textit{-Levi group for} $(G,\eta )$.

\begin{lemma}
$\pi $ is nonzero if and only if the inner twist $\eta _{\pi }:M_{\pi
}\rightarrow M^{\ast }$ is an $\mathbb{R}$-isomorphism.
\end{lemma}

\begin{proof}
The characterization cited (in sufficient generality for our setting) is
that $\pi \neq 0$ if and only if all $\mathcal{C}_{\pi }$-simple roots $%
\alpha $ such that $\left\langle \mu _{\pi },\alpha ^{\vee }\right\rangle =0$
are noncompact. In other words, \ $\pi \neq 0$ if and only if the splitting
of $M_{\pi }$ determined by $spl_{\pi }$ is of Whittaker type. Let $\eta
_{\pi }\sigma (\eta _{\pi })^{-1}=Int(u_{\pi }(\sigma )),$ where $u_{\pi
}(\sigma )\in T_{sc}.$ Then because $spl_{Wh}$ is of Whittaker type and $%
spl_{\pi }$ is fundamental, a calculation with root vectors shows that this
is the same as requiring that $\alpha (u_{\pi }(\sigma ))=1$ for all $%
B^{\ast }$-simple roots $\alpha $ of $T$ in $M^{\ast },$ \textit{i.e.,} that 
$u_{\pi }(\sigma )$ lies in the center of $M_{(sc)}^{\ast }.$ Here $%
M_{(sc)}^{\ast }$ denotes the inverse image of $M^{\ast }$ in $G_{sc}^{\ast
} $ under the natural projection $G_{sc}^{\ast }\rightarrow G^{\ast }.$ The
lemma then follows.
\end{proof}

\subsection{Generating packets}

Fix an inner form $(G,\eta )$ and assume s-elliptic $\varphi =\varphi (\mu
,\lambda )$ is relevant to $(G,\eta ).$ We consider the packet $\Pi $ of
representations of $G(\mathbb{R})$ attached to $\varphi $.

Let $\pi \in \Pi $ and assume $\pi \neq 0$. Then:

\begin{lemma}
(i) $spl_{\pi ^{\prime }}$ yields character data for nonzero $\pi ^{\prime }$
in $\Pi $ if and only if%
\begin{equation*}
\eta _{\pi ^{\prime }}=Int(g_{\ast })\circ \eta _{\pi }\circ Int(g),
\end{equation*}%
where $g\in G(\mathbb{R})\ $and where $g_{\ast }\in G_{sc}^{\ast }$
normalizes $T$ and is such that the restriction of $Int(g_{\ast })$ to $%
M^{\ast }$ is defined over $\mathbb{R}$.

(ii) Further, $\pi ^{\prime }=\pi $ if and only if $\eta _{\pi ^{\prime }}$
is of the form $\eta _{\pi }\circ Int(g),$ where $g\in G(\mathbb{R}).$
\end{lemma}

\begin{proof}
The second assertion is just a restatement of a well-known property of
limits of discrete series; see \cite[Section 4]{Sh82} or \cite[Theorem 1.1(c)%
]{KZ82}. The first assertion follows from Lemma 6.1.
\end{proof}

\subsection{An elliptic invariant}

We return now to using the notation $\pi $ \textit{only for nonzero}
representations.\textit{\ }Again fix an inner form $(G,\eta )$ for which
s-elliptic $\varphi =\varphi (\mu ,\lambda )$ is relevant and consider $\pi $
in the attached packet $\Pi \ $of representations of $G(\mathbb{R}).$ Recall 
$u_{\eta }(\sigma ),$ $u_{\pi }(\sigma )\in G_{sc}^{\ast }$; we have $\eta
\sigma (\eta )^{-1}=Int(u_{\eta }(\sigma ))$ and $\eta _{\pi }\sigma (\eta
_{\pi })^{-1}=Int(u_{\pi }(\sigma )).$ Since $\eta _{\pi }$ is in the inner
class of $\eta ,$ we write $\eta _{\pi }=Int(x_{\pi })\circ \eta $ and 
\begin{equation*}
u_{\pi }(\sigma )=x_{\pi }[u_{\eta }(\sigma )]\sigma (x_{\pi })^{-1},
\end{equation*}%
where $x_{\pi }\in G_{sc}^{\ast }$.

If we may choose $u_{\eta }(\sigma ),$ and thence $u_{\pi }(\sigma ),$ to be
a cocycle then we will say that $(G,\eta )$ is of quasi-split type because $%
(G,\eta )$ then occurs as a component of an extended group of quasi-split
type. An extended group (introduced by Kottwitz) consists of several pairs $%
(G,\eta ),$ $(G^{\prime },\eta ^{\prime })$, ... and conditions on the
twists $\eta ,\eta ^{\prime }$ ensure the property (6.3) below; see \cite%
{Sh08b} for a review and examples. There is a quasi-split component if and
only the coboundaries in (6.3) are trivial. Then we say the extended group
is of quasi-split type. The quasi-split component, if it exists, is unique 
\cite{Sh08b}.

For pairs $(G,\eta ),$ $(G^{\prime },\eta ^{\prime })$ in the same extended
group, relative factors for tempered spectral transfer are defined in \cite%
{Sh08b} (the relative geometric factors were introduced by Kottwitz). When
the extended group is of quasi-split type our already chosen Whittaker data
provides a unique normalization $\Delta _{Wh}$ of the absolute transfer
factors for each component $(G,\eta )$; see \cite{KS99}. The spectral
factors $\Delta _{Wh}$ possess the strong base-point property \cite{Sh08b}.
In particular, we have the formula (6.4) for discrete series
representations. For a general extended group, the results of Kaletha \cite%
{Ka13} provide a natural normalization for the absolute factors. The
setting, and in particular the definition of extended group, is modified
with additional structure. For our purposes it is convenient to work with
the minimal extended groups of the present setting, and we will allow any
normalization of the absolute factors that possesses geometric-spectral
compatibility in the sense of \cite{Sh10, Sh08b}. The extended groups will
play a more central role when we come to finer structure on packets in \cite%
{ShII}.

We begin our definition of elliptic invariants with the case that $(G,\eta )$
is of quasi-split type. In this setting we define an absolute invariant $%
inv(\pi )$ in $H^{1}(\Gamma ,T).$ Recall that $T$ is the elliptic maximal
torus in $G^{\ast }$ specified by $spl_{Wh}$. First we have by Lemma 6.3.1
that $u_{\pi }(\sigma )$ lies in the center $Z_{M_{(sc)}^{\ast }}$ of $%
M_{(sc)}^{\ast }$ and so defines an element of $H^{1}(\Gamma
,Z_{M_{(sc)}^{\ast }}).$ It depends only on $\pi $, \textit{i.e., }only on
the\textit{\ }$G(\mathbb{R})$-conjugacy class of $spl_{\pi }$. Now $inv(\pi
) $ is defined to be the image of this class under 
\begin{equation}
H^{1}(\Gamma ,Z_{M_{(sc)}^{\ast }})\rightarrow H^{1}(\Gamma ,Z_{M^{\ast
}})\rightarrow H^{1}(\Gamma ,T)
\end{equation}%
given by the composition of the obvious map $Z_{M_{(sc)}^{\ast }}\rightarrow
Z_{M^{\ast }}$ and inclusion $Z_{M^{\ast }}\rightarrow T.$ From the diagram%
\begin{equation}
\begin{array}{ccc}
Z_{M_{(sc)}^{\ast }} & \longrightarrow & T_{sc} \\ 
\downarrow &  & \downarrow \\ 
Z_{M^{\ast }} & \longrightarrow & T%
\end{array}%
\end{equation}%
we conclude that $inv(\pi )$ lies in the image $\mathcal{E}(T)$ of $%
H^{1}(\Gamma ,T_{sc})$ in $H^{1}(\Gamma ,T)$.

We will also make use of the following.

\begin{lemma}
Suppose that $T$ is a fundamental maximal torus in a connected reductive
group $G$ over $\mathbb{R}$. Then $H^{1}(\Gamma ,Z_{G})\rightarrow
H^{1}(\Gamma ,T)$ is injective.
\end{lemma}

\begin{proof}
Because $T$ is fundamental, both $T_{sc}(\mathbb{R})$, $T_{ad}(\mathbb{R})$
are connected and hence $T_{sc}(\mathbb{R})\rightarrow T_{ad}(\mathbb{R})$
is surjective; see the proof of Lemma 2.2 for references. A calculation then
shows that the kernel of $H^{1}(\Gamma ,Z_{G})\rightarrow H^{1}(\Gamma ,T)$
is trivial.
\end{proof}

In general, we define a relative invariant $inv(\pi ,\pi ^{\prime })$ when $%
(G,\eta ),$ $(G^{\prime },\eta ^{\prime })$ are components of the same
extended group and $\pi ,\pi ^{\prime }$ belong to packets $\Pi ,\Pi
^{\prime }$ for $G(\mathbb{R}),G^{\prime }(\mathbb{R})$ attached to relevant 
$s$-elliptic parameters $\varphi =\varphi (\mu ,\lambda ),\varphi ^{\prime
}=\varphi (\mu ^{\prime },\lambda ^{\prime })$ respectively. We follow the
method introduced in \cite{LS87, LS90}; see also \cite[Section 4.4]{KS99}.
First, recall that%
\begin{equation}
\partial u_{\pi }=\partial u_{\eta }=\partial u_{\eta ^{\prime }}=\partial
u_{\pi ^{\prime }}
\end{equation}%
takes values in $Z_{G_{sc}^{\ast }}$ as subgroup of $T$, the elliptic
maximal torus in $G^{\ast }$ specified by $spl_{Wh}$. As in the references,
set%
\begin{equation*}
U_{sc}=U(T_{sc},T_{sc})=T_{sc}\times T_{sc}\diagup \{(z^{-1},z):z\in
Z_{G_{sc}^{\ast }}\}
\end{equation*}%
and%
\begin{equation*}
U=U(T,T)=T\times T\diagup \{(z^{-1},z):z\in Z_{G^{\ast }}\}.
\end{equation*}%
Also consider%
\begin{equation*}
U(Z_{M_{(sc)}^{\ast }})=Z_{M_{(sc)}^{\ast }}\times Z_{M_{(sc)}^{\ast
}}\diagup \{(z^{-1},z):z\in Z_{G_{sc}^{\ast }}\}
\end{equation*}%
and%
\begin{equation*}
U(Z_{M^{\ast }})=Z_{M^{\ast }}\times Z_{M^{\ast }}\diagup \{(z^{-1},z):z\in
Z_{G^{\ast }}\}.
\end{equation*}%
Then we replace (6.1) above by%
\begin{equation*}
H^{1}(\Gamma ,U(Z_{M_{(sc)}^{\ast }}))\rightarrow H^{1}(\Gamma ,U(Z_{M^{\ast
}}))\rightarrow H^{1}(\Gamma ,U),
\end{equation*}%
and use the cocycle that is the image in $U(Z_{M_{(sc)}^{\ast }})$ of the
pair $(u_{\pi }(\sigma )^{-1},u_{\pi ^{\prime }}(\sigma ))$ in $%
Z_{M_{(sc)}^{\ast }}\times Z_{M_{(sc)}^{\ast }}$. Then we obtain $inv(\pi
,\pi ^{\prime })$ in the image of $H^{1}(\Gamma ,U_{sc})$ in $H^{1}(\Gamma
,U)$. If the extended group is of quasi-split type then $inv(\pi ,\pi
^{\prime })$ is the image of $(inv(\pi )^{-1},inv(\pi ^{\prime }))$ under
the evident homomorphism $H^{1}(\Gamma ,T)\times H^{1}(\Gamma ,T)\rightarrow
H^{1}(\Gamma ,U).$

\subsection{Application to endoscopic transfer}

We consider standard endoscopic transfer in the cuspidal elliptic setting.
If $\varphi _{1}$ is an $s$-elliptic Langlands parameter for $H_{1}$ then
its transport $\varphi $ to $G$ is also $s$-elliptic; see \cite[Section 11]%
{Sh08a} for how to transport attached character data from the $z$-extension $%
H_{1}$. If $\varphi _{1}$ is elliptic then $\varphi $ may, of course, fail
to be elliptic, but $\varphi $ is at least $s$-elliptic and moreover the
associated triples of nonzero character data are nondegenerate. It is then
straightforward to define spectral transfer factors via the Zuckerman
translation principle; this is recalled in Section 14 of \cite{Sh10}.

For a general $s$-elliptic related pair $(\varphi _{1},\varphi )$, however,
neither side of the spectral transfer statement has support on the regular
elliptic set. Then we have defined the associated transfer factors via an $L$%
-group version of the Knapp-Zuckerman (nondegenerate) decomposition of
unitary principal series; see \cite{Sh10, Sh08b}. In that form, the factors
display desired structure on the packet; see \cite[Section 11]{Sh08b}.

Our purpose now is to note a simpler description, based on the elliptic
invariant of Section 6.4, of the transfer factors for a general $s$-elliptic
related pair. Whittaker data for $G^{\ast }$ has been fixed. First we
transport a $\Gamma _{T}$-invariant $s_{T}$ in the complex dual $T^{\vee }$
of $T$ to an element $s$ of the maximal torus $\mathcal{T}$ in $G^{\vee }\ $%
via the method of Section 6.1. To the pair $(s,\varphi )$ we attach the
endoscopic data $\mathfrak{e}(s)$ of Section 7 of \cite{Sh08b}, now writing $%
\mathfrak{e}_{z}(s)$ since it is already supplemented, as well as the
related pair of parameters $(\varphi ^{s},\varphi )$. The attached
endoscopic group will be denoted $H^{(s)}.$ When $\varphi $ is singular we
have used a different representative, say $\varphi ^{\prime }$, to display
the structure on the packet via Knapp-Zuckerman theory. The conjugacy of $%
\varphi \ $and $\varphi ^{\prime }$ under $G^{\vee }$ determines a canonical
isomorphism of the attached abelian groups $\mathbb{S}_{\varphi }$ and $%
\mathbb{S}_{\varphi ^{\prime }}$ (see Section 6.6 or 6.7). To examine the
effect of this isomorphism on transfer, see \cite[Section 2]{Sh08b} for
passage to isomorphic endoscopic data and \cite[Section 11]{Sh08b} for
related results. For present needs, the results of Sections 6.6 - 6.8 will
be sufficient.

Recall from Section 6.4 that our Whittaker data also determine absolute
transfer factors $\Delta _{Wh}$ for any inner form $(G,\eta )$ of
quasi-split type. We use $\pi ^{s}$ to denote a representation in the packet
for $H^{(s)}(\mathbb{R})$ attached to $\varphi ^{s};$ the choice within the
packet will not matter. Finally, $\left\langle -,-\right\rangle _{tn}$ will
be the Tate-Nakayama pairing between $H^{1}(\Gamma ,T)$ and $\pi
_{0}((T^{\vee })^{\Gamma }),$ and the image of $s_{T}$ in $\pi _{0}((T^{\vee
})^{\Gamma })$ will again be written $s_{T}.$

\begin{lemma}
Suppose $(G,\eta )$ is of quasi-split type and $\varphi (\mu ,\lambda )$ is
an s-elliptic parameter relevant to $(G,\eta ).$ Then%
\begin{equation}
\Delta _{Wh}(\pi _{s},\pi )=\left\langle inv(\pi ),s_{T}\right\rangle _{tn}
\end{equation}%
for each limit of discrete series representation $\pi $ of $G(\mathbb{R})$
attached to $\varphi (\mu ,\lambda ).$
\end{lemma}

\begin{proof}
Although not necessary, we reduce easily to the case that the derived group
of $G$ is simply-connected as this allows us to refer directly to the first
half of the argument for the proof of Theorem 11.5 in \cite{Sh08b}. There a
totally degenerate parameter as in Section 6.6 was needed; now we apply the
coherent continuation argument to any relevant $s$-elliptic $\varphi ,$ so
obtaining the transfer identity in the middle of p. 400. The formula (6.4)
then follows from its truth in the case $\varphi $ is elliptic.
\end{proof}

Returning to the notation of Section 6.4, recall from \cite{LS90} that we
identify $(U_{sc})^{\vee }$ with $\mathcal{T}_{sc}\times \mathcal{T}%
_{sc}\diagup \{(z,z):z\in Z_{G_{sc}^{\vee }}\}\ $and define $s_{U}$ as
there. Now $\left\langle \_,\_\right\rangle _{tn}$ will denote the
Tate-Nakayama pairing for $U$. The following requires a minor variant of the
last proof but it will be convenient to have a separate statement.

\begin{lemma}
Suppose $(G,\eta ),(G^{\prime },\eta ^{\prime })$ are components of an
extended group and that $\Delta $ is an absolute transfer factor for the
extended group. Then%
\begin{equation*}
\Delta (\pi _{s},\pi )\diagup \Delta (\pi _{s},\pi ^{\prime })=\left\langle
inv(\pi ,\pi ^{\prime }),s_{U}\right\rangle _{tn}
\end{equation*}%
for all limits of discrete series representations $\pi ,\pi ^{\prime }$ of $%
G(\mathbb{R}),G^{\prime }(\mathbb{R}).$
\end{lemma}

\begin{remark}
We use transfer factors $\Delta $ for the classic version of the Langlands
correspondence for real groups. See \cite{Sh14} for (simple) transition to
the alternate factors $\Delta _{D}$.
\end{remark}

\subsection{Example: totally degenerate parameters}

First, the notion of \textit{totally degenerate} character data of Carayol
and Knapp \cite{CK07} extends to reductive groups, and since our data are
generated by a Langlands parameter we consider the parameter instead. We
call an $s$-elliptic parameter $\varphi =\varphi (\mu ,\lambda )$ \textit{%
totally degenerate} if $\left\langle \mu ,\alpha ^{\vee }\right\rangle =0$
for all roots $\alpha ^{\vee }$ of $\mathcal{T}$ in $G^{\vee }$; see \cite[%
Section 12]{Sh08b}.

This definition implies that a totally degenerate parameter is relevant to $%
(G,\eta ),$ \textit{i.e.}, there is a packet for $G(\mathbb{R})$ attached to
the parameter, if and only if $G$ is quasi-split. Thus we may as well assume
that $G=G^{\ast }\ $and $\eta =id$.

Further, an examination of the congruences for $\mu ,\lambda $ shows that
totally degenerate parameters exist only for certain cuspidal quasi-split
groups. For example, if $G_{der}$ is simply-connected then such $(\mu
,\lambda )$ do exist:\ they are the data for an extension of the rational
character $\iota $ on $T_{der},$ regarded as character on $T_{der}(\mathbb{R}%
)$, to a continuous quasicharacter on $T(\mathbb{R})$; see \cite{Sh08b}.
Then an elliptic endoscopic group for $G$ also has totally degenerate
parameters \cite{Sh08b}. So also does each cuspidal standard or $c$-Levi
group $X$ for $G$ because $X_{der}$ is also simply-connected. A $z$%
-extension $G_{z}$ of any cuspidal quasi-split group $G$ has totally
degenerate characters for the same reason.

Suppose now that $\varphi =\varphi (\mu ,\lambda )$ is a totally degenerate
parameter for $G=G^{\ast }.$ The congruences for $\mu ,\lambda $ further
show that the parameter $\boldsymbol{\varphi }$ is uniquely determined by $G$
up to multiplication by element of $H^{1}(W_{\mathbb{R}},Z_{G^{\vee }}),$
and hence that the attached packet is uniquely determined up to twisting by
a quasi-character on $G(\mathbb{R}).$

To describe the packet $\Pi $ attached to totally degenerate $\varphi $ in
terms of the elliptic character data provided by $(\mu ,\lambda )$ and the
Whittaker data, we may proceed as follows. Recall the fixed $\mathbb{R}$%
-splitting $spl^{\ast }=(B^{\ast },T^{\ast },\{X_{\alpha }\})$ for $G.$
There is another representative $\overline{\varphi }$ for $\boldsymbol{%
\varphi }$ attached to the maximally split maximal torus $T^{\ast }.$ We
obtain it by applying a sequence of dual Cartan transforms to $\varphi
=\varphi (\mu ,\lambda );$ the sequence is prescribed by a suitable set of
strongly orthogonal roots and the transforms are defined as in the proof of
Lemma 4.3.5 in \cite{Sh82}. Write $\overline{\varphi }=\varphi (\mu ,%
\overline{\lambda })$ relative to $T^{\ast }.$ These data determine an
essentially unitary minimal principal series representation for $G(\mathbb{R}%
).$ By definition of the Langlands correspondence, $\Pi $ consists of the
components of this representation. By Vogan's classification of generic
representations \cite{Vo78}, these components include generic $\pi ^{\ast }$
with attached fundamental splitting $spl_{Wh}.$ Then Lemma 6.2 shows that we
obtain the other components by applying $Int(g_{\ast })$ to $spl_{Wh},$
where $g_{\ast }\in G_{sc}$ and the automorphism $Int(g_{\ast })$ of $G$ is
defined over $\mathbb{R}$. Each such element $g_{\ast }$ determines an
element of $H^{1}(\Gamma ,Z_{sc}),$ where $Z_{sc}$ denotes the center of $%
G_{sc}$. Conversely, each element of $H^{1}(\Gamma ,Z_{sc})$ has trivial
image in $H^{1}(\Gamma ,G_{sc})$ \cite[Lemma 12.3]{Sh08b} and so determines $%
g_{\ast }$ such that $Int(g_{\ast })$ is defined over $\mathbb{R}$. Finally,
two elements of $H^{1}(\Gamma ,Z_{sc})$ determine the same component if and
only if they differ by an element of $Ker[H^{1}(\Gamma ,Z_{sc})\rightarrow
H^{1}(\Gamma ,Z)]$, where $Z$ denotes the center of $G,$ so that we have
bijections%
\begin{equation}
\Pi \leftrightarrow G_{ad}(\mathbb{R})\diagup Int(G(\mathbb{R}%
))\leftrightarrow Image[H^{1}(\Gamma ,Z_{sc})\rightarrow H^{1}(\Gamma ,Z)].
\end{equation}%
If we map the image of $\pi $ in $H^{1}(\Gamma ,Z)$ to $H^{1}(\Gamma ,T)$
under the injective $H^{1}(\Gamma ,Z)\rightarrow H^{1}(\Gamma ,T)\ $then we
recover the elliptic invariant $inv(\pi )$ defined in Section 6.4.

The group $S_{\overline{\varphi }}=Cent(\overline{\varphi }(W_{\mathbb{R}%
}),G^{\vee })$ consists of the fixed points in $G^{\vee }$ for the action of 
$\sigma \in \Gamma $ by $\overline{\sigma }=Int(\overline{\varphi }%
(w_{\sigma })).$ Thus%
\begin{equation*}
\mathbb{S}_{\overline{\varphi }}:=S_{\overline{\varphi }}\diagup \lbrack
(Z_{G^{\vee }})^{\Gamma }.S_{\overline{\varphi }}^{0}]=(G^{\vee })^{%
\overline{\Gamma }}\diagup \lbrack (Z_{G^{\vee }})^{\Gamma }.((G^{\vee })^{%
\overline{\Gamma }})^{0}].
\end{equation*}%
Notice that $\mathbb{S}_{\overline{\varphi }}$ is isomorphic to Langlands' $R
$-group $R_{\overline{\varphi }}$ for $\boldsymbol{\varphi }$ in this
setting; see \cite[Section 5.3]{Sh82}. Combining this with the pairing
obtained via nondegenerate Knapp-Zuckerman theory (see \cite{Sh08b, Sh10}),
we have that $\Pi $ determines a perfect pairing of%
\begin{equation}
Image[H^{1}(\Gamma ,Z_{sc})\rightarrow H^{1}(\Gamma ,Z)]
\end{equation}%
with%
\begin{equation*}
(G^{\vee })^{\overline{\Gamma }}\diagup (Z_{G^{\vee }})^{\Gamma }.((G^{\vee
})^{\overline{\Gamma }})^{0}.
\end{equation*}%
In particular, if $G$ is semisimple and simply-connected then our pairing
for the unique totally degenerate packet for $G(\mathbb{R})$ exhibits a
perfect pairing of $H^{1}(\Gamma ,Z)$ with $\pi _{0}[(G^{\vee })^{\overline{%
\Gamma }}]$.

\subsection{General limits: factoring parameters}

We return to general $s$-elliptic $\varphi =\varphi (\mu ,\lambda ):W_{%
\mathbb{R}}\rightarrow $ $^{L}G.$ Since the image of $\varphi $ lies in $%
\mathcal{M},$ we factor $\varphi $ through $^{L}M,$ and write $\varphi =\xi
_{M}\circ \varphi _{M},$ where $\varphi _{M}$ is the s-elliptic parameter $%
\varphi (\mu _{M},\lambda _{M})$ for $M^{\ast },$ with 
\begin{equation*}
\mu _{M}=\mu -(\iota -\iota _{M}),\lambda _{M}=\lambda .
\end{equation*}%
Clearly $\varphi _{M}$ is totally degenerate. In summary:

\begin{lemma}
An s-elliptic parameter $\varphi $ determines a well-defined totally
degenerate parameter for the $c$-Levi group attached to $\varphi .$
\end{lemma}

Turning to packets, we start with the quasi-split form $G^{\ast }$ and
generic $\pi ^{\ast }$ whose character data is the transport of $(\mu
,\lambda ,\mathcal{C})$ to $T$ provided by $spl_{Wh}.$ Our realization of $%
M_{\varphi }$ as $M^{\ast }$ in Section 6.2 determines a fundamental
splitting $spl_{Wh,M}$ and chamber $\mathcal{C}_{M}$ for $T.$ We use the
same notation for the inverse transport of this chamber to $M^{\vee }.$ The
transport by $spl_{Wh,M}$ of dual data$\mathcal{\ }(\mu _{M},\lambda _{M},%
\mathcal{C}_{M})$ attached to $\varphi _{M}$ determines a totally degenerate
limit of discrete series representation $\pi _{M}^{\ast }$ of $M^{\ast }(%
\mathbb{R}).$ By construction, $\pi _{M}^{\ast }$ is generic relative to the
Whittaker data attached to $\psi _{\mathbb{R}}$ and the $\mathbb{R}$%
-splitting $spl_{M}^{\ast }=(\overline{B}_{M},\overline{T}_{M},\{X_{\alpha
}\})$ for $M^{\ast }$ from Section 6.1.

Consider now general $(G,\eta )$ for which $\varphi (\mu ,\lambda )$ is
relevant. Let $\Pi $ be the attached packet and consider $\pi \in \Pi .$
Recall that $\eta _{\pi }:M_{\pi }\rightarrow M^{\ast }$ is an $\mathbb{R}$%
-isomorphism. Define the representation $\pi _{M}$ of $M_{\pi }(\mathbb{R})$
by transport: $\pi _{M}=$ $\pi _{M}^{\ast }\circ \eta _{\pi }.$ Then $\pi
_{M}$ lies in the totally degenerate packet $\Pi _{M_{\pi }}$ of
representations of $M_{\pi }(\mathbb{R})$ attached to $\varphi _{M}.$

We return to the elliptic invariants of (6.4) and consider the subgroup%
\begin{equation*}
Image(H^{1}(\Gamma ,Z_{M_{sc}^{\ast }})\rightarrow H^{1}(\Gamma ,T))
\end{equation*}%
of%
\begin{equation*}
Image(H^{1}(\Gamma ,Z_{M_{(sc)}^{\ast }})\rightarrow H^{1}(\Gamma ,T)).
\end{equation*}%
From (6.5) and Lemma 6.3, we have an isomorphism of this subgroup with $%
M_{ad}(\mathbb{R})\diagup Int(M(\mathbb{R})).$

On the other hand, notice that $S_{\varphi }=Cent(\varphi (W_{\mathbb{R}%
}),G^{\vee })$ is contained in $M^{\vee }$ and hence 
\begin{equation*}
S_{\varphi }=S_{\varphi _{M}}=Cent(\varphi _{M}(W_{\mathbb{R}}),M^{\vee })
\end{equation*}%
which is the group of fixed points of $M^{\vee }$ under either of the
automorphisms $Int(\varphi (w_{\sigma })),$ $Int(\varphi _{M}(w_{\sigma }))$%
; we arranged in Section 5.4 that these automorphisms act the same way on $%
M^{\vee }.$ Again write $\mathbb{S}_{\varphi }$ for the quotient $S_{\varphi
}\diagup (Z_{G^{\vee }})^{\Gamma }S_{\varphi }^{0}$. Then 
\begin{equation*}
\mathbb{S}_{\varphi _{M}}=S_{\varphi }\diagup (Z_{M^{\vee }})^{\Gamma
}S_{\varphi }^{0}
\end{equation*}%
and since $(Z_{M^{\vee }})^{\Gamma }\cap ((M^{\vee })^{\Gamma })^{0}$ is
contained in $(Z_{G^{\vee }})^{\Gamma }$ we have an exact sequence%
\begin{equation*}
1\rightarrow (Z_{M^{\vee }})^{\Gamma }\diagup (Z_{G^{\vee }})^{\Gamma
}\rightarrow \mathbb{S}_{\varphi }\rightarrow \mathbb{S}_{\varphi
_{M}}\rightarrow 1.
\end{equation*}

\subsection{General limits: companion standard Levi group}

We continue with the packet $\Pi $\ of the last section. It consists of the
components of several essentially tempered principal series representations
of $G(\mathbb{R}).$ To describe them we return to the representative $%
\overline{\varphi }_{M}=\varphi (\mu _{M},\overline{\lambda }_{M})$ for $%
\boldsymbol{\varphi }_{M}$ in Section 6.6 and set $\overline{\varphi }=\xi
_{M}\circ \overline{\varphi }_{M}.$ Then $\overline{\varphi }$ also
represents $\boldsymbol{\varphi }\mathbf{.}$

We may replace $spl_{Wh}$ by a $G^{\ast }(\mathbb{R})$-conjugate and then $%
M^{\ast }$ by its conjugate relative to the same element to arrange that the
maximal torus $\overline{T}_{M}$ in $M^{\ast }$ provided by $spl_{M}^{\ast }$
is a standard maximal torus in $G^{\ast }$. We then drop the subscript $M$
in notation for this torus. Here by \textit{standard} we mean that the
maximal split torus $\overline{S}$ in $\overline{T}$ is contained in $%
T^{\ast }$ provided by $spl^{\ast }=(spl^{\vee })^{\vee }$. Let $\overline{M}
$ be the standard Levi group $Cent(\overline{S},G^{\ast })$. Then $^{L}%
\overline{M}$ will denote the dual standard Levi group in $^{L}G,$ naturally
embedded by inclusion.

\begin{lemma}
The image of $\overline{\varphi }$ lies in $^{L}\overline{M}$ and defines an
elliptic parameter for $\overline{M}.$
\end{lemma}

\begin{proof}
We return to the notation of Section 5.4. We have arranged that $\sigma _{%
\overline{T}}=\sigma _{M}$ on $\mathcal{T}$. Then the element $n_{\overline{M%
}}\times w_{\sigma }$ in $^{L}G$ coincides with $\xi _{M}(w_{\sigma })$ up
to an element of $\mathcal{T}$ $\cap $ $G_{der}^{\vee }.$ It follows that $%
\overline{\varphi }(w_{\sigma })$ $\in $ $^{L}\overline{M}$ and then that $%
\overline{\varphi }(W_{\mathbb{R}})$ $\subseteq $ $^{L}\overline{M}$. Since $%
\sigma _{M}\alpha ^{\vee }=\sigma _{\overline{T}}\alpha ^{\vee }=-\alpha
^{\vee }$ for each root $\alpha ^{\vee }$ of $\mathcal{T}$ in $\overline{M}%
^{\vee }$, it is clear that $\overline{\varphi }$ is $s$-elliptic as
Langlands parameter for $\overline{M}$. If we write $\overline{\varphi }%
=\varphi (\mu ,\overline{\lambda })$ relative to $\overline{M},$ then $\mu $
is $\overline{M}$-regular for otherwise $\overline{T}$ would have an
imaginary root in $M^{\ast }.$ Thus the lemma is proved.
\end{proof}

We continue with $\overline{\varphi }=\varphi (\mu ,\overline{\lambda })$
and consider the quasi-split form $G^{\ast }$. The Whittaker data $Wh$ for $%
G^{\ast }$ determines, by restriction, Whittaker data $Wh_{\overline{M}}$
for $\overline{M}.$ We choose a corresponding fundamental splitting $%
spl_{Wh_{\overline{M}}}=(B_{\overline{M}},\overline{T},\{X_{\alpha }\})$ of
Whittaker type for $\overline{M}$, and then transport $(\mu ,\overline{%
\lambda })$ to discrete series character data on $\overline{T}$. Via unitary
parabolic induction, each discrete series representation in the packet for $%
\overline{M}(\mathbb{R})$ attached to $\overline{\varphi }$ determines an
essentially tempered principal series representation of $G^{\ast }(\mathbb{R}%
).$ Then $\Pi $ consists of the irreducible components of all principal
series representations so obtained. Consider next an inner form $(G,\eta )$
for which $\varphi $ is relevant. Recall that for $\pi \in \Pi ,$ $\eta
_{\pi }$ transports elliptic character data for $\pi $ to that for $\pi
^{\ast }.$ By (6.5) and Lemma 6.2 we may choose $\pi $ so that $\pi _{M}=$ $%
\pi _{M}^{\ast }\circ \eta _{\pi }$ is isomorphic to $\pi _{M}^{\ast }$. We
then adjust our discussion for $G^{\ast }$ to describe the packet for $G(%
\mathbb{R})$; we will not need details here.

From definitions (recalled in \cite[Section 5.3]{Sh82}) it is clear that
Langlands' version of the $R$-group is unchanged by passage from $^{L}M$ to $%
^{L}G$:%
\begin{equation*}
R_{\overline{\varphi }}=R_{\overline{\varphi }_{M}}.
\end{equation*}%
Also there is a surjective homomorphism $\mathbb{S}_{\overline{\varphi }%
}\rightarrow R_{\overline{\varphi }}\ $with kernel that may be identified
with the dual of $\mathcal{E}(\overline{T})$ (see \cite[Sections 5.3, 5.4]%
{Sh82}). Because $\varphi _{M}$ is totally degenerate we have that $\mathbb{S%
}_{\overline{\varphi }_{M}}\rightarrow R_{\overline{\varphi }_{M}}=R_{%
\overline{\varphi }}$ is an isomorphism. Then by the discussion around
(6.6.3) we have a perfect pairing of $R_{\overline{\varphi }}$ with%
\begin{equation*}
Image(H^{1}(\Gamma ,Z_{M_{sc}^{\ast }})\rightarrow H^{1}(\Gamma ,T))\simeq
M_{ad}(\mathbb{R})\diagup Int(M(\mathbb{R})).
\end{equation*}

\section{\textbf{Packets and parameters II}}

\subsection{Data for elliptic $u$-regular parameters}

Suppose that $\psi =(\varphi ,\rho )$ is an elliptic $u$-regular Arthur
parameter. Continuing from (5.5), we may assume that $\varphi $ takes the
following form:%
\begin{equation*}
\varphi (z)=z^{\mu }\overline{z}^{\sigma _{M}\mu }\times z
\end{equation*}%
for $z\in \mathbb{C}^{\times },$ and$\varphi (w_{\sigma })=e^{2\pi i\lambda
}.\xi _{M}(w_{\sigma }).$Here 
\begin{equation*}
\mu ,\lambda \in X_{\ast }(\mathcal{T})\otimes \mathbb{C}
\end{equation*}
and%
\begin{equation}
\left\langle \mu ,\alpha ^{\vee }\right\rangle =0,\text{ }\left\langle
\lambda ,\alpha ^{\vee }\right\rangle \in \mathbb{Z}
\end{equation}%
for all roots $\alpha ^{\vee }$ of $\mathcal{T}$ in $M^{\vee }.$ The element 
$\mu $ is uniquely determined by the $\mathcal{T}$-conjugacy class of the
representative $\varphi $, and $\lambda $ is determined uniquely modulo 
\begin{equation*}
\mathcal{K}_{M}=X_{\ast }(\mathcal{T})+\{\nu \in X_{\ast }(\mathcal{T}%
)\otimes \mathbb{C}:\sigma _{M}\nu =-\nu \}.
\end{equation*}%
We will use the notation $\varphi =\varphi \lbrack \mu ,\lambda ].$ Notice
that in the case $M_{\psi }=T$, where $\varphi $ is elliptic, we return to
the pair $(\mu ,\lambda )$ from Section 5.6.

From our construction of $\xi _{M}$ and the equation $\varphi (w_{\sigma
})^{2}=\varphi (-1)$ we have immediately the following congruence:%
\begin{equation}
\frac{1}{2}(\mu -\sigma _{M}\mu )-(\iota -\iota _{M})\equiv \lambda +\sigma
_{M}\lambda \text{ }\func{mod}X_{\ast }(\mathcal{T}).
\end{equation}%
The properties (7.1) allow us to replace $\sigma _{M}$ by $\sigma _{T}$ in
(7.2) and then to rewrite the congruence as%
\begin{equation}
\frac{1}{2}[\mu +\iota _{M}-\sigma _{T}(\mu +\iota _{M})]-\iota \equiv
\lambda +\sigma _{T}\lambda \text{ }\func{mod}X_{\ast }(\mathcal{T}).
\end{equation}

For the second component $\rho $ of $\psi $ we turn to Section 5.5 and the $u
$-regularity property. With first component $\varphi $ prescribed as above
we may assume that $\rho :$ $SL(2,\mathbb{C})\rightarrow M^{\vee }$ is in
standard form with cocharacter $2\iota _{M}$. Then%
\begin{equation*}
\rho (diag(\left\vert w\right\vert ^{1/2},\left\vert w\right\vert
^{-1/2}))=(z\overline{z})^{\iota _{M}},w\in W_{\mathbb{R}},
\end{equation*}%
where $w=z$ or $zw_{\sigma },$ $z\in \mathbb{C}^{\times }$, as in \cite{Ar89}%
. We write $\rho =\rho (\iota _{M}).$

We observe that (7.3) implies that $\mu +\iota _{M}\in X_{\ast }(\mathcal{T}%
)\otimes \mathbb{C}$ is integral, \textit{i.e.,} 
\begin{equation}
\left\langle \mu +\iota _{M},\alpha ^{\vee }\right\rangle \in \mathbb{Z},
\end{equation}%
for all roots of $\mathcal{T}$ in $G^{\vee }$.

\begin{remark}
$\mu $ is at least half-integral; $\mu $ is integral if the derived group of 
$G$ is simply-connected since $\iota _{M}$ is integral in that case.
\end{remark}

Recall that $\mathcal{B}$ denotes the Borel subgroup that is part of $%
spl^{\vee }.$

\begin{lemma}
Let $\mathbf{\psi }$ be an elliptic $u$-regular Arthur parameter. Then there
exists a representative $\psi =(\varphi ,\rho )$ for $\mathbf{\psi }$, where 
$\varphi =\varphi \lbrack \mu ,\lambda ]$ and $\rho =\rho (\iota _{M}),$
with both $\mu $ and $\mu +\iota _{M}$ $\mathcal{B}$-dominant$.$
\end{lemma}

\begin{proof}
First, we observe that it is sufficient to arrange that $\mu +\iota _{M}$ is
dominant. Since $M^{\vee }$ is Levi we have that $\left\langle \iota
_{M},\alpha ^{\vee }\right\rangle \leq 0$ for each $\mathcal{B}$-simple $%
\alpha ^{\vee }$ that is not a root of $\mathcal{T}$ in $M^{\vee }.$ Then
dominance of $\mu +\iota _{M}$ implies $\left\langle \mu ,\alpha ^{\vee
}\right\rangle \geq 0$ for all such $\alpha ^{\vee }$ and so by (7.1.3), $%
\mu $ is dominant.

Second, suppose we pick $\varphi =\varphi \lbrack \mu ,\lambda ],$ $\rho
=\rho (\iota _{M})$ as in the paragraphs above. There is $\omega $ in the
Weyl group of $\mathcal{T}$ in $G^{\vee }$ such that $\omega (\mu +\iota
_{M})$ is $\mathcal{B}$-dominant. Let $x\in G^{\vee }$ normalize $\mathcal{T}
$ and act on $\mathcal{T}$ as $\omega $. Set $M_{\omega }^{\vee }=xM^{\vee
}x^{-1}$ and $\psi _{x}=Int(x)\circ \psi .$ If $\alpha ^{\vee }$ is a $%
\mathcal{B}$-positive root of in $\mathcal{T}$ in $M_{\omega }^{\vee }$ then 
$\omega ^{-1}\alpha ^{\vee }$ is a root in $M^{\vee }$ and so 
\begin{equation*}
\left\langle \iota _{M},\omega ^{-1}\alpha ^{\vee }\right\rangle
=\left\langle \mu +\iota _{M},\omega ^{-1}\alpha ^{\vee }\right\rangle
=\left\langle \omega (\mu +\iota _{M}),\alpha ^{\vee }\right\rangle \geq 0.
\end{equation*}%
Then $\omega ^{-1}\alpha ^{\vee }$ must be $\mathcal{B}$-positive. It now
follows that $\iota _{M_{\omega }}=\omega \iota _{M}$. Then after
multiplying $x$ by an element of $\mathcal{T}$ we may replace $\psi $ by $%
\psi _{x}$ in our constructions to complete the proof.
\end{proof}

\begin{flushleft}
From now on we choose representative $\psi =(\varphi ,\rho )$ as in Lemma
7.3.
\end{flushleft}

From (7.3) we conclude that:

\begin{lemma}
$\mu +\iota _{M},\lambda $ are data for an s-elliptic Langlands parameter%
\begin{equation*}
\widehat{\varphi }=\varphi (\mu +\iota _{M},\lambda ).
\end{equation*}
\end{lemma}

Finally, we set 
\begin{equation}
\mu _{M}=\mu -(\iota -\iota _{M}),\text{ }\lambda _{M}=\lambda .
\end{equation}%
As one of the ingredients \cite{La89} of the Langlands correspondence for $%
M^{\ast }$, the parameter $\varphi _{M}:W_{\mathbb{R}}\rightarrow $ $%
^{L}Z_{M}$ from Section 5.5 defines a quasicharacter $\chi _{M^{\ast }}$ on $%
M^{\ast }(\mathbb{R})$. Because of (5.3) and (7.1) the restriction of $\chi
_{M^{\ast }}$ to each Cartan subgroup in $M^{\ast }(\mathbb{R})$ takes the
form 
\begin{equation}
\Lambda (\mu _{M},\lambda _{M})
\end{equation}%
in the Langlands correspondence for real tori \cite{La89}; see \cite[Section
9]{Sh81} for a discussion and \cite[Section 7]{Sh10} for notation. Further,
for an inner twist $\eta :M_{\eta }\rightarrow M^{\ast }$ we may replace $%
M^{\ast }$ by $M_{\eta }=\eta ^{-1}(M^{\ast })$. Then the new quasicharacter 
$\chi _{\eta }$ on $M_{\eta }(\mathbb{R})$ depends only on the inner class
of $\eta $.

On the other hand, $\widehat{\varphi }$ factors through the discrete series
parameter%
\begin{equation*}
\widehat{\varphi }_{M}=\varphi (\mu _{M}+\iota _{M},\lambda _{M})
\end{equation*}%
for $M^{\ast }.$

From (7.5) and (7.4) we see that $\mu _{M}$ is integral for $G^{\vee }.$
Clearly:

\begin{lemma}
(i) $\mu +\iota _{M}$ is regular, i.e., $\widehat{\varphi }$ is elliptic, if
and only if $\mu _{M}$ is $\mathcal{B}$-dominant.

(ii) $\mu +\iota _{M}$ is singular if and only if $\left\langle \mu
_{M},\alpha ^{\vee }\right\rangle =-1$ for some $\mathcal{B}$-simple root $%
\alpha ^{\vee }$ of $\mathcal{T}$.
\end{lemma}

Assume that $\mu +\iota _{M}$ is regular and $G$ has anisotropic center;
this is the setting of \cite[Section 5]{Ar89}, \cite[Section 9]{Ko90}. Here
we recover the same parameters, but now with data for use in canonical
transfer factors; see, for example, Section 8.2. Our parameter $\widehat{%
\varphi }$ coincides with the discrete series parameter constructed slightly
differently in \cite[Section 9]{Ko90}.

\subsection{Character data and elliptic $u$-regular parameters}

We combine the setting of Section 7.1 with that of Section 6.1. Thus $%
G^{\ast }$ is cuspidal, and we have fixed Whittaker data for $G^{\ast }$
together with an aligned fundamental splitting $spl_{Wh}=(B_{Wh},T,\{X_{%
\alpha }\})$ for $G^{\ast }$. We now transport the data $(\mu +\iota
_{M},\lambda ,\mathcal{C})$ for $\mathcal{T}\subseteq G^{\vee }$ of Section
7.1 to data for $T\subseteq G^{\ast },$ by the means described in Section
6.1. Recall that $M^{\ast }$ is the subgroup of $G^{\ast }$ generated by $T$
and the root vectors $\{X_{\alpha }\},$ for $\alpha ^{\vee }$ a simple root
of $\mathcal{T}$ in $M^{\vee }\cap \mathcal{B}$. We use the same notation
for the transported data, except that now we write $\iota _{M^{\ast }}$ for
the transport of $\iota _{M},$ \textit{i.e.}, for one-half the sum of the
roots of $T$ in $B_{Wh}\cap M^{\ast }$.

\subsection{Elliptic $u$-regular data: attached $s$-elliptic packet}

We start with the case that $\mu +\iota _{M}$ is regular. Consider an inner
form $(G,\eta ).$ Replacing $\eta $ by a member of its inner class if
necessary, we assume that the transport $spl_{\eta }=\eta ^{-1}(spl_{Wh})$
of $spl_{Wh}$ to $G$ is a fundamental splitting. As in Section 6.2, to each $%
G(\mathbb{R})$-conjugacy class of fundamental splittings for $G$ is attached
a discrete series representation $\widehat{\pi }$ of $G(\mathbb{R})$ in the
packet $\widehat{\Pi }_{G}$ for $\widehat{\varphi },$ and conversely. Again
write $spl_{\widehat{\pi }}$ for a representative of this conjugacy class
and $\eta _{\widehat{\pi }}=Int(x_{\widehat{\pi }})\circ \eta $ for the
inner twist carrying $spl_{\widehat{\pi }}$ to $spl_{Wh}.$ Set $M_{\widehat{%
\pi }}=\eta _{\widehat{\pi }}^{-1}(M^{\ast }).$ By definition, $\eta _{%
\widehat{\pi }}$ transports character data $(\mu _{\widehat{\pi }}+\iota
_{M_{\widehat{\pi }}},\lambda _{\widehat{\pi }},\mathcal{C}_{\widehat{\pi }})
$ for $\widehat{\pi }$ to the data $(\mu +\iota _{M^{\ast }},\lambda ,%
\mathcal{C})$ for the elliptic torus $T$ in $G^{\ast }$ that is part of $%
spl_{Wh}$. Recall that the latter triple serves as character data for the $Wh
$-generic discrete series representation of $G^{\ast }(\mathbb{R})$ in the
packet $\widehat{\Pi }_{G^{\ast }}$ attached to $\widehat{\varphi }.$

Now allow $\mu +\iota _{M}$ to be singular. Then we assume that $\widehat{%
\varphi }=\varphi (\mu +\iota _{M},\lambda )$ is relevant to $G$ so that $%
\widehat{\Pi }_{G}$ is nonempty. As in Section 6.2, there is attached to $%
\widehat{\varphi }$ a $c$-Levi group which we will call $E^{\ast }.$ Notice
that $E^{\ast }\cap M^{\ast }=T.$ Each $G(\mathbb{R})$-conjugacy class of
fundamental splittings of $G$ again has a representative $spl_{\widehat{\pi }%
}$, but now $\widehat{\pi }$ (or, more precisely, the attached distribution
character) may be zero. We obtain precisely the members $\widehat{\pi }$ of $%
\widehat{\Pi }_{G}$ by requiring that $\eta _{\widehat{\pi }}:E_{\widehat{%
\pi }}\rightarrow E^{\ast }$ be defined over $\mathbb{R}$ (Lemma 6.1).

\subsection{Elliptic $u$-regular data: attached Arthur packet}

For the rest of Parts 7 and 8 we will limit our attention to the case that $%
\mu +\iota _{M}$ is regular as we will need it to structure our arguments
for the singular case. For convenience we could also require the center of $G
$ to be anisotropic, but the general case requires no extra notation and so
we will at least write it here. Finally there is the matter of how we treat $%
z$-extensions. We will continue to use the construction needed for the
twisted case (see Section 3.1) but defer checking that the Adams-Johnson
results may be extended in this manner until we come to the general twisted
case.

Consider an inner form $(G,\eta ),\ $where $spl_{\eta }=(B_{\eta },T_{\eta
},\{X_{\alpha }\})$ is fundamental and $\eta $ carries $spl_{\eta }$ to $%
spl_{Wh}.$ We may fix a Cartan involution $c$ on $G$ of the form $%
Int(t_{\eta }),$ where $t_{\eta }\in T_{\eta }(\mathbb{R})$ and $(t_{\eta
})^{2}$ is central in $G$. Then $B_{\eta },$ $M_{\eta }$ together generate a 
$c$-stable parabolic subgroup $P_{\eta }$ of $G$ with $M_{\eta }$ as Levi
subgroup defined over $\mathbb{R}$. We have the quasi-character $\chi =\chi
_{\eta }$ on $M_{\eta }(\mathbb{R})$ described in Section 7.1. Because of
(7.6), it is clear that $\chi _{\eta }$ is unitary modulo the center of $G(%
\mathbb{R})$. As usual, we will identify a representation with its
(appropriate) isomorphism class. Define $\pi (\eta )$ to be the irreducible
essentially unitary representation of $G(\mathbb{R})$ attached to $\chi
_{\eta }$ by the method of \cite[Theorems 1.2, 1.3]{Vo84}; see Lemma 2.10 of 
\cite{AJ87}.

Suppose we replace $\eta $ by $\eta ^{\dag }$ within its inner class and
that $\eta ^{\dag }$ also carries a fundamental splitting of $G$ to $%
spl_{Wh}.$ It is convenient to write $\eta ^{\dag }$ in the form%
\begin{equation*}
\eta ^{\dag }=Int(m^{\ast })\circ \eta \circ Int(g),
\end{equation*}%
where $m^{\ast }\in M_{sc}^{\ast }$ and $g\in G_{sc}.$ Let $m=\eta
_{sc}^{-1}(m^{\ast })$. Then we insist also that $Int(m)$ transports $%
spl_{\eta }$ to another fundamental splitting of $G$. Since we are concerned
with splittings only up to $G(\mathbb{R})$-conjugacy there is no harm in
considering only those $\eta ^{\dag }$ for which $T_{\eta ^{\dag }}=T_{\eta }
$, and requiring that both $Int(m)$ and $Int(g)$ preserve $T_{\eta }.$

We define $\pi (\eta ^{\dag })$ by replacing $B_{\eta },M_{\eta },\chi
_{\eta }$ from the definition of $\pi (\eta )$ with $B_{\eta ^{\dag
}},M_{\eta ^{\dag }},\chi _{\eta ^{\dag }}$. Then:

\begin{lemma}
(i) $\pi (\eta ^{\dag })$ lies in same Arthur packet $\Pi _{G\text{ }}$
prescribed by Adams-Johnson (\textit{enlarged packet} in their terminology)
as $\pi (\eta )$, and all members of the packet are so obtained.

(ii) $\pi (\eta ^{\dag })=\pi (\eta )$ if and only if $\eta ^{\dag }$ is of
the form $Int(m^{\ast })\circ \eta \circ Int(g),$ where $m^{\ast }\in
M_{sc}^{\ast }$ and $g\in G(\mathbb{R}).$
\end{lemma}

Let $\omega _{M},$ $\omega _{G}$ be the elements of the complex Weyl group $%
\Omega (G,T_{\eta })$ of $T_{\eta }$ in $G$ defined by the restrictions of $%
Int(m),$ $Int(g)$ to $T_{\eta }$. Then (ii) says $\pi (\eta ^{\dag })=\pi
(\eta )$ if and only if we may arrange that $\omega _{G}$ lies in the
subgroup $\Omega _{\mathbb{R}}(G,T_{\eta })$ of $\Omega (G,T_{\eta })$
consisting of those elements that are realized in $G(\mathbb{R}).$

\begin{proof}
To compare explicitly with Lemma 2.10 of \cite{AJ87}, first note that since
we do not assume that the center of $G$ is anisotropic our elliptic data
have an extra component, namely $\lambda $ as above. The "$\lambda ,\rho $"
of \cite{AJ87} are our $\mu _{M},\iota $. Note that (7.5) says that%
\begin{equation*}
\mu _{M}+\iota =\mu +\iota _{M}.
\end{equation*}%
We further have the alternative short definition after Remark 7.4 for the
quasicharacter $\chi _{\eta },$ but it is clear from calculations of Section
2 of \cite{AJ87} (or see \cite[Lemma 9.2]{Sh79b}, \cite[Section 9]{Sh81})
that we obtain the same character when we require the center of $G$ to be
anisotropic. The claim (i) now follows. More accurately, we have adapted the
definitions of Adams-Johnson to the case where there is no restriction on
the center for $G$ while retaining (i) of their Lemma 2.10. Because $\mu
+\iota _{M}$ is regular, the claim (ii) follows easily from the character
formulas we will recall in Section 8.3.
\end{proof}

\section{\textbf{Standard factors for elliptic }$u$\textbf{-regular packets}}

Here by standard factors we mean the spectral transfer factors for standard
endoscopy. We introduce these, with a two-fold purpose, for the elliptic $u$%
-regular Arthur packets $\Pi _{G}$ of the last section. First, we will check
that the Adams-Johnson transfer can be recast in terms of these factors and
thereby made compatible with the transfer of orbital integrals using the
canonical factors of \cite{LS87}. Second, we will write the factors in a way
that allows quick generalization to the twisted setting \cite{ShII}.

\subsection{Canonical relative factor: setting}

We continue with the setting at the end of Part 7. In summary, $\psi
=(\varphi ,\rho )$ is an elliptic $u$-regular Arthur parameter with $\varphi
=\varphi \lbrack \mu ,\lambda ]$ and $\rho =\rho (\iota _{M})$ as in Lemma
7.2. We assume $\mu +\iota _{M^{\ast }}$ is regular as well as dominant.
Then $\widehat{\varphi }\ $is the attached elliptic parameter $\varphi (\mu
+\iota _{M^{\ast }},\lambda ).$

To introduce elliptic endoscopic groups as in Section 6.5, we turn to the $%
\Gamma $-invariants in the maximal torus $\mathcal{T}$ from $spl^{\vee }.$
We use the elliptic action of $\Gamma $,\textit{\ }so that $\sigma $ acts by 
$Int$ $\widehat{\varphi }(w_{\sigma }).$ Consider the elliptic SED $%
\mathfrak{e}_{z}(s)$ as in Section 6.5, using the notation $(H,\mathcal{H},s)
$ for the endoscopic data and $H^{(s)}$ for the endoscopic group. It will be
sufficient for our purposes in \cite{ShII} to consider the case that the $%
\Gamma $-invariant $s$ lies in the center $Z_{M^{\vee }}$ of $M^{\vee }$.
This is the same as requiring that $H^{\vee }:=Cent(s,G^{\vee })^{0}$
contain $M^{\vee }$. Thus we place ourselves in the setting of
Adams-Johnson; see \cite[2.16]{AJ87}.

Because $\psi $ is elliptic, the subgroup $\mathcal{M=M}_{\psi }$ of $^{L}G$
may be generated by $M^{\vee }$ and either $\varphi (W_{\mathbb{R}})$ or $%
\widehat{\varphi }(W_{\mathbb{R}})$. Thus $\mathcal{H},$ generated by $%
H^{\vee }$ and $\widehat{\varphi }(W_{\mathbb{R}}),$ contains $\mathcal{M}$.
Since the endoscopic group $H^{(s)}$ is a $z$-extension of the endoscopic
datum $H,$ we will need to thicken $\mathcal{M}$.

Recall that $(Z_{M^{\vee }})^{\Gamma }=S_{\psi }:=Cent(Image(\psi ),G^{\vee
}).$ Thus the image of $\psi $ lies in $\mathcal{M\subseteq H}$. As a
component of the SED $\mathfrak{e}_{z}(s),$ we have $\xi ^{(s)}:\mathcal{H}$ 
$\mathcal{\rightarrow }$ $^{L}H^{(s)}$ and thus an elliptic $u$-regular
Arthur parameter for $H^{(s)}$ represented by $\psi ^{(s)}=\xi ^{(s)}\circ
\psi $. Now we attach $\mathcal{M}^{(s)}$ to $\psi ^{(s)}$ in the same way
we attached $\mathcal{M}$ to $\psi .$ Then $\mathcal{M}^{(s)}$ is what we
mean by the thickened version of $\mathcal{M}$. We will thicken various
other subgroups when needed, again using the super- or subscript $(s)$ to
indicate this.

We now describe transfer factors attached to the pair $(\psi ^{(s)},\psi );$
see \cite[Section 9]{Sh10}, \cite[Sections 7, 11]{Sh08b} for the tempered
analogue. Recall $\psi =(\varphi ,\rho ).$ Then we write $\psi ^{(s)}$ as $%
(\varphi ^{(s)},\rho ).$

First, $\varphi ^{(s)}=\varphi \lbrack \mu ^{(s)},\lambda ^{(s)}]\ $and%
\begin{equation}
\mu ^{(s)}=\mu -\mu ^{\ast },\text{ }\lambda ^{(s)}=\lambda -\lambda ^{\ast
}.
\end{equation}%
The pair $(\mu ^{\ast },\lambda ^{\ast })$ is from \cite{Sh81}; it is
typically nontrivial and is critical for a well-defined transfer of orbital
integrals. Here we need its construction for general standard transfer with $%
z$-extensions; see Section 11 of \cite{Sh08a}. Also see Section 9.3 below
for a detailed contruction in the general twisted case. The formula (8.1)
follows from combining the construction with that in Lemma 7.2.

Second, the component $\rho $ of $\psi ^{(s)}$ may be written again as $\rho
(\iota _{M})$. For this we recall the splittings involved in our
constructions: we have $spl^{\vee }=(\mathcal{B},\mathcal{T},\{X_{\alpha
^{\vee }}\})$ for $G^{\vee }$ with attached $spl_{M}^{\vee }=(\mathcal{B}$ $%
\mathcal{\cap }$ $M^{\vee },\mathcal{T},\{X_{\alpha ^{\vee }}\})$ for $%
M^{\vee },$ along with $spl_{H}^{\vee }=(\mathcal{B}$ $\mathcal{\cap }$ $%
H^{\vee },\mathcal{T},\{X_{\alpha ^{\vee }}\})$ for $H^{\vee }$ and
thickened $spl_{(s)}^{\vee }=(\mathcal{B}^{(s)}$ $,\mathcal{T}%
^{(s)},\{X_{\alpha ^{\vee }}\})$ for $(H^{(s)})^{\vee }.$ Then $\iota _{M}$
is one-half the sum of the coroots of $\mathcal{T}$ in $\mathcal{B}$ $%
\mathcal{\cap }$ $M^{\vee }=\mathcal{B}$ $\mathcal{\cap }$ $\mathcal{M}$.
Each such coroot is naturally identified as a coroot of $\mathcal{T}^{(s)}$
in $\mathcal{B}^{(s)}$ $\mathcal{\cap }$ $\mathcal{M}^{(s)}$ and conversely,
which justifies our use of $\iota _{M}$ for $\rho $ as component of $\psi
^{(s)}$.

The $c$-Levi group $M^{(s)}$ in $H^{(s)}$ is the analogue for $\psi ^{(s)}$
of the $c$-Levi group $M^{\ast }$ in $G^{\ast }$ attached to $\psi $. There
is a $c$-Levi group $M_{H}$ in $H$ such that $M^{(s)}\rightarrow M_{H}$ is a 
$z$-extension with kernel $Z_{1}$, \textit{i.e.}, with same kernel as the $z$%
-extension $H^{(s)}\rightarrow H$ provided by the SED $\mathfrak{e}_{z}(s).$

Our next step is to define an $\mathbb{R}$-isomorphism $M_{H}\rightarrow
M^{\ast }$ uniquely up to composition with an element of $Int[M^{\ast }(%
\mathbb{R})]$, and thence a surjective homomorphism $M_{H}^{(s)}\rightarrow
M^{\ast }$ with kernel $Z_{1}.$ For this, recall that in the construction of 
$M^{\ast }$ at the beginning of (6.2) we also determined an $\mathbb{R}$%
-splitting for $M^{\ast }$ uniquely up to $M^{\ast }(\mathbb{R})$-conjugacy.
The same is then true for $M_{H}.$ There is a unique $\mathbb{R}$%
-isomorphism $M_{H}\rightarrow M^{\ast }\ $transporting the latter splitting
to the former. We may further assume the isomorphism carries chosen elliptic
maximal torus $T_{H}$ in $H$ to chosen $T$ in $G^{\ast };$ recall that each
torus is part of an appropriate fundamental splitting of Whittaker type. If $%
T^{(s)}$ is the inverse image of $T_{H}$ in $M_{H}^{(s)}$ then we have now
have a well-defined transport to $T$ of our various data attached to $%
T^{(s)} $.

\subsection{Canonical relative factor: definition}

We now define a relative transfer factor in preparation for a nontempered
supplement to Section 6.5. Thus $(G,\eta )$ is an inner form of $G^{\ast }$
such that $\eta $ transports fundamental splitting $spl_{\eta },$ of $G$ to $%
spl_{Wh}.$ Let $\pi \in \Pi _{G}\ $(Arthur packet for $G(\mathbb{R})$
attached to $\psi )\ $and let $\widehat{\pi }\in \widehat{\Pi }_{G}$
(discrete series packet for $G(\mathbb{R})$ attached to $\psi ).$ Also let $%
\pi _{s}\in \Pi _{H^{(s)}}$ and $\widehat{\pi }_{s}\in \widehat{\Pi }%
_{H^{(s)}}$ (packets for $H^{(s)}(\mathbb{R})$ attached to $\psi ^{(s)}).$
Then our first concern will be a relative factor $\Delta (\pi _{s},\pi ;$ $%
\widehat{\pi }_{s},$ $\widehat{\pi }).$

Attach the cochain $x_{\widehat{\pi }}(\sigma )\in T_{sc}$ to the discrete
series representation $\widehat{\pi }$ as in Section 6.4; recall that $\eta
_{\widehat{\pi }}=Int(x_{\widehat{\pi }})\circ \eta $ and $x_{\widehat{\pi }%
}(\sigma )=x_{\widehat{\pi }}.u_{\eta }(\sigma ).\sigma (x_{\widehat{\pi }%
})^{-1}.$ Again we write $\mathcal{E}(T)$ for the image of $H^{1}(\Gamma
,T_{sc})$ in $H^{1}(\Gamma ,T)$ under the homomorphism induced by $%
T_{sc}\rightarrow T.$ Then if $(G,\eta )$ is a component of an extended
group of quasi-split type, so that $u_{\eta }(\sigma )$ is a cocycle, we map
the class of $x_{\widehat{\pi }}(\sigma )$ in $H^{1}(\Gamma ,T_{sc})$ to $%
H^{1}(\Gamma ,T)$ to obtain the element $inv(\widehat{\pi })$ of $\mathcal{E}%
(T).$

Turning to $\pi $ in the Arthur packet for $G(\mathbb{R}),$ we pick a twist $%
\eta ^{\dag }$ such that $\pi =\pi (\eta ^{\dag })\ $as in Section 7.4. We
write $\eta ^{\dag }$ as $Int(x^{\dag })\circ \eta $ and form the cochain $%
x^{\dag }(\sigma )=x^{\dag }.u_{\eta }(\sigma ).\sigma (x^{\dag })^{-1}.$
Recall the torus $U_{sc}$ from Section 6.4. The image in $U_{sc}$ of the
cochain $(x^{\dag }(\sigma )^{-1},x_{\widehat{\pi }}(\sigma ))$ in $%
T_{sc}\times T_{sc}$ is a cocycle whose class in $H^{1}(\Gamma ,U_{sc})$ we
denote by $\mathbf{x}_{sc}(\eta ^{\dag },\widehat{\pi }).$ Then $\mathbf{x}%
(\eta ^{\dag },\widehat{\pi })$ is the image of this class in $H^{1}(\Gamma
,U).$ Recall $s_{U}$ from Section 6.4 and that in the present setting we
assume that the $\Gamma $-invariant $s$ lies in the center of $M^{\vee }.$

\begin{lemma}
$\left\langle \mathbf{x}(\eta ^{\dag },\widehat{\pi }),\text{ }%
s_{U}\right\rangle _{tn}$ depends only on $\pi ,\widehat{\pi }.$
\end{lemma}

Then we define%
\begin{equation*}
pair_{(s)}(\pi ,\widehat{\pi }):=\left\langle \mathbf{x}(\eta ^{\dag },%
\widehat{\pi }),\text{ }s_{U}\right\rangle _{tn}.
\end{equation*}

Before proving Lemma 8.1 we examine $x^{\dag }(\sigma )$ in the case that $%
(G,\eta )$ is a component of an extended group of quasi-split type. Then $%
x^{\dag }(\sigma )$ is a cocycle and so defines an element $\mathbf{x}(\eta
^{\dag })$ of $\mathcal{E}(T).$ We have $T\subseteq M^{\ast }\subseteq
G^{\ast }.$ Then $\mathcal{E}_{M^{\ast }}(T)$ is the image of $H^{1}(\Gamma
,T_{M_{sc}^{\ast }})\rightarrow H^{1}(\Gamma ,T).$ It is a subgroup of $%
\mathcal{E}(T).$

\begin{lemma}
The image of $\mathbf{x}(\eta ^{\dag })$ in $\mathcal{E}(T)\diagup \mathcal{E%
}_{M^{\ast }}(T)$ depends only on $\pi .$
\end{lemma}

\begin{proof}
There is no harm in replacing $x^{\dag }(\sigma )$ by its inverse. The twist 
$\eta ^{\dag }$ may be replaced only by $Int(m^{\ast })\circ \eta \circ
Int(g),$ where $m^{\ast },$ $g$ are as specified in Section 7.4. Then $%
x^{\dag }(\sigma )^{-1}$ is replaced by $\sigma (m^{\ast })(m^{\ast
})^{-1}.m^{\ast }x^{\dag }(\sigma )^{-1}.(m^{\ast })^{-1}.$ Our assumptions
on $m^{\ast }$ ensure that $\sigma (m^{\ast })(m^{\ast })^{-1}$ is a cocycle
in $T_{sc};$ its class then has image in $\mathcal{E}_{M^{\ast }}(T).$
Finally, the $\mathbb{R}$-automorphism $Int(m^{\ast }):T_{sc}\rightarrow
T_{sc}$ induces a homomorphism $H^{1}(\Gamma ,T_{sc})\rightarrow
H^{1}(\Gamma ,T_{sc})$. Passing to $T,$ we may then define a homomorphism $%
\mathcal{E}(T)\rightarrow \mathcal{E}(T)\diagup \mathcal{E}_{M^{\ast }}(T)$.
From the Tate-Nakayama isomorphism of $H^{1}(\Gamma ,T_{sc})$ with $%
H^{-1}(\Gamma ,X_{\ast }(T_{sc}))$, we see that the homomorphism coincides
with the natural projection, and the lemma follows.
\end{proof}

Now define%
\begin{equation*}
inv(\pi ):=\mathbf{x}(\eta ^{\dag }).\mathcal{E}_{M^{\ast }}(T).
\end{equation*}

Because $s$ is a $\Gamma $-invariant in the center of $M^{\vee },$ we have
that%
\begin{equation*}
\left\langle \mathcal{E}_{M^{\ast }}(T),\text{ }s_{T}\right\rangle _{tn}=1,
\end{equation*}%
and so the Tate-Nakayama pairing for $T$ determines a well-defined sign we
will write as%
\begin{equation*}
\left\langle inv(\pi ),\text{ }s_{T}\right\rangle .
\end{equation*}%
We may view $\left\langle \_,\_\right\rangle $ as a pairing between $%
\mathcal{E}(T)\diagup \mathcal{E}_{M^{\ast }}(T)$ and $(Z_{M^{\vee
}})^{\Gamma }$ or, better, between $\mathcal{E}(T)\diagup \mathcal{E}%
_{M^{\ast }}(T)$ and $(Z_{M^{\vee }})^{\Gamma }\diagup (Z_{G^{\vee
}})^{\Gamma }$. In the latter case we identify $s_{T}$ with its image in $%
(Z_{M^{\vee }})^{\Gamma }\diagup (Z_{G^{\vee }})^{\Gamma }$ without change
in notation. We will say more about the pairing in \cite{ShII}.

Notice that Lemma 8.1 is now proved in this setting, \textit{i.e.}, for an
extended group of quasi-split type, because 
\begin{equation}
\left\langle \mathbf{x}(\eta ^{\dag },\widehat{\pi }),\text{ }%
s_{U}\right\rangle _{tn}=pair_{(s)}(\pi ,\widehat{\pi })=\left\langle
inv(\pi ),s_{T}\right\rangle ^{-1}.\left\langle inv(\widehat{\pi }%
),s_{T}\right\rangle _{tn}.
\end{equation}

\begin{proof}
\lbrack of Lemma 8.1] A factoring via the method for the proof of Lemma 8.2,
but now in $U_{sc}$ instead of $T_{sc},$ may be applied to the cocycle
defining $\mathbf{x}_{sc}(\eta ^{\dag },\widehat{\pi }).$ Then we follow
closely the rest of the argument to complete the proof.
\end{proof}

Next, we recall the sign%
\begin{equation*}
\varepsilon (G):=(-1)^{q(G)\text{ }-\text{ }q(G^{\ast })},
\end{equation*}%
where $2q(G)$ is the rank of the symmetric space attached to $G_{sc}$. It is
well-defined in general and appears in the tempered character identities for
transfer from the inner form $(G,\eta )$ to $G^{\ast }$; see \cite[Theorem
6.3]{Sh79a}. This sign is recast by Kottwitz in \cite[p.295]{Ko83} in terms
of Galois cohomology. Notice that the choice of inner twist does not matter;
see \cite[p.292]{Ko83}. In our present setting we have $\pi =\pi (\eta
^{\dag }).$ Let $M_{\eta ^{\dag }}=(\eta ^{\dag })^{-1}(M^{\ast }).$ Then it
is clear from either definition that $\varepsilon (M_{\eta ^{\dag }})$ is
independent of the various choices for $\eta ^{\dag }$ and so we write it as 
$\varepsilon _{M}(\pi )$.

We conclude then that the relative factor%
\begin{equation}
\Delta (\pi _{s},\pi ;\widehat{\pi }_{s},\widehat{\pi }):=\varepsilon
_{M}(\pi ).pair_{(s)}(\pi ,\widehat{\pi })
\end{equation}%
is well-defined, \textit{i.e.}, depends only on $s,\pi \ $and $\widehat{\pi }%
.$ This factor and others similarly defined have useful transitivity
properties (see \cite[Section 4.1]{LS87}, \cite[Section 4]{Sh10}). We will
ignore them for now except to remark that if the discrete series
representation $\widehat{\pi }$ has the property that $\eta _{\widehat{\pi }%
} $ serves as $\eta ^{\dag }$ then%
\begin{equation}
\Delta (\pi _{s},\pi ;\widehat{\pi }_{s},\widehat{\pi })=\varepsilon
_{M}(\pi ).
\end{equation}

To define an absolute factor $\Delta (\pi _{s},\pi )$, assume that we have
absolute geometric factors and absolute spectral factors for the essentially
tempered spectrum that are compatible in the sense of \cite[Section 12]{Sh10}%
. This notion of compatibility is defined via another canonical relative
factor, and compatible factors are easily shown to exist for all inner forms 
$(G,\eta )$; see \cite[Section 4]{Sh10}. We then set%
\begin{equation}
\Delta (\pi _{s},\pi ):=\Delta (\pi _{s},\pi ;\widehat{\pi }_{s},\widehat{%
\pi }).\Delta (\widehat{\pi }_{s},\widehat{\pi }).
\end{equation}%
In particular if $M^{\ast }$ is a torus, so that $(\pi _{s},\pi )$ is a
related pair of discrete series representations, we return the original
constructions for the (essentially) tempered spectrum; see \cite[Section 9]%
{Sh10}.

Consider an extended group of quasi-split type and use the Whittaker
normalization $\Delta _{Wh}$ of absolute factors attached to our choice of
Whittaker data \cite[Section 5.3]{KS99}. Then (8.5), (8.3), (8.2) and the
strong base-point property of Whittaker normalization \cite[Theorem 11.5]%
{Sh08b} (recall Section 6.5) imply:

\begin{lemma}
\begin{equation*}
\Delta _{Wh}(\pi _{s},\pi )=\varepsilon _{M}(\pi ).\left\langle inv(\pi ),%
\text{ }s_{T}\right\rangle .
\end{equation*}
\end{lemma}

\subsection{Application to the transfer of Adams-Johnson}

Continuing in the same setting, we write the correspondence of test
functions (more precisely, test measures) as $(f,f^{(s)}).$ Then%
\begin{equation}
SO(\gamma ,f^{(s)})=\sum_{\delta ,\text{ }conj}\Delta (\gamma ,\delta )\text{
}O(\delta ,f)
\end{equation}%
for all strongly $G$-regular $\gamma $ in $H^{(s)}(\mathbb{R})$ and%
\begin{equation}
St\text{-}Trace\text{ }\widehat{\pi }_{s}(f^{(s)})=\tsum\nolimits_{\widehat{%
\pi }}\Delta (\widehat{\pi }_{s},\widehat{\pi })\text{ }Trace\text{ }%
\widehat{\pi }(f).
\end{equation}

Now to consider the pair $(\pi _{s},\pi ),$ we observe that the
Adams-Johnson stable combination \cite[Theorem 2.13]{AJ87} agrees with

\begin{equation*}
St-Trace\pi _{s}(f^{(s)}):=\tsum\nolimits_{\pi _{s}^{\prime }\in \Pi
_{H^{(s)}}}\varepsilon _{M}(\pi _{s}^{\prime })Trace\pi _{s}^{\prime
}(f^{(s)}),
\end{equation*}%
up to the sign $(-1)^{\gamma (M^{\ast })}$ defined in \cite[2.12]{AJ87}.

Next we claim the following transfer for $(\pi _{s},\pi ):$

\begin{equation}
St\text{-}Trace\text{ }\pi _{s}(f^{(s)})=\tsum\nolimits_{\pi \in \Pi
_{G}}\Delta (\pi _{s},\pi )\text{ }Trace\text{ }\pi (f).
\end{equation}%
Here $(f,f^{(s)})$ is any pair of test functions related by the geometric
transfer (8.6) and $\Delta (\pi _{s},\pi )$ is given by (8.3), (8.5) (or by
(8.9) below).

Suppose $G$ has anisotropic center, so that we may apply the main transfer
theorem of Adams-Johnson directly. We recast the geometric transfer of \cite[%
Section 2]{AJ87} as the correspondence $(f,f^{(s)})$ above; see \cite[%
Theorem 2.6.A]{LS90}. Also, because we must work with $C_{c}^{\infty }$%
-functions, we have applied Bouaziz's Theorem as in \cite[Sections 1, 2]%
{Sh12}. From \cite[Theorem 2.21]{AJ87} we then have that the transfer (8.8)
is true for some choice of the coefficients, say $\Delta ^{\prime }(\pi
_{s},\pi ).$ With a little more effort we may show that our choice of $%
\Delta (\pi _{s},\pi )$ is correct up to a constant, but we will not need
that. Instead, we turn to the transfer (8.7) in the case of the discrete
series pairs $(\widehat{\pi }_{s},\widehat{\pi })$ from Sections 7.3 and 8.2.

For each pair $(\pi _{s},\pi ),$ where $\pi =\pi (\eta ^{\dag }),$ we
consider all pairs $(\widehat{\pi }_{s},\widehat{\pi })$ such that $\eta _{%
\widehat{\pi }}$ serves as $\eta ^{\dag }$. From (8.4) and (8.5) we have that%
\begin{equation}
\Delta (\pi _{s},\pi )=\varepsilon _{M}(\pi ).\Delta (\widehat{\pi }_{s},%
\widehat{\pi }).
\end{equation}%
Now we choose $(f,f^{(s)})$ with support within the very regular elliptic
set (see Section 3.4). We follow the comparison in \cite[Section 9]{Ko90} of
the Vogan-Zuckerman character formula for $\pi $ on the regular elliptic set
with the Harish-Chandra formulas for the discrete series characters $%
\widehat{\pi }$ attached to $\pi $. From this we deduce that%
\begin{equation}
Trace\text{ }\pi (f)=(-1)^{q(M_{\eta ^{\dag }})}\tsum\nolimits_{\widehat{\pi 
}}Trace\text{ }\widehat{\pi }(f)
\end{equation}%
for our particular pairs $(f,f^{(s)}).$ Multiply across (8.7) by $%
(-1)^{q(M^{\ast })}.$ From that identity, together with (8.9) and (8.10), we
then have that%
\begin{equation*}
\tsum\nolimits_{\pi \in \Pi _{G}}[\Delta (\pi _{s},\pi )-\Delta ^{\prime
}(\pi _{s},\pi )]Trace\pi (f)=0
\end{equation*}%
for all $f$ supported in the strongly regular elliptic set. It now follows
that the coefficients $\Delta (\pi _{s},\pi )-\Delta ^{\prime }(\pi _{s},\pi
)$ are all zero; we could also argue this directly with the transfer of
characters as functions. We conclude then that our choice of the constants $%
\Delta (\pi _{s},\pi )$ in (8.8) is correct.

\section{\textbf{Parameters and twistpackets}}

We now return to the general twisted setting of Section 3.1 and finish the
proof of various assertions made earlier.

\subsection{Twistpackets}

Attached to the triple $(G^{\ast },\theta ^{\ast },a)$ is the automorphism $%
^{L}\theta _{a}$ of $^{L}G$. We are interested in Langlands parameters $%
\boldsymbol{\varphi }$ preserved by $^{L}\theta _{a},$ \textit{i.e.,} those $%
\boldsymbol{\varphi }$ for which%
\begin{equation*}
S_{\varphi }^{tw}:=\{s\in G^{\vee }:^{L}\theta _{a}\circ \varphi
=Int(s)\circ \varphi \}
\end{equation*}%
is nonempty, for some, and hence any, representative $\varphi $. Then we may
construct supplemented endoscopic data for $(G^{\ast },\theta ^{\ast },a)$
following the last paragraphs of \cite[Chapter 2]{KS99}; see Section 3.1.

Let $(G,\theta ,\eta )$ be an inner form of $(G^{\ast },\theta ^{\ast })$.
It follows quickly from the Langlands classification, at least in the
essentially tempered case, that the $L$-packet $\Pi $ for $G(\mathbb{R})$
attached to $\varphi $ is stable under the operation $\pi \rightarrow \varpi
^{-1}\otimes (\pi \circ \theta ).$ As in Section 4.1, we then say $\Pi $ is $%
(\theta ,\varpi )$-stable. Conversely the parameter for a $(\theta ,\varpi )$%
-stable packet is preserved by $^{L}\theta _{a}$. In general, this operation
on a $(\theta ,\varpi )$-stable packet $\Pi $ need have no fixed points, 
\textit{i.e.,} the twistpacket $\Pi ^{\theta ,\varpi }$ introduced in
Section 4.1 may be empty. We examine this further for $\varphi $ elliptic.

Suppose $\varphi $ is elliptic and preserved by $^{L}\theta _{a}.$ We use
the standard representative $\varphi =\varphi (\mu ,\lambda )$ from (5.6)
and define data $(\mu _{a},\lambda _{a})$ for the cocycle $a$ in the usual
manner: $a(z)=z^{\mu _{a}}\overline{z}^{\sigma \mu _{a}}$ for $z\in \mathbb{C%
}^{\times },$ and $a(w_{\sigma })=e^{2\pi i\lambda _{a}}.$ First we observe
that because $\varphi $ is regular, each element $s$ of $S_{\varphi }^{tw}$
must normalize $\mathcal{T}$. Then because $\theta ^{\vee }$ preserves $%
spl^{\vee },$ $s$ lies in $\mathcal{T}$, so that $S_{\varphi }^{tw}\subseteq 
\mathcal{T}.$ We conclude then that%
\begin{equation}
\theta ^{\vee }\mu =\mu +\mu _{a}\text{ }and\text{ }\theta ^{\vee }\lambda
\equiv \lambda +\lambda _{a}\func{mod}\mathcal{K}_{f},
\end{equation}%
where%
\begin{equation*}
\mathcal{K}_{f}\text{ }\mathcal{=}\text{ }X_{\ast }(\mathcal{T})+[1-\varphi
(w_{\sigma })]X_{\ast }(\mathcal{T})\otimes \mathbb{C}.
\end{equation*}

Returning to Section 6.1, we now assume the chosen Whittaker data for $%
G^{\ast }$ is $\theta ^{\ast }$-stable (see \cite[Section 5.3]{KS99}). We
have a uniquely defined transport of $(\mu ,\lambda ,\mathcal{C})\ $to
character data for the generic discrete series representation $\pi ^{\ast }$
attached to $\varphi .$ Then (9.1) implies that $\pi ^{\ast }\circ \theta
^{\ast }\approx \varpi \otimes \pi ^{\ast },$ or we could argue this
directly from Whittaker properties.

\subsection{Nonempty fundamental twistpackets}

Since the general fundamental case requires only a trivial modification, we
continue with the elliptic setting of Section 9.1. We have attached a
fundamental splitting $spl_{\pi }$ to a discrete series representation $\pi $
of $G(\mathbb{R})$ in Section 6.1. It is unique up to $G(\mathbb{R})$%
-conjugacy.

\begin{lemma}
\textit{Suppose that }$\pi $\textit{\ is a discrete series representation of}
$G(\mathbb{R})$ \textit{such that} $\pi \circ \theta \approx \varpi \otimes
\pi .$ \textit{Then there exists} $\delta _{\pi }\in G(\mathbb{R})$ \textit{%
such that }$Int(\delta _{\pi })\circ \theta $ \textit{preserves} $spl_{\pi
}. $ \textit{If} $spl_{\pi }$ \textit{is replaced by another fundamental
splitting} $Int(x).spl_{\pi },$ \textit{where} $x\in G(\mathbb{R}),$ \textit{%
then} $\delta _{\pi }$ \textit{is replaced by an element }$\delta _{\pi
}^{\prime }$ \textit{of the form} $zx\delta _{\pi }\theta (x)^{-1},$ \textit{%
where} $z\in Z_{G}(\mathbb{R})$.
\end{lemma}

\begin{proof}
Since $\theta $ transports $spl_{\pi }$ to a fundamental splitting for $\pi
\circ \theta $ and we may use $spl_{\pi }$ as splitting for $\varpi \otimes
\pi ,$ the existence of $\delta _{\pi }$ is clear. Now, with $spl_{\pi }$
fixed, $\delta _{\pi }$ may be replaced only by an element of $Z_{G}(\mathbb{%
R})\delta _{\pi }.$ Next replace $spl_{\pi }$ by $Int(x).spl_{\pi },$ where $%
x\in G(\mathbb{R}).$ Then%
\begin{equation*}
Int(x)\circ (Int(\delta _{\pi })\circ \theta )\circ Int(x)^{-1}=Int(x\delta
_{\pi }\theta (x)^{-1})\circ \theta
\end{equation*}%
preserves $Int(x).spl_{\pi },$ and the lemma follows.
\end{proof}

\begin{lemma}
If there exist nonempty twistpackets of discrete series (or fundamental
series) representations of $G(\mathbb{R})$ then there is $(\theta _{f},\eta
_{f})$ in the inner class of $(\theta ,\eta )$ such that $\theta _{f}$
preserves a fundamental splitting for $G$. The converse is also true in the
case that there exists an elliptic (fundamental) Langlands parameter
preserved by $^{L}\theta _{a}$.
\end{lemma}

\begin{proof}
A nonempty twistpacket provides us with a $\theta $-fundamental element $%
\delta _{\pi }$, and so Lemma 2.5 applies. For the converse, we may assume
that $\theta $ is as in (ii) of Lemma 2.5. We then apply the remarks of
Section 9.1 using the transport of data to $G$ provided by the inner twist $%
\eta .$
\end{proof}

We see then that, as on the geometric side, to capture the elliptic
(fundamental) contribution we may assume that, up to a twist by an element
of $G(\mathbb{R}),$ $\theta $ is the transport of $\theta ^{\ast }$ to $G$
by an inner twist $\eta $ which also carries $spl_{Wh}$ to a fundamental
splitting for $G$, \textit{i.e.}, that we are in the setting I of Section
3.2. We will need further information about the element $\delta _{\pi }$ of
Lemma 9.1.

\begin{lemma}
\textit{(i)} $\delta _{\pi }\in G(\mathbb{R})$ \textit{has a norm }$\gamma
_{\pi }$\textit{\ in} $H_{1}(\mathbb{R}).$ \textit{(ii) }$\gamma _{\pi }$%
\textit{\ lies in} $Z_{H_{1}}(\mathbb{R})$ \textit{and its image in} $Z_{H}(%
\mathbb{R})$ under the projection $H_{1}\rightarrow H$ \textit{is determined
uniquely by} $\delta _{\pi }.$ \textit{(iii) If} $\delta _{\pi }$ \textit{is
replaced by} $\delta _{\pi }^{\prime }=zx\delta _{\pi }\theta (x)^{-1},$ 
\textit{where} $z\in Z_{G}(\mathbb{R})$ and $x\in G(\mathbb{R}),$ \textit{%
then} $\gamma _{\pi }$ \textit{is replaced by an element} $\gamma _{\pi
}^{\prime }=z_{1}\gamma _{\pi },$ \textit{where }$(z_{1},z)\in C(\mathbb{R}%
). $
\end{lemma}

We will explain what we mean by (i) in the proof. The group $C(\mathbb{R})$
is from Section 5.1 of \cite{KS99}; it was recalled in Section 4.2.

\begin{proof}
We may as well assume that we are in the setting I of Section 3.2 since the
modifications for a further twist by an element of $G(\mathbb{R})$ are
immediate. Suppose $Int(x_{\pi })\circ \eta $ carries $spl_{\pi }$ to $%
spl_{Wh}.$ Then a calculation shows that\textit{\ }$\delta _{\pi }^{\ast
}=x_{\pi }\eta (\delta _{\pi })\theta ^{\ast }(x_{\pi })^{-1}$ has the
property that $Int(\delta _{\pi }^{\ast })\circ \theta ^{\ast }$\textit{\ }%
preserves\textit{\ }$spl_{Wh}.$ Because $\theta ^{\ast }$\textit{\ }also%
\textit{\ }preserves\textit{\ }$spl_{Wh},$ we conclude that $\delta _{\pi
}^{\ast }$ lies in\textit{\ }$Z_{G^{\ast }}.$ Further, we calculate that%
\textit{\ }$\sigma (\delta _{\pi }^{\ast })^{-1}.\delta _{\pi }^{\ast }\in
(1-\theta ^{\ast })T.$ As in (5.1) of \cite{KS99}, we regard the
coinvariants $(Z_{G^{\ast }})_{\theta ^{\ast }}$of $Z_{G^{\ast }}$ as a
subgroup of the coinvariants $T_{\theta ^{\ast }}$ of $T.$ Then under the
projection $N:T\rightarrow T_{\theta ^{\ast }},$ $\delta _{\pi }^{\ast }$
maps into $(Z_{G^{\ast }})_{\theta ^{\ast }}.$ Since $N(\sigma (\delta _{\pi
}^{\ast })^{-1}.\delta _{\pi }^{\ast })=1,$ we have that $N(\delta _{\pi
}^{\ast })\in (Z_{G^{\ast }})_{\theta ^{\ast }}(\mathbb{R}).$ We identify $%
(Z_{G^{\ast }})_{\theta ^{\ast }}(\mathbb{R})$ as a subgroup of $Z_{H}(%
\mathbb{R})$ and then as a subgroup of $T_{H}(\mathbb{R}).$ Let\textit{\ }$%
\gamma _{\pi }$\textit{\ }be an element of $Z_{H_{1}}(\mathbb{R})$ whose
image under $p:H_{1}\rightarrow H$ coincides with the image of $N(\delta
_{\pi }^{\ast })$ in $Z_{H}(\mathbb{R}).$ Then $\gamma _{\pi }$ is a $T_{1}$%
-norm of $\delta _{\pi }$ in the sense of Section 6 of \cite{Sh12}. In
general, $\gamma _{\pi }$ is determined up to stable conjugacy by $\delta
_{\pi }$ \cite{Sh12}. Since $\gamma _{\pi }\in Z_{H_{1}}(\mathbb{R}),$ it is
uniquely determined by $\delta _{\pi }.$ The rest is immediate.
\end{proof}

\subsection{Elliptic related pairs of parameters}

Let $\mathfrak{e}_{z}$ be a supplemented set of endoscopic data for $%
(G^{\ast },\theta ^{\ast },a)$ as in Section 3.1. We may define \textit{%
related pairs }of essentially tempered parameters $(\varphi _{1},\varphi )$
as in Section 2 of \cite{Sh10} for the standard case. The arguments there,
and accompanying definitions, apply word for word apart from the shift in
notation to $\varpi _{1}$ for the character on the central subgroup $Z_{1}(%
\mathbb{R})$ of $H_{1}(\mathbb{R}).$

We return to the cuspidal-elliptic setting of Section 3.4 since the general
fundamental case follows quickly from this. If an elliptic parameter $%
\varphi _{1}$ for the endoscopic group $H_{1}$ satisfies the stronger
requirement of $G$-regularity then there is an elliptic parameter $\varphi $
for $G^{\ast }$ providing us with a related pair $(\varphi _{1},\varphi ).$
Here it is assumed that $\varphi _{1}$ factors, in the sense of \cite[%
Section 2]{Sh10}, through the group $\mathcal{H}$ included in the chosen SED.

We now recall explicit data attached to such pairs $(\varphi _{1},\varphi )$%
. To the $\theta ^{\vee }$-stable $\Gamma $-splitting $spl_{G^{\vee }}=(%
\mathcal{B},\mathcal{T},\{X\})$ of $G^{\vee }$ we attach a $\Gamma $%
-splitting $spl_{G^{\vee }}^{\theta ^{\vee }}$ for the identity component of 
$(G^{\vee })^{\theta ^{\vee }}$ in the standard manner (see, for example,
p.61 of \cite{KS99}). We adjust the endoscopic datum $\mathfrak{e}=(H,%
\mathcal{H},s)$ within its isomorphism class so that $s\in \mathcal{T}$, and
then fix a $\Gamma $-splitting $spl_{H^{\vee }}=(\mathcal{B}_{H^{\vee }},%
\mathcal{T}_{H^{\vee }},\{Y\})$ for $H^{\vee },$ where $\mathcal{B}_{H^{\vee
}}=\mathcal{B}$ $\mathcal{\cap }$ $H^{\vee }$ and $\mathcal{T}_{H^{\vee }}=%
\mathcal{T}$ $\mathcal{\cap }$ $H^{\vee }=(\mathcal{T}^{\theta ^{\vee
}})^{0}.$ Embed $H^{\vee }$ in $H_{1}^{\vee }$ and extend $spl_{H^{\vee }}$
to $spl_{H_{1}^{\vee }}=(\mathcal{B}_{1},\mathcal{T}_{1},\{Y\})$ by taking $%
\mathcal{B}_{1}=Norm(\mathcal{B}_{H^{\vee }},H_{1}^{\vee })$ and $\mathcal{T}%
_{1}=Cent(\mathcal{T}_{H^{\vee }},H_{1}^{\vee }).$ None of these choices
will matter for transfer factors. Nor will the choice of $\chi $-data (this
choice does matter for the construction of geometric $\Delta _{II}$ and $%
\Delta _{III}).$ We will thus define all Langlands data $\mu _{1},\lambda
_{1},$ \textit{etc.} for packets in familiar terms \cite{La89}; this amounts
to the choice of $\chi $-data such that $\chi _{(\alpha ^{\vee })_{res}}=(z/%
\overline{z})^{1/2}$, where $(\alpha ^{\vee })_{res}$ denotes the
restriction to $(\mathcal{T}^{\theta ^{\vee }})^{0}$ of a root $\alpha
^{\vee }$ of $\mathcal{T}$ in $\mathcal{B}$.

We follow the approach of Section 11 of \cite{Sh08a} for standard endoscopy.
To $spl_{H^{\vee }}$ we attach the representative $\varphi _{1}=\varphi (\mu
_{1},\lambda _{1})$ as in Section 5.6.\textit{\ }Now consider an elliptic
parameter $\varphi =\varphi (\mu ,\lambda )$ for $G^{\ast }.$ We alter the
construction slightly. To fix an element of $G_{der}^{\vee }$ $\rtimes $ $W_{%
\mathbb{R}}$ acting on $\mathcal{T}$ $\cap $ $G_{der}^{\vee }$\ as $%
t\rightarrow t^{-1},$ we may use either $n_{G}\times w_{\sigma }$ defined
relative to $spl_{G^{\vee }}$ or $n_{G,\theta }\times w_{\sigma }$ defined
relative to $spl_{G^{\vee }}^{\theta ^{\vee }}.\ $It is more convenient to
choose the latter. Thus $\varphi =\varphi (\mu ,\lambda )$ will mean that%
\begin{equation*}
\varphi (w_{\sigma })=e^{2\pi i\lambda }.n_{G,\theta }\times w_{\sigma }.
\end{equation*}%
We may also drop the dominance requirement on $\mu $. We do require that $%
\mu _{1}$ is $\mathcal{B}_{1}$-dominant. While $G$-regularity of $\varphi
_{1}$ requires that $\mu $ be regular when $\varphi (\mu _{1},\lambda _{1})$
and $\varphi (\mu ,\lambda )$ are related, $\mathcal{B}_{1}$-dominance of $%
\mu _{1}$ does not ensure that $\mu $ is $\mathcal{B}$-dominant. That case,
however, is the only one that will matter to us (in general, an extra sign
is needed in transfer factors, see Sections 7, 9 of \cite{Sh10}). Thus we
call $\varphi _{1}$ \textit{well-positioned relative to} $\varphi $ if $\mu $
is $\mathcal{B}$-dominant, and make that our assumption throughout. Given $%
\varphi $ we can always find such $\varphi _{1}$ and it is unique up to $%
\mathcal{T}$-conjugacy. It is not difficult to check that this notion is
independent of the choices made in its formulation; again see \cite{Sh10}.

Finally, we determine the conditions on $(\mu _{1},\lambda _{1})$ and $(\mu
,\lambda )$ for $\varphi _{1}(\mu _{1},\lambda _{1})$ and $\varphi (\mu
,\lambda )$ to be related. First, for $w\in W_{\mathbb{R}}$ pick $u(w)\in 
\mathcal{H}$ projecting to $w,$ as follows. For $z\in \mathbb{C}^{\times },$ 
$u(z)$ is to act trivially and $u(zw_{\sigma })$ is to act on $\mathcal{T}%
_{H}$ and $\mathcal{T}_{1}$ as $n_{H}\times w_{\sigma }\in $ $^{L}H.$ Since $%
\xi _{1}$ (part of the chosen SED) embeds $\mathcal{H}$ in $^{L}H_{1}$ we
may define%
\begin{equation*}
\xi _{1}(u(zw_{\sigma }))=t_{\xi _{1}}(zw_{\sigma }).n_{H}\times zw_{\sigma }
\end{equation*}%
and%
\begin{equation*}
\xi _{1}(u(z))=t_{\xi _{1}}(z)\times z,
\end{equation*}%
where each $t_{\xi _{1}}(w)$ lies in $\mathcal{T}_{1}.$ On the other hand,
in $^{L}G$ we have that $u(w_{\sigma })$ acts as $n_{G,\theta }\times
zw_{\sigma }.$ Write%
\begin{equation*}
u(w)=t(w).u^{\prime }(w),
\end{equation*}%
where%
\begin{equation*}
u^{\prime }(z)=1\times z,u^{\prime }(zw_{\sigma })=n_{G,\theta }\times
w_{\sigma }
\end{equation*}%
for $z\in \mathbb{C}^{\times }.$ Then $t(w)\in \mathcal{T}$ . Let%
\begin{equation*}
\mathcal{T}_{2}=(\mathcal{T}_{1}\times \mathcal{T)}\diagup \mathcal{T}_{H},
\end{equation*}%
where $\mathcal{T}_{H}$ is embedded by $t\rightarrow (t^{-1},t).$ On $%
\mathcal{T}_{2}$ we use the elliptic action $\sigma _{2}$ of $\Gamma $
inflated to $W_{\mathbb{R}}$: $\sigma _{2}$ acts as $n_{H}\times w_{\sigma }$
on the first component and as $n_{G,\theta }\times w_{\sigma }$ on the
second. Let $t_{2}(w)$ denote the image in $\mathcal{T}_{2}$ of $(t_{\xi
_{1}}(w)^{-1},t(w))\in \mathcal{T}_{1}\times \mathcal{T}.$ Then we define%
\begin{equation*}
(\mu ^{\ast },\lambda ^{\ast })\in (X_{\ast }(\mathcal{T}_{2})\otimes 
\mathbb{C})^{2}
\end{equation*}%
by%
\begin{equation*}
t_{2}(z)=z^{\mu ^{\ast }}.\overline{z}^{\sigma _{2}\mu ^{\ast }}\times z
\end{equation*}%
for $z\in \mathbb{C}^{\times }$, and%
\begin{equation*}
t_{2}(w_{\sigma })=e^{2\pi i\lambda ^{\ast }}\times w_{\sigma }.
\end{equation*}

Notice that we have constructed $(\mu ^{\ast },\lambda ^{\ast })$
independently of $\varphi _{1},$ $\varphi .$ The cochain $t_{2}(w)$ is 
\textit{not} the cocycle $a_{T}(w)$ of (4.4) in \cite{KS99}; $a_{T}(w)$
requires a $\rho $-shift\textit{\ }($\iota $-shift in our notation) to be
applied to the datum $\mu ^{\ast }$. See Section 11 of \cite{Sh08a} for the
case of standard endoscopy, where the torus $\mathcal{T}_{2}$ collapses to $%
\mathcal{T}_{1}$.

Recall that $\varphi _{1}(W_{\mathbb{R}})$ is assumed to lie in $\xi _{1}(%
\mathcal{H)}$. Identify $\mu _{1},\mu $ with their images in $X_{\ast }(%
\mathcal{T}_{2})\otimes \mathbb{C}$ under the componentwise embeddings. We
may write%
\begin{equation*}
\varphi _{1}(w)=t_{H}(w).\xi _{1}(u(w))
\end{equation*}%
and%
\begin{equation*}
\varphi (w)=t_{H}(w).t(w).u^{\prime }(w),
\end{equation*}%
where $t_{H}(w)\in \mathcal{T}_{H}.$ We now conclude that:

\begin{lemma}
An elliptic pair $(\varphi _{1}(\mu _{1},\lambda _{1})$, $\varphi (\mu
,\lambda ))$ as above is related if and only if 
\begin{equation*}
\mu _{1}+\mu ^{\ast }=\mu \text{ }and\text{ }\lambda _{1}+\lambda ^{\ast
}\equiv \lambda \func{mod}\mathcal{K}_{f}.
\end{equation*}
\end{lemma}

\subsection{An application}

We finish with a proof of the formula (4.4) from our discussion on
properties required of spectral factors.

\begin{lemma}
\begin{equation*}
\varpi _{C}((z_{1},z))=\varpi _{\pi _{1}}(z_{1}).\varpi _{\pi }(z)^{-1},
\end{equation*}%
\textit{for all} $(z_{1},z)\in C(\mathbb{R}).$
\end{lemma}

\begin{proof}
There is no harm in arguing in $G^{\ast }\ $since $\varpi _{\pi _{1}}$, $%
\varpi _{\pi }$ may be calculated there. Thus we embed $Z_{G^{\ast }}$ in
fundamental $T$ and write $z\in Z_{G^{\ast }}(\mathbb{R})$ in the form $%
z=\exp Y.\exp i\pi \lambda ^{\vee },$ where $Y$ lies in the Lie algebra $%
\mathfrak{z}_{G^{\ast }}(\mathbb{R})$ viewed as a subspace of $X_{\ast
}(T)\otimes \mathbb{C}$ and $\lambda ^{\vee }\in X_{\ast }(T)$ is $\sigma
_{T}$-invariant. Then it follows easily from the Langlands parametrization
that%
\begin{equation*}
\varpi _{\pi }(z)=e^{<\mu ,Y>}e^{2i\pi <\lambda ,\lambda ^{\vee }>}.
\end{equation*}%
Here we have, as usual, identified $X^{\ast }(T)\otimes \mathbb{C}$ with $%
X_{\ast }(\mathcal{T})\otimes \mathbb{C}$. Similarly, for $z_{1}=\exp
Y_{1}.\exp i\pi \lambda _{1}^{\vee },$ with $Y_{1}\in \mathfrak{z}_{H_{1}}(%
\mathbb{R})\subset X_{\ast }(T_{1})\otimes \mathbb{C}$ and $\lambda
_{1}^{\vee }\in X_{\ast }(T)^{\sigma _{T_{1}}},$ we have%
\begin{equation*}
\varpi _{\pi _{1}}(z_{1})=e^{<\mu _{1},Y_{1}>}e^{2i\pi <\lambda _{1},\lambda
_{1}^{\vee }>}.
\end{equation*}%
Now identify the torus $\mathcal{T}_{2}$ from Section 9.2 as the dual of the
torus $T_{2}.$ Then if $z,z_{1}$ have the same image in $Z_{H}(\mathbb{R}),$ 
\textit{i.e.,} if $(z_{1},z)\in C(\mathbb{R}),$ it follows from our remarks
in Section 9.2 that%
\begin{equation*}
\varpi _{\pi _{1}}(z_{1}).\varpi _{\pi }(z)^{-1}=e^{-\text{ }<\mu ^{\ast
},Y_{2}>}e^{-\text{ }2i\pi <\lambda ^{\ast },\lambda _{2}^{\vee }>},
\end{equation*}%
where $Y_{2}=(Y_{1},Y)$ and $\lambda _{2}^{\vee }=(\lambda _{1}^{\vee
},\lambda ^{\vee }).$ On the other hand, it is clear from the definitions of 
$\varpi _{C}$ and $(\mu ^{\ast },\lambda ^{\ast }),$ and from the relation
of the cochain $t_{2}(w)$ to the cocycle $a_{T}(w)$ of p. 45 of \cite{KS99}
(see Section 9.2), that this last expression is the same as $\varpi
_{C}((z_{1},z)).$
\end{proof}

\bigskip

Mathematics Department, Rutgers University-Newark

shelstad@rutgers.edu

\end{document}